\documentclass[11pt, reqno]{amsart}
\usepackage{amssymb,amscd,amsthm,amsxtra}
\usepackage{latexsym}
\usepackage{blindtext}
\usepackage{chngcntr}

\usepackage{graphicx}  \usepackage{epstopdf}

\usepackage{ dsfont }

\usepackage{color}
\usepackage[usenames,dvipsnames,svgnames,table]{xcolor}
\bibliography{refs}

\usepackage[mathscr]{euscript}
\renewcommand{\mathcal}[1]{{\mathscr#1}}

\makeatletter
\@addtoreset{section}{part}
\makeatother

\newtheorem{theorem}{Theorem}[section]

\newtheorem{lemma}[theorem]{Lemma}
\newtheorem{prop}[theorem]{Proposition}

\theoremstyle{definition}
\newtheorem{definition}[theorem]{Definition}
\theoremstyle{remark}
\newtheorem{remark}[theorem]{Remark}
\numberwithin{equation}{section}

\newcommand{\R}{{\mathbb R}}

\renewcommand{\leq}{\leqslant}

\renewcommand{\geq}{\geqslant}

\newcommand{\ophi}{\overline{\phi}}
\newcommand{\oE}{\overline{E}}
\newcommand{\hphi}{\hat{\phi}}
\newcommand{\ovarphi}{\overline{\varphi}}
\newcommand{\hvarphi}{\hat{\varphi}}
\newcommand{\oophi}{\overline{\overline{\phi}}}
\newcommand{\hhphi}{\hat{\hat{\phi}}}
\newcommand{\oovarphi}{\overline{\overline{\varphi}}}
\newcommand{\hhvarphi}{\hat{\hat{\varphi}}}
\newcommand{\ozeta}{\overline{\zeta}}
\newcommand{\hzeta}{\hat{\zeta}}
\newcommand{\oalpha}{\overline{\alpha}}
\newcommand{\halpha}{\hat{\alpha}}
\newcommand{\oT}{\overline{T}}
\newcommand{\hT}{\hat{T}}
\newcommand{\oh}{\overline{h}}
\newcommand{\hh}{\hat{h}}
\newcommand{\oZ}{\overline{Z}}
\newcommand{\hZ}{\hat{Z}}
\newcommand{\oc}{\overline{c}}
\newcommand{\hc}{\hat{c}}
\newcommand{\ot}{\overline{t}}
\newcommand{\hr}{\hat{r}}
\newcommand{\ovr}{\overline{r}}
\newcommand{\hatt}{\hat{t}}
\newcommand{\ow}{\overline{w}}
\newcommand{\hw}{\hat{w}}
\newcommand{\os}{\overline{s}}
\newcommand{\hs}{\hat{s}}
\newcommand{\od}{\overline{d}}
\newcommand{\hd}{\hat{d}}
\newcommand{\otheta}{\overline{\theta}}
\newcommand{\htheta}{\hat{\theta}}
\newcommand{\ocos}{\overline{\cos}}
\newcommand{\osin}{\overline{\sin}}
\newcommand{\oone}{\overline{1}}
\newcommand{\ozero}{\overline{0}}
\newcommand{\hcos}{\hat{\cos}}
\newcommand{\hsin}{\hat{\sin}}
\newcommand{\hone}{\hat{1}}
\newcommand{\hzero}{\hat{0}}
\newcommand{\opi}{\overline{\pi}}
\newcommand{\hpi}{\hat{\pi}}

\renewcommand{\epsilon}{\varepsilon }

\newlength{\defbaselineskip}
\newcommand{\setlinespacing}[1]
           {\setlength{\baselineskip}{#1 \defbaselineskip}}

\begin{document}
\title[Nonradial Nodal Solutions  with Maximal Rank]{Desingularization of Clifford Torus   and Nonradial Solutions to the Yamabe
Problem with Maximal Rank }

\author[M. Medina]{Maria Medina}
\address[Maria Medina]{Facultad de Matem\'aticas, Pontificia Universidad Cat\'olica de Chile, Avenida Vicu\~na Mackenna 4860, Santiago, Chile}
\email{mamedinad@mat.puc.cl}

\author[M. Musso]{Monica Musso}
\address[Monica Musso]{Facultad de Matem\'aticas, Pontificia Universidad Cat\'olica de Chile, Avenida Vicu\~na Mackenna 4860, Santiago, Chile}
\email{musso@mat.puc.cl}

\author[J.Wei]{Juncheng Wei}
\address[Juncheng Wei]{Department of Mathematics, University of British Columbia, Vancouver, BC V6T 1Z2}
\email{jcwei@math.ubc.ca}

\thanks{The first author is supported by the grant FONDECYT Postdoctorado, No. 3160077, CONICYT (Chile) and grant MTM2013-40846-P, MINECO (Spain). The second author is supported by
 FONDECYT Grant 1160135 and Millennium Nucleus Center for Analysis of PDE, NC130017. The third author is supported by NSERC of Canada. }

\medskip

\begin{abstract}
Through desingularization of Clifford torus, we prove the existence of a sequence of nondegenerate (in the sense of Duyckaerts-Kenig-Merle (\cite{DKM})) nodal nonradial solutions to the critical Yamabe problem
$$-\Delta u=\frac{n(n-2)}{4}|u|^{\frac{4}{n-2}}u,\qquad u\in\mathcal{D}^{1,2}(\R^n).$$
The case $n=4$ is the first example in the literature of a solution with {\em maximal rank} ${\mathcal N}=2n+1+\frac{n(n-1)}{2}$.

\end{abstract}

\maketitle

%{\small
%\tableofcontents
%}

\part*{Introduction}

\noindent Consider the problem
\begin{equation}\label{prob}
-\Delta u=\gamma |u|^{p-1}u\hbox{ in }\mathbb{R}^n,\qquad \gamma:=\frac{n(n-2)}{4},\qquad u\in\mathcal{D}^{1,2}(\R^n),
\end{equation}
where $n\geq 4$, $p=\frac{n+2}{n-2}$ and $\mathcal{D}^{1,2}(\R^n)$ is the completion of $C_0^\infty(\R^n)$ with the norm $\|\nabla u\|_{L^2(\R^n)}$.
\medskip

When $u>0$, problem (\ref{prob}) arises in  the classical Yamabe problem or extremal equation for  Sobolev inequality. For positive or  sign-changing $u$ Problem (\ref{prob}) corresponds to  the steady state of the energy-critical focusing nonlinear wave equation
\begin{equation}
\label{2m}
\partial_t^2 u-\Delta u- |u|^{\frac{4}{n-2}} u=0, \ (t,x)\in \R \times \R^n.
\end{equation}
These are classical problems that have attracted the attention of many  researchers (\cite{DKM2, DKM3, KM1, KM2, KST}). The study of (\ref{2m})  naturally relies on the complete classification of the set of non-zero finite energy solutions to Problem  (\ref{prob}), which is defined by
\begin{equation}
\label{3m}
\Sigma:= \left\{ Q \in {\mathcal D}^{1,2} (\R^n) \backslash \{0\}: \ -\Delta Q= \frac{n(n-2)}{4} |Q|^{\frac{4}{n-2}} Q\right \}.
\end{equation}

\medskip

 By the classical work of Caffarelli-Gidas-Spruck \cite{CGS} all positive solutions to (\ref{prob}) are given by
\begin{equation}\label{bubble}
U(y)=\left(\frac{2}{1+|y|^2}\right)^{\frac{n-2}{2}},
\end{equation}
and all its translations and dilations
\begin{equation}
\label{bubble1}
U_{\alpha,\bar{y}}:=\alpha^{-\frac{n-2}{2}}U\left(\frac{y-\bar{y}}{\alpha}\right),\qquad \alpha>0,\; \bar{y}\in\mathbb{R}^n.
\end{equation}

 For sign-changing solutions much less is known. A direct application of Pohozaev's identity gives that all sign-changing solutions  to Problem (\ref{prob}) are nonradial. The existence of elements of $\Sigma$  that are nonradial, sign-changing, and with arbitrary large energy was first proved by Ding \cite{D} using  Ljusternik-Schnirelman  category theory. However no other qualitative properties are known for Ding's solutions. Recently more explicit constructions of sign-changing solutions to Problem (\ref{prob}) have been obtained by del Pino-Musso-Pacard-Pistoia \cite{dPMPP, dmpp2}.  In \cite{MW}, the second and the third authors established the rigidity of the solutions constructed in \cite{dPMPP} by showing that they are {\em nondegenerate} in the sense of Duyckaerts-Kenig-Merle (\cite{DKM}, see definitions below).

\medskip

The purpose of this work is to give a positive answer to an open  question formulated in the work of M. Musso and J. Wei (\cite{MW}): whether there exists a solution that, apart from nondegenerate, is {\em maximal}. To properly explain this framework, let us denote by
$$\Sigma:=\left\{Q\in \mathcal{D}^{1,2}(\R^n)\setminus\{0\}:-\Delta Q=\gamma |Q|^{p-1}Q\right\}$$
the set of nontrivial finite energy solutions of \eqref{prob}. It can be seen that the equation in \eqref{prob} is invariant under four transformations: translation, dilation, orthogonal transformation and Kelvin transform. More precisely, if $Q\in\Sigma$, then:
\begin{itemize}
\item[(i)] $Q(y+a)\in \Sigma$ for every $a\in\R^n$;

\item[(ii)] $\lambda^{\frac{n-2}{2}}Q(\lambda y)\in\Sigma$ for every $\lambda>0$;

\item[(iii)] $Q(Py)\in\Sigma$ for every $P\in \mathcal{O}_n$, where $\mathcal{O}_n$ denotes the classical orthogonal group;

\item[(iv)] $|y|^{2-n}Q(|y|^{-2}y)\in\Sigma$.
\end{itemize}
Denote by $\mathcal{M}$ the group of isometries of $\mathcal{D}^{1,2}(\R^n)$ generated by these transformations. Then, $\mathcal{M}$ derives a family of transformations in a neighborhood of the identity (see \cite[Lemma 3.8]{DKM}) of dimension
\begin{equation}\label{defN}
{\mathcal N}:=2n+1+\frac{n(n-1)}{2}.
\end{equation}
In particular, $\mathcal{M}$ generates the vector space
$$\tilde{\mathcal{I}}_Q=\mbox{span}\left\{
\begin{array}{c}
(2-n)y_\alpha Q+|y|^2\partial_{y_\alpha}Q-2y_\alpha y\cdot\nabla Q,\;\;\partial_{y_\alpha}Q,\;\;1\leq\alpha\leq n,\\
(y_\alpha \partial_{y_\beta}-y_\beta\partial_{y_\alpha})Q,\;\;1\leq\alpha<\beta\leq n,\;\;\frac{n-2}{2}Q+y\cdot Q
\end{array}
\right\}.$$
Consider the associated linearized operator around $Q\in\Sigma$, i.e.,
$$L_Q:=-\Delta-\gamma p |Q|^{p-2}Q,$$
and its kernel
$$\mathcal{I}_Q:=\{f\in\mathcal{D}^{1,2}(\R^n):\,L_Qf=0\}.$$
Clearly $\tilde{\mathcal{I}}_Q\subseteq\mathcal{I}_Q$ and, following the work of T. Duyckaerts, C. Kenig and F. Merle (\cite{DKM}), we can define the notion of {\it nondegeneracy}.
\begin{definition}\label{defNondeg}
$Q\in\Sigma$ is said to be {\it nondegenerate} if $\mathcal{I}_Q=\tilde{\mathcal{I}}_Q$.
\end{definition}
\noindent Let $Q$ be  nondegenerate. Its {\em rank} is defined as  the dimension of $\tilde{\mathcal{I}}_Q$, which is at most $\mathcal{N}$. Actually, the positive solutions $Q=W$ can be proved to be nondegenerate as a consequence of the radial symmetry, and $\tilde{\mathcal{I}}_W$, which is
$$\tilde{\mathcal{I}}_W=\left\{\frac{n-2}{2}W+y\cdot\nabla W,\;\;\partial_{y_\alpha}W,\;\;1\leq \alpha\leq n\right\},$$
has rank $n+1$ (\cite{Rey}). In this case, the rank is strictly less than ${\mathcal N}$.

In \cite{MW}, the authors give the first example of nodal nonradial sign-changing solution satisfying the nondegeneracy condition. Indeed, they consider the solution $u_k$ of \eqref{prob} built in \cite[Theorem 1]{dPMPP} given by
$$u_k(y)=U(y)-\sum_{j=1}^k \mu_k^{-\frac{n-2}{2}}U(\mu_k^{-1}(y-\xi_j))+o(1),$$
where
$$\mu_k:=\frac{c_n}{k^2}, \qquad \xi_j:=(e^{\frac{2j\pi i}{k}},0,\ldots),\qquad U(y):=\left(\frac{2}{1+|y|^2}\right)^{\frac{n-2}{2}},$$
and they prove that $\tilde{\mathcal{I}}_{u_k}=\mathcal{I}_{u_k}$, where the dimension of these vector spaces is $3n$, i.e. the rank is $3n$. Also in this case, the rank is strictly less than $\mathcal{N}$.

The purpose of this work is to provide the first example in the literature of a nondegenerate solution $u$ to \eqref{prob} which has the {\em maximal rank} ${\mathcal N}$.
\begin{definition}\label{defMaximal}
A nondegenerate solution $Q\in\Sigma$ is said to be {\it maximal} if
$$\mbox{dim}(\tilde{\mathcal{I}}_Q)=\mbox{dim}(\mathcal{I}_Q)={\mathcal N},$$
where ${\mathcal N}$ was defined in \eqref{defN}.
\end{definition}
\noindent Thus, our main result can be formulated as follows.
\begin{theorem}\label{teounico}
Let $n\geq 4$. Then, there exists a sequence of nodal solutions to \eqref{prob}, with arbitrarily large energy, which are nondegenerate according to Definition \ref{defNondeg}. If $n=4$ these solutions are maximal in the sense of Definition \ref{defMaximal}.
\end{theorem}
To prove this result, we will build a solution in the following way: let $k$ and $h$ be two large positive integers (not necessarily equal), and
\begin{equation}\label{deltaeps}
\mu:=\frac{\delta^{\frac{2}{n-2}}}{k^2},\qquad \lambda:=\frac{\varepsilon^{\frac{2}{n-2}}}{h^2},
\end{equation}
where $\delta$ and $\varepsilon$ are positive parameters so that
$$c_1<\delta<c_1^{-1},\qquad c_2<\varepsilon<c_2^{-1},$$
for some constants $c_1,c_2>0$ which are independent of $k$ and $h$ as they tend to infinity. Consider now the points
\begin{equation}\begin{split}\label{points}
\xi_j&:=\sqrt{1-\mu^2}(e^{\frac{2\pi i(j-1)}{k}},0,\ldots,0)\in \mathbb{R}^2\times \mathbb{R}^{n-2}, j=1,\ldots,k,\\
\eta_l&:=\sqrt{1-\lambda^2}(0,0,e^{\frac{2\pi i(l-1)}{h}},0,\ldots,0)\in \mathbb{R}^2\times\mathbb{R}^2\times \mathbb{R}^{n-4}, l=1,\ldots,h,
\end{split}\end{equation}
which satisfy
\begin{equation}\label{mod1}
|\xi_j|^2+\mu^2=1,\qquad |\eta_l|^2+\lambda^2=1.
\end{equation}
Consider
\begin{equation}\label{nodalSol}
u(y)=U(y)-\sum_{j=1}^k U_{\mu,\xi_j}(y)-\sum_{l=1}^h U_{\lambda,\eta_l}(y)+\phi(y)
\end{equation}
where $U$ is defined in \eqref{bubble},
\begin{equation}\label{transBubbles}
U_{\mu,\xi_j}(y):=\mu^{-\frac{n-2}{2}}U\left(\frac{y-\xi_j}{\mu}\right),\qquad U_{\lambda,\eta_l}(y):=\lambda^{-\frac{n-2}{2}}U\left(\frac{y-\eta_l}{\lambda}\right),
\end{equation}
and $\phi$ is a small function when compared with the other terms (for the sake of simplicity we do not make explicit the dependence of $u$ in $k$ and $h$).

Notice that functions $U$, $U_{\mu,\xi_j}$ and $U_{\lambda,\eta_l}$ are invariant under rotation of angle $\frac{2\pi}{k}$ in the $(y_1,y_2)$ plane and of angle $\frac{2\pi}{h}$ in the $(y_3,y_4)$ angle. Furthermore, they are even in the $y_\alpha$-coordinates, for $\alpha=2,4,5,\ldots,n$ and invariant under Kelvin's transform (due to \eqref{mod1}). Assume that  $\phi$ also satisfies these properties (we will prove this in Part \ref{Existence}).

Consider the following set of functions:
\begin{equation}\label{z0}
z_0(y):=\frac{n-2}{2}u(y)+\nabla u(y)\cdot y,
\end{equation}
\begin{equation}\label{z1}
z_\alpha(y):=\frac{\partial}{\partial y_\alpha}u(y),\qquad \alpha=1,\ldots,n,
\end{equation}
\begin{equation}\begin{split}\label{z2}
z_{n+1}(y):=&-y_2\frac{\partial}{\partial y_1}u(y)+y_1\frac{\partial}{\partial y_2}u(y),\\
z_{n+2}(y):=&-y_4\frac{\partial}{\partial y_3}u(y)+y_3\frac{\partial}{\partial y_4}u(y),
\end{split}\end{equation}
\begin{equation}\begin{split}\label{z3}
z_{n+\alpha+2}(y)&:=-2y_\alpha z_0(y)+|y|^2z_\alpha(y), \quad \alpha = 1, 2, 3, 4,
\end{split}\end{equation}
\begin{equation}\begin{split}\label{z4}
z_{n+\alpha+4}(y)&:=-y_\alpha z_1(y)+y_1z_\alpha(y), \qquad \alpha=3,\ldots,n,\\
z_{2n+\alpha+2}(y)&:=-y_\alpha z_2(y)+y_2z_\alpha(y), \qquad \alpha=3,\ldots,n,
\end{split}\end{equation}
and
\begin{equation}\begin{split}\label{z5}
z_{3n+\alpha-2}(y)&:=-y_\alpha z_3(y)+y_3z_\alpha(y), \qquad \alpha=5,\ldots,n,\\
z_{4n+\alpha-6}(y)&:=-y_\alpha z_4(y)+y_4z_\alpha(y), \qquad \alpha=5,\ldots,n.
\end{split}\end{equation}
Functions \eqref{z0} and \eqref{z1} are related to the invariance of problem \eqref{prob} under dilations and translations respectively, and \eqref{z2} to the rotation in the $(y_1,y_2)$ and $(y_3,y_4)$ planes. Likewise, \eqref{z3} arises from the invariance under Kelvin transform,  and \eqref{z4}, \eqref{z5} from the rotation in the planes $(y_1, y_\alpha)$, $(y_2,y_\alpha)$, for $\alpha=3,\ldots,n$, and $(y_3, y_\alpha)$, $(y_4,y_\alpha)$ for $\alpha=5,\ldots,n$.
If we denote by $L$ the linearized operator around $u$ associated to \eqref{prob}, i.e.,
\begin{equation}\label{lin}
L(\varphi):=\Delta \varphi+p\gamma|u|^{p-2}u\varphi,
\end{equation}
\eqref{z0}-\eqref{z5} provide $N_0:=5(n-1)$ elements of the kernel of $L$.

We will prove that these are indeed all the elements in the kernel, i.e., solution \eqref{nodalSol} is a second example of nodal nondegenerate solution of \eqref{prob}. But what is more remarkable here is that if $n=4$, then $N_0={\mathcal N}$, that is, the solution is maximal in the sense of Definition \ref{defMaximal}, which is the first example of a nondegenerate maximal solution in the literature, and answers the open question formulated in \cite{MW}.

\begin{remark}

When $\mu \not =\lambda, h \not = k$, our solution is different from the ones constructed in \cite{dPMPP, dmpp2}.  In \cite{dmpp2} the symmetric case $\mu=\lambda, h=k$ is considered, which corresponds to the {\em Clifford torus}. In this case the solution has an additional symmetry which reduces the problem to one dimensional. Because of this symmetry the rank of the solutions constructed in \cite{dmpp2} can be shown to be strictly less than ${\mathcal N}$.  Thus our solutions are  {\em new}. Our construction can be considered as a sort of {\em desingularization} of Clifford torus. For geometric application of desingularization of Clifford torus, we refer to the recent papers \cite{BreKa, Ka} and the references therein.

\end{remark}

\begin{remark}
The construction can be extended to higher even dimensions, that is, one can anagolously set bubbles in the $(y_5,y_6)$, $(y_7, y_8)$, $\ldots$, planes, in such a way that the solution is expected to be nondegenerate and the elements corresponding to  the invariances generate a space of dimension exactly ${\mathcal N}$. Therefore, this type of construction presumably provides a sequence of nodal nondegenerate and maximal rank solutions of \eqref{prob} for any even dimension $n\geq 4$. The existence of a maximal solution for odd dimensions is still an open question.
\end{remark}

\begin{remark}
Nondegenerate solutions to (\ref{prob}) play an important role in the analysis of possible singularity formations in energy-critical wave equations. We refer to \cite{DKM, DKM2, DKM3, KM1, KM2, KST} and the references therein.
\end{remark}

\begin{remark}
The existence of sign-changing solutions for critical exponents in other contexts has been studied in \cite{HV, RV1, RV2}.
\end{remark}

\begin{remark}
There exist many works on the uniqueness and nondegeneracy of {\em positive solutions} to semilinear equations, whether or not for classical nonlinear Schrodinger equations \cite{K} or for  nonlinear fractional equations \cite{FL, FLS}. The rank of the positive solutions is at most $n+1$. For sign-changing solutions the nondegeneracy  question is in general quite difficult without knowing the precise behavior of the solution. Our result is the first of the type for sign-changing solutions with maximal rank.

\end{remark}

Along the work we will denote points $y\in\mathbb{R}^n$, $n\geq 4$, as
$$y=(\overline{y},\hat{y}),\;\overline{y}:=(y_1,y_2),\;\hat{y}:=(y_3,y_4), \mbox{ if } n=4,$$
$$y=(\overline{y},\hat{y},y'),\;\overline{y}:=(y_1,y_2),\;\hat{y}:=(y_3,y_4),\;y':=(y_5,\ldots,y_n),\mbox{ if } n\geq 5,$$
and we will work with the norms
\begin{equation}\label{normStarStar}
\|h\|_{**}:=\|(1+|y|)^{n+2-\frac{2n}{q}}h\|_{L^q(\mathbb{R}^n)}, \qquad \|\phi\|_*:=\|(1+|y|^{n-2})\phi\|_{L^\infty(\mathbb{R}^n)},
\end{equation}
where $\frac{n}{2}<q<n$ is a fixed number.
\medskip

Part \ref{Existence} of the paper is devoted to prove that \eqref{nodalSol} solves \eqref{prob}, and Part \ref{nondegeneracyProof} concerns the proof of its nondegeneracy.

\part{Construction of the solution}\label{Existence}

To prove that \eqref{nodalSol} is a solution of \eqref{prob} we use a Lyapunov-Schmidt reduction method, following the ideas of \cite{dPMPP}. We linearize the equation around a first approximation and  take advantage of the invertibility tools available for this setting. Then, performing a careful analysis of the error of the approximation and of the non linear terms we solve the problem by a fixed point argument. Let us point out that the precise scaling of the parameters $\mu$ and $\lambda$ plays a fundamental role here.

Recalling the definitions given in \eqref{points}, \eqref{deltaeps} and \eqref{transBubbles}, the main result of this part can be stated as follows.

\begin{theorem}\label{mainThm}
Let $n\geq 4$, and let $k, h$ be positive integers so that $k=O(h)$. Then, for sufficiently large $k$ and $h$ there is a finite energy solution of the form
$$u_{k,h}(x)=U(x)-\sum_{j=1}^k U_{\mu,\xi_j}(x)-\sum_{l=1}^h U_{\lambda,\eta_l}(x)+o_k(1)+o_h(1),$$
where $o_k(1)$ and $o_h(1)$ denote quantities that tend to zero when $k$ and $h$ tend to infinity respectively.
\end{theorem}

\section{Error of the approximation}\label{error}
\noindent Denote
$$U_*(y):=U(y)-\sum_{j=1}^k U_{\mu,\xi_j}(y)-\sum_{l=1}^h U_{\lambda,\eta_l}(y),$$
and suppose that the solution $u$ we are looking for has the form
$$u=U_*+\phi,$$
where $\phi$ is a small function when compared with $U_*$. Then solving equation \eqref{prob} is equivalent to find $\phi$ such that
\begin{equation}\label{probPhi}
\Delta\phi+p\gamma |U_*|^{p-1}\phi +E+\gamma N(\phi)=0,
\end{equation}
where
\begin{equation*}\begin{split}
E&:=\Delta U_*+\gamma|U_*|^{p-1}U_*,\\
N(\phi)&:=|U_*+\phi|^{p-1}(U_*+\phi)-|U_*|^{p-1}U_*-p|U_*|^{p-1}\phi.
\end{split}\end{equation*}
In this section we try to estimate the error term $E$. In particular,
\begin{equation*}\begin{split}
\gamma^{-1}E=&\left|U-\sum_{j=1}^kU_{\mu,\xi_j}-\sum_{l=1}^h U_{\lambda,\eta_l}\right|^{p-1}\left(U-\sum_{j=1}^kU_{\mu,\xi_j}-\sum_{l=1}^h U_{\lambda,\eta_l}\right)\\
&-U^p-\sum_{j=1}^kU_{\mu,\xi_j}^p-\sum_{l=1}^h U_{\lambda,\eta_l}^p
\end{split}\end{equation*}
We divide the study of the error in three diffentent regions. Roughly speaking, we will estimate first the $\|\cdot\|_{**}$ norm of the error far from the points $\xi_j$ and $\eta_l$, then around $\xi_j$, and finally around $\eta_l$, for any $j=1,\ldots,k$ and $l=1,\ldots,h$. Indeed, let $\oalpha$ and $\halpha$ be positive numbers independent of $k$ and $h$.
\medskip

\noindent {\it Exterior region:} $y\in \{\cap_{j=1}^k\{|y-\xi_j|>\frac{\oalpha}{k}\}\}\cap\{\cap_{l=1}^h\{|y-\eta_l|>\frac{\halpha}{h}\}\}.$

\noindent For $y$ in this region we can estimate
\begin{equation}\begin{split}\label{errExtEst}
|E|\leq&\,C\left[\frac{1}{(1+|y|^2)^2}+\bigg|\sum_{j=1}^k\frac{\mu^{\frac{n-2}{2}}}{|y-\xi_j|^{n-2}}\bigg|^{\frac{4}{n-2}}+\bigg|\sum_{l=1}^h\frac{\lambda^{\frac{n-2}{2}}}{|y-\eta_l|^{n-2}}\bigg|^{\frac{4}{n-2}}\right]\\
&\,\,\cdot\left(\sum_{j=1}^k\frac{\mu^{\frac{n-2}{2}}}{|y-\xi_j|^{n-2}}+\sum_{l=1}^h\frac{\lambda^{\frac{n-2}{2}}}{|y-\eta_l|^{n-2}}\right)\\
\leq & \,\frac{C}{(1+|y|^2)^2}\left[\sum_{j=1}^k\frac{\mu^{\frac{n-2}{2}}}{|y-\xi_j|^{n-2}}+\sum_{l=1}^h\frac{\lambda^{\frac{n-2}{2}}}{|y-\eta_l|^{n-2}}\right],
\end{split}\end{equation}
where in the last inequality we have used that here
$$\sum_{j=1}^k\frac{\mu^{\frac{n-2}{2}}}{|y-\xi_j|^{n-2}}\leq C\left(1-\frac{k-1}{k^{n-2}}\right),\qquad \sum_{l=1}^h\frac{\lambda^{\frac{n-2}{2}}}{|y-\eta_l|^{n-2}} \leq C\left(1-\frac{h-1}{h^{n-2}}\right).$$
Thus
\begin{equation}\begin{split}\label{errExt}
\|(1+|y|)^{n+2-\frac{2n}{q}}E&\|_{L^q(\{\cap_{j=1}^k\{|y-\xi_j|>\frac{\oalpha}{k}\}\}\cap\{\cap_{l=1}^h\{|y-\eta_l|>\frac{\halpha}{h}\}\})}\\
&\leq C\left(\frac{\mu^{\frac{n-2}{2}}k^{n-2}}{k^{\frac{n}{q}-1}}\textcolor{black}{+\frac{\lambda^{\frac{n-2}{2}}h^{n-2}}{h^{\frac{n}{q}-1}}}\right)\\
&\leq C(k^{1-\frac{n}{q}}\textcolor{black}{+h^{1-\frac{n}{q}}}).
\end{split}\end{equation}
\medskip

\noindent {\it Interior regions around $\xi_j$}: $y\in \{|y-\xi_j|<\frac{\oalpha}{k}\}$ for some $j=1,\ldots,k$.

\noindent Let $j$ be fixed. For some $s\in (0,1)$ we have
\begin{equation*}\begin{split}
\gamma^{-1}E=&\,p(U_{\mu,\xi_j}+s(-\sum_{i\neq j} U_{\mu,\xi_i}+U-\sum_{l=1}^h U_{\lambda,\eta_l}))^{p-1}(-\sum_{i\neq j}U_{\mu,\xi_i}+U-\sum_{l=1}^h U_{\lambda,\eta_l})\\
&-U^p-\sum_{i\neq j}U_{\mu,\xi_i}^p-\sum_{l=1}^h U_{\lambda,\eta_l}^p.
\end{split}\end{equation*}
Let us define
$$\oE_j(y):=\mu^{\frac{n+2}{2}}E(\xi_j+\mu y), \qquad |y|<\frac{\oalpha}{\mu k}.$$
Thus,
\begin{equation}\begin{split}\label{ErrIntXi}
\gamma^{-1}\oE_j(y)=&\,p\left(-U(y)+s(-\sum_{i\neq j}U(y-\mu^{-1}(\xi_j-\xi_i))+\mu^{\frac{n-2}{2}}U(\xi_j+\mu y)\right.\\
&\left. -\sum_{l=1}^h\mu^{\frac{n-2}{2}}\lambda^{-\frac{n-2}{2}}U(\lambda^{-1}(\xi_j+\mu y-\eta_l)\right)^{p-1}(-\sum_{i\neq j}U(y-\mu^{-1}(\xi_j-\xi_i))\\
&+\mu^{\frac{n-2}{2}}U(\xi_j+\mu y)-\sum_{l=1}^h\mu^{\frac{n-2}{2}}\lambda^{-\frac{n-2}{2}}U(\lambda^{-1}(\xi_j+\mu y-\eta_l))\\
&-\mu^{\frac{n+2}{2}}U^p(\xi_j+\mu y)-\sum_{i\neq j} U^p(y-\mu^{-1}(\xi_i-\xi_j))\\
&-\sum_{l=1}^h\mu^{\frac{n+2}{2}}\lambda^{-\frac{n+2}{2}}U^p(\lambda^{-1}(\xi_j+\mu y-\eta_l),
\end{split}\end{equation}
and consequently
\begin{equation}\begin{split}\label{errIntEst}
|\overline{E}_j(y)|&\leq C\left[\frac{(k\mu)^{n-2}+(\mu\lambda)^{\frac{n-2}{2}}h}{1+|y|^4}+\mu^{\frac{n+2}{2}}+(\mu\lambda)^{\frac{n+2}{2}}h^{\frac{n+2}{n-2}}\right]\\
&\leq C\left[\frac{\mu^{\frac{n-2}{2}}}{1+|y|^4}+\mu^{\frac{n+2}{2}}\right].
\end{split}\end{equation}
Noticing that $h=O(k)$ we can compute
\begin{equation}\label{errInt1}
\|(1+|y|)^{n+2-\frac{2n}{q}}\overline{E}_j\|_{L^q(|y|<\frac{\oalpha}{\mu k})}\leq C\mu^{\frac{n}{2q}}\leq Ck^{-\frac{n}{q}}.
\end{equation}

\noindent {\it Interior regions around $\eta_l$}: $y\in \{|y-\eta_l|<\frac{\halpha}{h}\}$ for some $l=1,\ldots,h$.

\noindent The estimates in this region follow analogously to the previous case, but interchanging the role of $\mu$, $k$ and $\lambda$, $h$. Thus, considering
$$\hat{E}_l(y):=\lambda^{\frac{n+2}{2}}E(\eta_l+\mu y), \qquad |y|<\frac{\halpha}{\lambda h},$$
we get
\begin{equation}\begin{split}\label{intErrEst2}
|\hat{E}_l(y)|&\leq C\left[\frac{(h\lambda)^{n-2}+(\mu\lambda)^{\frac{n-2}{2}}k}{1+|y|^4}+\lambda^{\frac{n+2}{2}}+(\mu\lambda)^{\frac{n+2}{2}}k^{\frac{n+2}{n-2}}\right]\\
&\leq C\left[ \frac{\lambda^{\frac{n-2}{2}}}{1+|y|^4}+\lambda^{\frac{n+2}{2}}\right]
\end{split}\end{equation}
and therefore
\begin{equation}\label{errInt2}
\|(1+|y|)^{n+2-\frac{2n}{q}}\hat{E}_l\|_{L^q(|y|<\frac{\halpha}{\lambda h})}\leq C\lambda^{\frac{n}{2q}}\leq Ch^{-\frac{n}{q}}.
\end{equation}

\section{Building the solution}\label{sec3}
\noindent Recall from Section \ref{error} that to find a solution to \eqref{prob} we will prove the existence of a function $\phi$ that solves \eqref{probPhi}. We will try to build this function in a special form.

Let $\zeta(s)$ be a smooth function such that $\zeta(s)=1$ for $s>1$ and $\zeta(s)=0$ for $s>2$, and let $\overline{\alpha}, \hat{\alpha}>0$ be fixed numbers independent of $k$ and $h$. Define
\begin{equation*}\begin{split}
\overline{\zeta}_j(y)&:=\begin{cases}
\zeta(k\overline{\alpha}^{-1}|y|^{-2}|y-\xi_j|y|^2|)\hbox{ if }|y|>1,\\
\zeta(k\overline{\alpha}^{-1}|y-\xi_j|)\hbox{ if }|y|\leq 1,
\end{cases}\\
\hat{\zeta}_l(y)&:=\begin{cases}
\zeta(h\hat{\alpha}^{-1}|y|^{-2}|y-\eta_l|y|^2|)\hbox{ if }|y|>1,\\
\zeta(h\hat{\alpha}^{-1}|y-\eta_l|)\hbox{ if }|y|\leq 1.
\end{cases}\end{split}\end{equation*}
A function of the form
\begin{equation}\label{defPhi}
\phi=\sum_{j=1}^k\overline{\phi}_j+\sum_{l=1}^h\hat{\phi}_l+\psi
\end{equation}
is a solution of \eqref{probPhi} if we solve the system
\begin{equation}\label{system1}
\Delta \ophi_j+p\gamma |U_*|^{p-1}\ozeta_j\ophi_j+\ozeta_j\left[p\gamma |U_*|^{p-1}\psi+E+\gamma N(\phi)\right]=0, \; º;j=1,\ldots,k,
\end{equation}
\begin{equation}\label{system2}
\Delta \hphi_l+p\gamma |U_*|^{p-1}\hzeta_l\hphi_l+\hzeta_l\left[p\gamma |U_*|^{p-1}\psi+E+\gamma N(\phi)\right]=0, \; \;l=1,\ldots,h,
\end{equation}
\begin{equation}\label{system3}\begin{split}
\Delta\psi&+p\gamma U^{p-1}\psi+[p\gamma(|U_*|^{p-1}-U^{p-1})(1-\sum_{j=1}^k\ozeta_j-\sum_{l=1}^h\hzeta_l)\\
&+p\gamma U^{p-1}(\sum_{j=1}^k\ozeta_j+\sum_{l=1}^h\hzeta_l)]\psi+ p\gamma |U_*|^{p-1}\sum_{j=1}^k(1-\ozeta_j)\ophi_j\\
&+p\gamma |U_*|^{p-1}\sum_{l=1}^h(1-\hzeta_l)\hphi_l+(1-\sum_{j=1}^k\ozeta_j-\sum_{l=1}^h\hzeta_l)(E+\gamma N(\phi))=0.
\end{split}\end{equation}
We assume in addition the following symmetry properties on $\ophi_j$ and $\hphi_l$,
\begin{equation}\label{cond1}
\ophi_j(\overline{y},\hat{y},y')=\ophi_1(e^{-\frac{2\pi \textcolor{black}{(j-1)}}{k}i} \overline{y},\hat{y},y'),\qquad j=1,\ldots\,\textcolor{black}{k},
\end{equation}
where
\begin{equation}\begin{split}\label{cond2}
&\ophi_1(y_1,\ldots,y_j,\ldots,y_n)=\ophi_1(y_1,\ldots,-y_j,\ldots,y_n),\qquad j=2,4,5,\ldots,n,\\
&\ophi_1(y)=|y|^{2-n}\ophi_1(|y|^{-2}y),\\
&\ophi_1(\overline{y},\hat{y},y')=\ophi_1(\overline{y},e^{\frac{2\pi (l-1)}{h}i}\hat{y},y'),\qquad \l=1,\ldots,h,
\end{split}\end{equation}
and
\begin{equation}\label{cond3}
\hphi_j(\overline{y},\hat{y},y')=\hphi_1(\overline{y},e^{-\frac{2\pi (l-1)}{h}i} \hat{y},y'),\qquad l=1,\ldots\,h,
\end{equation}
where
\begin{equation}\begin{split}\label{cond4}
&\hphi_1(y_1,\ldots,y_j,\ldots,y_n)=\hphi_1(y_1,\ldots,-y_j,\ldots,y_n),\qquad j=2,4,5,\ldots,n,\\
&\hphi_1(y)=|y|^{2-n}\hphi_1(|y|^{-2}y),\\
&\hphi_1(\overline{y},\hat{y},y')=\hphi_1(e^{\frac{2\pi (j-1)}{k}i}\overline{y},\hat{y},y'),\qquad  j=1,\ldots,k.
\end{split}\end{equation}

\noindent Assume in addition that
\begin{equation}\begin{split}\label{cond5}
&\|\overline{\ophi}_1\|_*\leq \rho, \qquad \overline{\ophi}_1(y):=\mu^{\frac{n-2}{2}}\ophi_1(\xi_1+\mu y),\\
&\|\hat{\hphi}_1\|_*\leq \rho, \qquad \hat{\hphi}_1(y):=\lambda^{\frac{n-2}{2}}\hphi_1(\eta_1+\lambda y),
\end{split}\end{equation}
for $\rho>0$ small.

\begin{lemma}
There exist constants $k_0,h_0,C,\rho_0$ such that, for all $k\geq k_0$ and $h\geq h_0$, if $\ophi_j$, $j=1,\ldots,k$, and $\hphi_l$, $l=1,\ldots,h$ satisfy conditions \eqref{cond1}-\eqref{cond5} with $\rho<\rho_0$ then there exists a unique solution $\psi=\Psi(\oophi_1,\hhphi_1)$ to equation \eqref{system3}, that satisfies the symmetries
\begin{equation}\label{condpsi1}
\psi(y_1,\ldots,y_\alpha,\ldots)=\psi(y_1,\ldots,-y_\alpha,\ldots),\;\; \alpha=5,\ldots,n,\end{equation}
\begin{equation}\label{condpsi2}
\psi(\overline{y},\hat{y},y')=\psi(e^{\frac{2\pi (j-1)}{k}i}\overline{y},\hat{y},y'),\;\;  j=1,\ldots,k,
\end{equation}
\begin{equation}\label{condpsi3}
\psi(\overline{y},\hat{y},y')=\psi(\overline{y},e^{\frac{2\pi (l-1)}{h}i}\hat{y},y'),\;\;  l=1,\ldots,h,
\end{equation}
\begin{equation}\label{condpsi4}
\psi(y)=|y|^{2-n}\psi(|y|^{-2}y),
\end{equation}
and such that
\begin{equation}\label{opPsi}
\|\psi\|_*\leq C\left[\|\oophi_1\|_*+\|\hhphi_1\|_*+k^{1-\frac{n}{q}}+h^{1-\frac{n}{q}}\right].
\end{equation}
Moreover, the operator $\Psi$ satisfies
\begin{equation}\label{lipPsi}
\|\Psi(\oophi_1^1,\hhphi_1^1)-\Psi(\oophi_1^2,\hhphi_1^2)\|_*\leq C(\|\oophi_1^1-\oophi_1^2\|_*+\|\hhphi_1^1-\hhphi_1^2\|_*).
\end{equation}
\end{lemma}

\begin{proof}
We write equation \eqref{system3} as
\begin{equation}\label{probPsi}
\Delta\psi+p\gamma U^{p-1}\psi+V(y)\psi+ p\gamma |U_*|^{p-1}(\sum_{j=1}^k(1-\ozeta_j)\ophi_j+\sum_{l=1}^h(1-\hzeta_l)\hphi_l)+M(\psi)=0,
\end{equation}
where
\begin{equation*}\begin{split}
V(y)&:=p\gamma(|U_*|^{p-1}-U^{p-1})(1-\sum_{j=1}^k\ozeta_j-\sum_{l=1}^h\hzeta_l)+p\gamma U^{p-1}(\sum_{j=1}^k\ozeta_j+\sum_{l=1}^h\hzeta_l),\\
&=:V_1(y)+V_2(y),
\end{split}\end{equation*}
and
$$M(\psi):=(1-\sum_{j=1}^k\ozeta_j-\sum_{l=1}^h\hzeta_l)(E+\gamma N(\phi))=0.$$
Consider first the problem
\begin{equation}\label{linh}
\Delta\psi+p\gamma U^{p-1}\psi=h,
\end{equation}
where $h$ is a function satisfying \eqref{condpsi1}-\eqref{condpsi3}, $\|h\|_{**}<+\infty$ and
\begin{equation}\label{propH}
h(y)=|y|^{-n-2}h(|y|^{-2}y).
\end{equation}
Let
\begin{equation}\label{defZ}
Z_\alpha:=\partial_{y_\alpha}U,\;\;\alpha=1,\cdots,n\;\;\hbox{and}\;\;Z_{0}=y\cdot\nabla U+\frac{n-2}{2}U.
\end{equation}
Due to the oddness of $Z_\alpha$ and assumption \eqref{condpsi1} on $h$ it yields
$$\int_{\mathbb{R}^n}Z_\alpha h=0\hbox{ for all }\alpha=5,\cdots,n.$$
The cases $\alpha=0,1,2,3,4$ also vanish proceeding as in the proof of \cite[Lemma 4.1]{dPMPP} as a consequence of \eqref{condpsi2}-\eqref{condpsi3}. Thus we can apply the linear existence result \cite[Lemma 3.1]{dPMPP} to ensure the existence of a unique solution $\psi$ to \eqref{linh} such that
$$\int_{\mathbb{R}^n}U^{p-1}Z_\alpha\psi=0\hbox{ for all }\alpha=0,1,\ldots,n,$$
and $\|\psi\|_*\leq C\|h\|_{**}$. Notice in addition that the functions
\begin{equation*}\begin{split}
\psi_\alpha(y)&:=\psi(\overline{y},\hat{y},\ldots,-y_i,\ldots,y_n),\;\; \alpha=5,\cdots,n,\\
\psi_{12j}(y)&:=\psi(e^{\frac{2\pi (j-1)}{k}i}\overline{y},\hat{y},y'),\;\;  j=1,\ldots,k,\\
\psi_{34l}(y)&:=\psi(\overline{y},e^{\frac{2\pi (l-1)}{h}i}\hat{y},y'),\;\; l=1,\ldots,h,\\
\psi_{n+1}(y)&:=|y|^{2-n}\psi(|y|^{-2}y),
\end{split}\end{equation*}
also satisfy \eqref{linh} and thus, by the uniqueness,
$$\psi=\psi_\alpha=\psi_{12j}=\psi_{34l}=\psi_{n+1}$$
for all $\alpha=5,\ldots,n$, $j=1,\ldots,k$, $l=1,\ldots,h$, i.e., $\psi$ satisfies \eqref{condpsi1}-\eqref{condpsi4}. Therefore, \eqref{linh} has a unique bounded solution $\psi=T(h)$ satisfying symmetries \eqref{condpsi1}-\eqref{condpsi4} and
$$\|\psi\|_*\leq C\|h\|_{**}$$
for a constant depending only on $q$ and $n$.

We will solve \eqref{probPsi} by means of a fixed point argument, writting
$$\psi=-T(V\psi+ p\gamma |U_*|^{p-1}(\sum_{j=1}^k(1-\ozeta_j)\ophi_j+\sum_{l=1}^h(1-\hzeta_l)\hphi_l)+M(\psi))=:\mathcal{M}(\psi),$$
$\psi\in X,$ where $X$ is the space of continuous functions $\psi$ with $\|\psi\|_*<+\infty$ satisfying \eqref{condpsi1}-\eqref{condpsi4}. Thanks to the special form of $U_*$ and to the symmetry assumptions on $\ophi_j$ and $\hphi_l$,
$$V\psi+ p\gamma |U_*|^{p-1}(\sum_{j=1}^k(1-\ozeta_j)\ophi_j+\sum_{l=1}^h(1-\hzeta_l)\hphi_l)+M(\psi)$$
satisfies \eqref{condpsi1}-\eqref{condpsi3} and \eqref{propH} if $\psi\in X$, and $\mathcal{M}$ is well defined. Actually, we claim that $\mathcal{M}$ is a contraction mapping in the $\| \|_*$ norm in a small ball around the origin in $X$. Indeed,
\begin{equation*}\begin{split}
|V_1(y)|&\leq \gamma p(p-1)\bigg|U-s(\sum_{j=1}^k U_{\mu,\xi_j}+\sum_{l=1}^h U_{\lambda,\eta_l})\bigg|^{p-2}\left(\sum_{j=1}^k U_{\mu,\xi_j}+\sum_{l=1}^h U_{\lambda,\eta_l}\right)\\
&\leq CU^{p-2}(y)\left(\sum_{j=1}^k\frac{\mu^{\frac{n-2}{2}}}{|y-\xi_j|^{n-2}}+\sum_{l=1}^h\frac{\lambda^{\frac{n-2}{2}}}{|y-\eta_l|^{n-2}}\right),
\end{split}\end{equation*}
and thus
$$|V_1\psi(y)|\leq C\|\psi\|_*U^{p-1}(y)\left(\sum_{j=1}^k\frac{\mu^{\frac{n-2}{2}}}{|y-\xi_j|^{n-2}}+\sum_{l=1}^h\frac{\lambda^{\frac{n-2}{2}}}{|y-\eta_l|^{n-2}}\right).$$
Proceeding as in \eqref{errExt} we get
\begin{equation}\label{V1}
\|V_1\psi\|_{**}\leq C\|\psi\|_*(k^{1-\frac{n}{q}}+h^{1-\frac{n}{q}}).
\end{equation}
On the other hand,
\begin{equation}\begin{split}\label{V2}
\|V_2\psi\|_{**}&\leq \sum_{j=1}^k\|pU^{p-1}\psi\ozeta_j\|_{**}+\sum_{l=1}^h\|pU^{p-1}\psi\hzeta_l\|_{**}\\
&\leq C\|\psi\|_*(k^{1-\frac{n}{q}}+h^{1-\frac{n}{q}}),
\end{split}\end{equation}
and putting together \eqref{V1} and \eqref{V2} we conclude
\begin{equation}\label{V}
\|V\psi\|_{**}\leq C\|\psi\|_*(k^{1-\frac{n}{q}}+h^{1-\frac{n}{q}}).
\end{equation}
Assume $|y-\xi_j|>\frac{\oalpha}{2k}$ and $|y-\eta_l|>\frac{\halpha}{2h}$ for all $j$ and $l$. We knew that in this region
\begin{equation*}
\|E\|_{**}\leq C(k^{1-\frac{n}{q}}+h^{1-\frac{n}{q}}).
\end{equation*}
Moreover,
$$\bigg|N\left(\sum_{j=1}^k\ophi_j+\sum_{l=1}^h\hphi_l+\psi\right)\bigg|\leq C U^{p-2}\left(\bigg|\sum_{j=1}^k\ophi_j\bigg|^2+\bigg|\sum_{l=1}^h\hphi_l\bigg|^2+|\psi|^2\right).$$
From \eqref{cond5}, we get
$$
|\ophi_j(y)| \leq C \|\oophi_1\|_*\frac{\mu^\frac{n-2}{2}}{\mu^{n-2} + |y-\xi_j|^{n-2}} , \quad
|\hphi_l(y)| \leq C\|\hhphi_1\|_*\frac{\lambda^\frac{n-2}{2}}{\lambda^{n-2} + |y-\eta_l|^{n-2}}.
$$
Moreover,
$$U^{p-2}|\psi|^2\leq U^p\|\psi\|_*^2.$$
Hence, proceeding again as in \eqref{errExt}, we obtain
\begin{equation}\label{M}
\|M(\psi)\|_{**}\leq Ck^{1-\frac{n}{q}}(1+\|\oophi_1\|_*^2)+Ch^{1-\frac{n}{q}}(1+\|\hhphi_1\|_*^2)+C\|\psi\|_*^2.
\end{equation}
Likewise, if $|y-\xi_j|>\frac{\oalpha}{2k}$ and $|y-\eta_l|>\frac{\halpha}{2h}$,
\begin{equation}\label{Ustar}
\||U_*|^{p-1}(\sum_{j=1}^k\ophi_j+\sum_{l=1}^h\hphi_l)\|_{**}\leq Ck^{1-\frac{n}{q}}\|\oophi_1\|_*+Ch^{1-\frac{n}{q}}\|\hhphi_1\|_*.
\end{equation}
Moreover, for $\psi_1,\psi_2$ satisfying $\|\psi_1\|<\rho$, $\|\psi_2\|<\rho$ it follows
$$\|M(\psi_1)-M(\psi_2)\|_{**}\leq C\rho \|\psi_1-\psi_2\|_*.$$
Joining \eqref{V}, \eqref{M} and \eqref{Ustar}, we see that for $\rho$ small enough the operator $\mathcal{M}$ defines a contraction map in the set of functions $\psi\in X$ with
$$\|\psi\|_*\leq C[\|\oophi_1\|_*+k^{1-\frac{n}{q}}+\|\hhphi_1\|_*+h^{1-\frac{n}{q}}],\;\;\|\overline{\ophi}_1\|_*<\rho,\;\;\|\hat{\hphi}_1\|_*<\rho,$$
for $\rho$ small. Therefore, there exists a solution of \eqref{probPsi} satisfying conditions \eqref{condpsi1}-\eqref{opPsi}. The Lipschitz condition \eqref{lipPsi} easily follows.
\end{proof}

%Let us define the vector integrals
%\begin{equation*}\begin{split}
%&I:=\int_{\mathbb{R}^n}h\left[\begin{array}{c}Z_1\\Z_2\end{array}\right]\,dy=c_n\int_{\mathbb{R}^n}\frac{h(y)}{(1+|y|^2)^{n/2}}\overline{y}\,dy,\\
%&J:=\int_{\mathbb{R}^n}h\left[\begin{array}{c}Z_3\\Z_4\end{array}\right]\,dy=c_n\int_{\mathbb{R}^n}\frac{h(y)}{(1+|y|^2)^{n/2}}\hat{y}\,dy.
%\end{split}\end{equation*}
%Performing respectively the change of variables $\overline{y}=e^{\frac{2\pi i}{k}}\overline{z}$ and $\hat{y}=e^{\frac{2\pi i}{h}}\hat{z}$ and using that $k,h\geq 2$ and conditions \eqref{condpsi2} and \eqref{condpsi3} we conclude that $I=J=0$.

Consider the operator $\Psi(\oophi_1,\hhphi_1)$. Equations \eqref{system1} and \eqref{system2} reduce to solve one of each, for example for $\ophi_1$ and $\hphi_1$.

We try to solve first
\begin{equation*}
\Delta \ophi_1+p\gamma |U_*|^{p-1}\ozeta_1\ophi_1+\ozeta_1\left[p\gamma |U_*|^{p-1}\Psi(\oophi_1,\hhphi_1)+E+\gamma N(\phi)\right]=0\hbox{ in }\mathbb{R}^n,
\end{equation*}
or equivalently,
\begin{equation}\label{eq1}
\Delta \ophi_1+p\gamma |U_{\mu,\xi_1}|^{p-1}\ophi_1+\ozeta_1 E+\gamma\overline{\mathcal{N}}(\overline{\ophi}_1,\hat{\hphi}_1)=0,
\end{equation}
where
$$\overline{\mathcal{N}}(\overline{\ophi}_1,\hat{\hphi}_1):=p(|U_*|^{p-1}\ozeta_1-|U_{\mu,\xi_1}|^{p-1})\ophi_1+\ozeta_1\left[p|U_*|^{p-1}\Psi(\oophi_1,\hhphi_1)+N(\phi)\right].$$
Consider first a general function $\overline{h}$ and the problem
\begin{equation}\label{projProb}
\Delta\ophi+p\gamma U_{\mu,\xi_1}^{p-1}\ophi+\overline{h}=\overline{c}_{0}U_{\mu,\xi_1}^{p-1}\overline{Z}_{0}\hbox{ in }\mathbb{R}^n,
\end{equation}
where
$$\overline{Z}_{0}(y):=\mu^{-\frac{n-2}{2}}Z_{0}\left(\frac{y-\xi_1}{\mu}\right),\qquad \overline{c}_{0}:=\frac{\int_{\mathbb{R}^n}\overline{h}\,\overline{Z}_{0}}{\int_{\mathbb{R}^n}U_{\mu,\xi_1}^{p-1}\overline{Z}_{0}^2},$$
with $Z_0$ defined in \eqref{defZ}.
\begin{lemma}\label{proj1}
Suppose that $\overline{h}$ is even with respect to each of the variables $y_2,y_4,y_5,\ldots,y_n$ and such that
\begin{equation}\label{hypRot}
\overline{h}(y)=|y|^{-n-2}\oh(|y|^{-2}y),\qquad \oh(y)=\oh(\overline{y},e^{\frac{2\pi(l-1)}{h}}\hat{y},y'), \;\;l=1,\ldots,h.
\end{equation}
Assume that
$$h(y):=\mu^{\frac{n+2}{2}}\overline{h}(\xi_1+\mu y)$$
satisfies $\|h\|_{**}<+\infty$. Then problem \eqref{projProb} has a unique solution $\ophi:=\overline{T}(\overline{h})$ that is even with respect to the variables $y_2,y_4,y_5,\ldots,y_n$, invariant under Kelvin's tranform, i.e.,
$$\ophi(y)=|y|^{2-n}\ophi(|y|^{-2}y),$$
and with $\oophi(y):=\mu^{\frac{n-2}{2}}\ophi(\xi_1+\mu y)$ satisfying
$$\int_{\mathbb{R}^n}\oophi U^{p-1}Z_{0}=0,\qquad \|\oophi\|_{*}\leq C\|h\|_{**}.$$
\end{lemma}

\begin{proof}
We assume with no loss of generality that
$$\int_{\mathbb{R}^n}\oh\,\overline{Z}_{0}=0,\qquad \hbox{ i.e., }\qquad\overline{c}_{0}=0.$$
Thus, equation \eqref{projProb} is equivalent to
$$\Delta \oophi+p\gamma|U|^{p-1}\oophi=-h \qquad \hbox{in }\mathbb{R}^n.$$
Due to the evenness of $h$ we know that
\begin{equation}\label{vanZ}
\int_{\mathbb{R}^n}h Z_\alpha=0, \qquad \alpha=2,4,5,\cdots,n.
\end{equation}
The proof of
$$\int_{\mathbb{R}^n}h Z_1=0$$
follows exactly as in the proof of \cite[Lemma 4.2]{dPMPP}, so we focus on the case $\alpha=3$.

Indeed, denote by
$w_{\mu}(y):=\mu^{-\frac{n-2}{2}}U(\mu^{-1}y)$, and $J(t):=\int_{\mathbb{R}^n}w_{\mu}(y-\xi_1+te_3)\oh(y)\,dy $.
Notice first that
\begin{equation}\label{der1}
\frac{d}{dt} J(t)\bigg|_{t=0}=\int_{\mathbb{R}^n}\partial_{y_3}w_{\mu}(y-\xi_1)\oh(y)\,dy=\int_{\mathbb{R}^n}hZ_3.
\end{equation}
On the other hand, defining $\tilde{y}:=(\overline{y},e^{\frac{2\pi(l-1)}{h}}\hat{y},y')$ for some $l=2,3,\ldots,h$, it can be checked that
$$|\tilde{y}-\xi_1+te_3|^2=|y-\xi_1+t\tilde{e}|^2,\qquad t\in\R^n,$$
where $\tilde{e}:=(0,0,\cos(\frac{2\pi(l-1)}{h}),-\sin(\frac{2\pi(l-1)}{h}),0,\ldots)$. Thus, after a change of variables, by \eqref{hypRot},
\begin{equation*}\begin{split}
J(t)&=\int_{\mathbb{R}^n}w_{\mu}(\tilde{y}-\xi_1+te_3)\oh(\tilde{y})\,d\tilde{y}
=\int_{\mathbb{R}^n}w_{\mu}(y-\xi_1+t\tilde{e})\oh(y)\,dy.
\end{split}\end{equation*}
Differentiating here,
\begin{equation*}\begin{split}
\frac{d}{dt}J(t)\bigg|_{t=0}&=\cos\left(\frac{2\pi(l-1)}{h}\right)\int_{\mathbb{R}^n}\partial_{y_3}w_{\mu}(y-\xi_1)\oh(y)\,dy\\
&\;\;\;-\sin\left(\frac{2\pi(l-1)}{h}\right)\int_{\mathbb{R}^n}\partial_{y_4}w_{\mu}(y-\xi_1)\oh(y)\,dy\\
&=\cos\left(\frac{2\pi(l-1)}{h}\right)\int_{\mathbb{R}^n}hZ_3\,dy-\sin\left(\frac{2\pi(l-1)}{h}\right)\int_{\mathbb{R}^n}hZ_4\,dy.
\end{split}\end{equation*}
Applying \eqref{vanZ} and \eqref{der1} we conclude that necessarily
$\int_{\mathbb{R}^n}h Z_3=0$.
Thus, by \cite[Lemma 3.1]{dPMPP} there exists a unique solution $\oophi$ satisfying
$$\|\oophi\|_*\leq C\|h\|_{**},\qquad \int_{\mathbb{R}^n}\oophi U^{p-1}Z_{n+1}=0.$$
The invariance under Kelvin transform and the symmetries are obtained as a consequence of the uniqueness.
\end{proof}

\noindent Likewise, we rewrite
\begin{equation*}
\Delta \hphi_1+p\gamma |U_*|^{p-1}\hzeta_1\hphi_1+\hzeta_1\left[p\gamma |U_*|^{p-1}\Psi(\oophi_1,\hhphi_1)+E+\gamma N(\phi)\right]=0\hbox{ in }\mathbb{R}^n,
\end{equation*}
as
\begin{equation}\label{eq2}
\Delta \hphi_1+p\gamma |U_{\lambda,\eta_1}|^{p-1}\hphi_1+\hzeta_1 E+\gamma\hat{\mathcal{N}}(\overline{\ophi}_1,\hat{\hphi}_1)=0,
\end{equation}
with
$$\hat{\mathcal{N}}(\overline{\ophi}_1,\hat{\hphi}_1):=p(|U_*|^{p-1}\hzeta_1-|U_{\lambda,\eta_1}|^{p-1})\hphi_1+\hzeta_1\left[p|U_*|^{p-1}\Psi(\oophi_1,\hhphi_1)+N(\phi)\right],$$
and we consider the problem
\begin{equation}\label{projProb2}
\Delta\hphi+p\gamma U_{\lambda,\eta_1}^{p-1}\hphi+\hat{h}=\hat{c}_{0}U_{\lambda,\eta_1}^{p-1}\hat{Z}_{0}\hbox{ in }\mathbb{R}^n,
\end{equation}
where $\hat{h}$ is a general function and
$$\hat{Z}_{0}(y):=\lambda^{-\frac{n-2}{2}}Z_{0}\left(\frac{y-\eta_1}{\lambda}\right),\qquad \hat{c}_{0}:=\frac{\int_{\mathbb{R}^n}\hat{h}\hat{Z}_{0}}{\int_{\mathbb{R}^n}U_{\lambda,\eta_1}^{p-1}\hat{Z}_{0}^2}.$$
\begin{lemma}\label{proj2}
Suppose that $\hat{h}$ is even with respect to each of the variables $y_2,y_4,y_5,\ldots,y_n$ and such that
$$\hat{h}(y)=|y|^{-n-2}\hh(|y|^{-2}y),\qquad \hat{h}(y)=\hat{h}(e^{\frac{2\pi(j-1)}{k}}\overline{y},\hat{y},y'),\;\;j=1,\ldots,k.$$
Assume that
$$h(y):=\lambda^{\frac{n+2}{2}}\hat{h}(\eta_1+\lambda y)$$
satisfies $\|h\|_{**}<+\infty$. Then problem \eqref{projProb2} has a unique solution $\hphi:=\hat{T}(\hat{h})$ that is even with respect to the variables $y_2,y_4,y_5,\ldots,y_n$, invariant under Kelvin's tranform, i.e.,
$$\hphi(y)=|y|^{2-n}\hphi(|y|^{-2}y),$$
and with $\hhphi(y):=\lambda^{\frac{n-2}{2}}\hphi(\eta_1+\lambda y)$ satisfying
$$\int_{\mathbb{R}^n}\hhphi U^{p-1}Z_{0}=0,\qquad \|\hhphi\|_{*}\leq C\|h\|_{**}.$$
\end{lemma}

The proof of this result is analogous to the one for Lemma \ref{proj1}, interchanging the roles of $\mu$ and $\xi_1$  with $\lambda$ and $\eta_1$, so we skip it.
\medskip

We use these lemmas to solve the projected versions of \eqref{eq1} and \eqref{eq2}, that is,
\begin{equation}\begin{split}\label{projEq}
&\Delta \ophi_1+p\gamma |U_{\mu,\xi_1}|^{p-1}\ophi_1+\ozeta_1 E+\gamma\overline{\mathcal{N}}(\overline{\ophi}_1,\hat{\hphi}_1)=\overline{c}_{0}U_{\mu,\xi_1}^{p-1}\overline{Z}_{0},\\
&\Delta \hphi_1+p\gamma |U_{\lambda,\eta_1}|^{p-1}\hphi_1+\hzeta_1 E+\gamma\hat{\mathcal{N}}(\overline{\ophi}_1,\hat{\hphi}_1)=\hat{c}_{0}U_{\lambda,\eta_1}^{p-1}\hat{Z}_{0},
\end{split}\end{equation}
in $\mathbb{R}^n$, with
$$\overline{c}_{0}:=\frac{\int_{\mathbb{R}^n}(\ozeta_1 E+\gamma\overline{\mathcal{N}}(\overline{\ophi}_1,\hat{\hphi}_1))\overline{Z}_{0}}{\int_{\mathbb{R}^n}U_{\mu,\xi_1}^{p-1}\overline{Z}_{0}^2},\qquad\hat{c}_{0}:=\frac{\int_{\mathbb{R}^n}(\hzeta_1 E+\gamma\hat{\mathcal{N}}(\overline{\ophi}_1,\hat{\hphi}_1))\hat{Z}_{0}}{\int_{\mathbb{R}^n}U_{\lambda,\eta_1}^{p-1}\hat{Z}_{0}^2}.$$
\begin{prop}\label{existProjProb}
There exists a unique solution $\phi_1=(\overline{\ophi}_1,\hat{\hphi}_1)=(\overline{\ophi}_1(\delta),\hat{\hphi}_1(\varepsilon))$ to \eqref{projEq} such that
$$\|\phi_1\|_*:=\|\overline{\ophi}_1\|_*+\|\hat{\hphi}_1\|_*\leq C(k^{-\frac{n}{q}}+h^{-\frac{n}{q}}),$$
and
\begin{equation}\label{estN}
\|\overline{\mathcal{N}}(\overline{\ophi}_1,\hat{\hphi}_1)\|_{**}\leq C k^{-\frac{2n}{q}},\qquad \|\hat{\mathcal{N}}(\overline{\ophi}_1,\hat{\hphi}_1)\|_{**}\leq Ch^{-\frac{2n}{q}}.
\end{equation}
\end{prop}

\begin{proof}
Denote by $\oT$ and $\hT$ the linear operators predicted by Lemma \ref{proj1} and Lemma \ref{proj2} respectively. Thus, solving \eqref{projEq} is equivalent to solve the fixed point problem
$$\phi_1=
\left(\begin{array}{c}
\ophi_1\\
\hphi_1
\end{array}
\right)
=\left(
\begin{array}{c}
\oT(\ozeta_1E+\gamma\overline{\mathcal{N}}(\overline{\ophi}_1,\hat{\hphi}_1))\\
\hT(\hzeta_1E+\gamma\hat{\mathcal{N}}(\overline{\ophi}_1,\hat{\hphi}_1))
\end{array}
\right)
=\left(
\begin{array}{c}
\overline{\mathcal{M}}(\overline{\ophi}_1,\hat{\hphi}_1)\\
\hat{\mathcal{M}}(\overline{\ophi}_1,\hat{\hphi}_1)
\end{array}
\right)
=:\mathcal{M}(\phi_1).$$
Let us focus first on $\overline{\mathcal{M}}(\overline{\ophi}_1,\hat{\hphi}_1):=\oT(\ozeta_1E+\gamma\overline{\mathcal{N}}(\overline{\ophi}_1,\hat{\hphi}_1))$. Recall that
$$\overline{\mathcal{N}}(\overline{\ophi}_1,\hat{\hphi}_1):=p(|U_*|^{p-1}\ozeta_1-|U_{\mu,\xi_1}|^{p-1})\ophi_1+\ozeta_1\left[p|U_*|^{p-1}\Psi(\oophi_1,\hhphi_1)+N(\phi)\right].$$
Denote in general
$$\tilde{f}(y)=\mu^{\frac{n+2}{2}}f(\xi_1+\mu y).$$
Consider
$$f_1(y):=p\ozeta_1(|U_*|^{p-1}-|U_{\mu,\xi_1}|^{p-1})\ophi_1.$$
For $|y|<\frac{\oalpha}{\mu k}$,
\begin{equation*}\begin{split}
|\tilde{f}_1(y)|=\bigg|&p((U(y)+\sum_{j=2}^kU(y+\mu^{-1}(\xi_1-\xi_j))\\
&+\sum_{l=1}^h\mu^{\frac{n-2}{2}}\lambda^{-\frac{n-2}{2}}U(\lambda^{-1}(\xi_1+\mu y-\eta_l))\\
&-\mu^{\frac{n-2}{2}}U(\xi_1+\mu y))^{p-1}-U^{p-1}(y))\overline{\ophi}_1(y)\bigg|.
\end{split}\end{equation*}
Noticing that
\begin{equation}\label{u1}
\sum_{j=2}^kU(y+\mu^{-1}(\xi_1-\xi_j))\leq C\mu^{n-2}k^{n-2}\sum_{j=1}^k\frac{1}{j^{n-2}}\leq C\mu^{\frac{n-2}{2}},
\end{equation}
and
\begin{equation}\begin{split}\label{u2}
\sum_{l=1}^h\mu^{\frac{n-2}{2}}\lambda^{-\frac{n-2}{2}}&U(\lambda^{-1}(\xi_1+\mu y-\eta_l))\\
&\leq C\sum_{l=1}^h\mu^{\frac{n-2}{2}}\lambda^{-\frac{n-2}{2}}\leq C\mu^{\frac{n-2}{2}}\lambda^{-\frac{n-2}{2}}h\leq C\mu^{\frac{n-2}{2}},
\end{split}\end{equation}
doing a Taylor expansion we get
\begin{equation}\label{tilf1}
|\tilde{f}_1(y)|\leq C \mu^{\frac{n-2}{2}}U^{p-2}(y)|\overline{\ophi}_1(y)|\leq C\mu^{\frac{n-2}{2}}U^{p-1}(y)\|\overline{\ophi}_1\|_*
\end{equation}
and, proceeding as in the computations for the interior error,
\begin{equation*}
\|\tilde{f}_1\|_{**}\leq C\mu^{\frac{n}{2q}}\|\overline{\ophi}_1\|_*.
\end{equation*}
For the term
$$f_2:=(\ozeta_1-1)U_{\mu,\xi_1}^{p-1}\ophi_1,$$
we have
\begin{equation}\label{tilf2}
|\tilde{f}_2(y)|\leq U^p(y)\|\overline{\ophi}_1\|_*,\qquad |y|>\frac{\oalpha}{\mu k},\qquad \|\tilde{f}_2\|_{**}\leq C\mu^{\frac{n}{2q}}\|\overline{\ophi}_1\|_*.
\end{equation}
Consider now
$$f_3:=\ozeta_1p|U_*|^{p-1}\Psi(\ophi_1,\hphi_1).$$
Using \eqref{u1}, \eqref{u2} and \eqref{opPsi} we get that, for $|y|\leq \frac{\oalpha}{\mu k}$,
\begin{equation}\begin{split}\label{tilf3}
|\tilde{f}_3(y)|&\leq CU^{p-1}\mu^{\frac{n-2}{2}}\|\Psi(\ophi_1,\hphi_1)\|_{L^\infty(\R^n)}\\
&\leq CU^{p-1}\mu^{\frac{n-2}{2}}(\|\oophi_1\|_*+\|\hhphi_1\|_*+k^{1-\frac{n}{q}}+h^{1-\frac{n}{q}}),
\end{split}\end{equation}
\begin{equation*}
\|\tilde{f}_3\|_{**}\leq C\mu^{\frac{n}{2q}}(\|\oophi_1\|_*+\|\hhphi\|_*+k^{1-\frac{n}{q}}+h^{1-\frac{n}{q}}).
\end{equation*}
Denote
$$f_4:=\ozeta_1N(\phi),\qquad f_5:=\ozeta_1 E.$$
Notice that
$$\tilde{N}(\phi)=|V_*+\tilde{\phi}_1|^{p-1}(V_*+\tilde{\phi}_1)-|V_*|^{p-1}V_*-p|V_*|^{p-1}\tilde{\phi}_1,$$
where $\tilde{\phi}_1(y):=\mu^{\frac{n-2}{2}}\phi(\xi_1+\mu y)$, and
\begin{equation*}\begin{split}
V_*(y):=&-U(y)-\sum_{j=2}^kU(y+\mu^{-1}(\xi_1-\xi_j))\\
&-\sum_{l=1}^h\mu^{\frac{n-2}{2}}\lambda^{-\frac{n-2}{2}}U(\lambda^{-1}(\xi_1+\mu y-\eta_l))\\
&+\mu^{\frac{n-2}{2}}U(\xi_1+\mu y).
\end{split}\end{equation*}
Hence, for
$$\phi=\ophi_1+\sum_{j=2}^k\ophi_j+\sum_{l=1}^h\hphi_l+\Psi(\oophi_1,\hhphi_1),$$
one has
$$|\tilde{\phi}_1|\leq C\mu^{\frac{n-2}{2}}(\|\oophi_1\|_*+\|\hhphi_1\|_*)+\mu^{\frac{n-2}{2}}\|\Psi(\oophi_1,\hhphi_1)\|_{L^\infty(\R^n)}.$$
Furthermore, in the region $|y|<\frac{\oalpha}{\mu k}$ it holds $U(y)\sim \mu^{\frac{n-2}{2}}$ and thus, after a second order Taylor expansion one has
\begin{equation}\begin{split}\label{tilf4}
|\tilde{f}_4(y)|&\leq C |V_*|^{p-2}|\tilde{\phi}_1|^2\leq C U^{p-2} \mu^{\frac{n-2}{2}}|\tilde{\phi}_1|\\
&\leq CU^{p-1}\mu^{\frac{n-2}{2}}(\|\oophi_1\|_*+\|\hhphi_1\|_*+\|\Psi\|_*),
\end{split}\end{equation}
and
\begin{equation*}\|\tilde{f}_4\|_{**}\leq C\mu^{\frac{n}{2q}}(\|\oophi_1\|_*+\|\hhphi_1\|_*+k^{1-\frac{n}{q}}+h^{1-\frac{n}{q}}).
\end{equation*}
Finally, by \eqref{errInt1}, we know
\begin{equation}\label{tilf5}
\|\tilde{f}_5\|_{**}\leq C\mu^{\frac{n}{2q}}.
\end{equation}
Likewise, one can obtain analogous estimates for
$$\hT(\hzeta_1E+\gamma\hat{\mathcal{N}}(\oophi_1,\hhphi_1))$$
to conclude that $\mathcal{M}$ maps functions $\phi_1$ with $\|\phi_1\|_*\leq C(\mu^{\frac{n}{2q}}+\lambda^{\frac{n}{2q}})$ into the same class of functions. Besides, one can prove that the map is indeed a contraction, and thus we conclude the existence of a unique solution to the system \eqref{projEq}.
\end{proof}

\begin{remark}
The symmetry conditions \eqref{cond2} and \eqref{cond4} follow straightforward as consequence of the uniqueness.
\end{remark}

\section{Proof of Theorem \ref{mainThm}}
Thanks to Proposition \ref{existProjProb} we have $\ophi_1$ and $\hphi_1$ solutions to \eqref{projEq}. Thus, if we find $\delta$ and $\varepsilon$ in \eqref{deltaeps} so that $\overline{c}_{0}(\delta,\varepsilon)=\hat{c}_{0}(\delta,\varepsilon)=0$ they actually solve \eqref{eq1} and \eqref{eq2}. Repeating this argument for every $j=1,\ldots,k-1$ and $l=1,\ldots,h-1$ we conclude that
$$u=U_*+\phi$$
with $\phi$ defined in \eqref{defPhi} is the solution to problem  \eqref{prob}  we were looking for.
Thus, we want to prove the existence of $\delta$ and $\varepsilon$ so that (we keep the names in an abuse of notation)
\begin{equation*}\begin{split}
\overline{c}_{0}(\delta,\varepsilon)&=\int_{\mathbb{R}^n}(\ozeta_1E+\gamma\overline{\mathcal{N}}(\oophi_1,\hhphi_1))\oZ_{0}=0,\\
\hat{c}_{0}(\delta,\varepsilon)&=\int_{\mathbb{R}^n}(\hzeta_1E+\gamma\hat{\mathcal{N}}(\oophi_1,\hhphi_1))\hZ_{0}=0.
\end{split}\end{equation*}
Indeed, we will prove that
\begin{equation}\begin{split}\label{eqsc}
\overline{c}_{0}(\delta,\varepsilon)&=-A_n\frac{\delta}{k^{n-2}}[\delta a^1_{n,k} -a^2_{n,k} ]+\frac{1}{k^{n-1}} \Theta_{k,h}(\delta,\varepsilon),\\
\hat{c}_{0}(\delta,\varepsilon)&=-A_n\frac{\varepsilon}{h^{n-2}}[\varepsilon b^1_{n,h} -b^2_{n,h}]+\frac{1}{h^{n-1}}\Theta_{k,h}(\delta,\varepsilon).
\end{split}\end{equation}
Here $A_n$ is a fixed positive constant that depends on $n$, while for $i=1,2$, $ a^i_{n,k}$, $b^i_{n,h}$ are positive constants, of the form
 $a^i_{n,k} = a^i_n + O({1\over k}) $, $b^i_{n,h} = b^i_n + O({1\over h}) $, as $k,h \to \infty$, with $a_n^i$ and $b_n^i$ positive constants.
 Furthermore, $\Theta_{k,h}(\delta,\varepsilon)$ denotes a generic function, which is smooth in its variables,  and it is  uniformly bounded, together with its first derivatives, in $\delta$ and $\varepsilon$ satisfying the bounds \eqref{deltaeps}, when $k\rightarrow \infty$ and $h\rightarrow\infty$.
By a fixed point argument one can prove the existence of a solution $(\delta,\varepsilon)$ to the system
\begin{equation}\begin{split}
\bar c_0 (\delta , \varepsilon ) = \hat c_0 (\delta , \varepsilon ) = 0.
\end{split}\end{equation}
Thus, if we prove that \eqref{eqsc} holds, we conclude the proof of Theorem \ref{mainThm}.

Both estimates in \eqref{eqsc} follow in the same way, so let us prove the first one. We write
\begin{equation}\label{threeTerms}
\overline{c}_{0}(\delta,\varepsilon)=\int_{\mathbb{R}^n}E\oZ_{0}+\int_{\mathbb{R}^n}(\ozeta_1-1)E\oZ_{0}+\gamma\int_{\mathbb{R}^n}\overline{\mathcal{N}}(\oophi_1,\hhphi_1)\oZ_{0}.
\end{equation}
and we analyze every term independently.
\medskip
For $\delta $ and $\varepsilon$ satisfying \eqref{deltaeps}, we have that

\noindent {\bf Claim 1:}
%\begin{equation}\label{claim1}
%\int_{\mathbb{R}^n}E\oZ_{n+1}=-A_n\frac{\delta}{k^{n-2}}[\delta a_n +\overline{a_n}\frac{\varepsilon}{h^{n-3}}-1]+\frac{1}{k^{n-1}}\overline{\Theta}_k(\delta)+\frac{1}{k^{n-\frac{n}{q}}h^{\frac{n}{q}-1}}\overline{\Theta}_k(\delta)\tilde{\Theta}_h(\varepsilon).
%\end{equation}
\begin{equation}\label{claim1}
\int_{\mathbb{R}^n}E\oZ_{0}=-A_n\frac{\delta}{k^{n-2}}[\delta a^1_{n,k} -a^2_{n,k} ]+\frac{1}{k^{n-1}} \Theta_{k,h}(\delta,\varepsilon).
\end{equation}

\noindent{\bf Claim 2:}
\begin{equation}\label{claim2}
\int_{\mathbb{R}^n}(\ozeta_1-1)E\oZ_{0}=\frac{1}{k^{n-1}} \Theta_{k,h}(\delta,\varepsilon).
\end{equation}

\noindent {\bf Claim 3:}
\begin{equation}\label{claim3}
\int_{\mathbb{R}^n}\overline{\mathcal{N}}(\oophi_1,\hhphi_1)\oZ_{0}=\frac{1}{k^{n-1}} \Theta_{k,h}(\delta,\varepsilon),
\end{equation}
as $k$ and $h \to \infty$.
It is clear that these claims imply the validity of the first equation  in \eqref{eqsc}.
\medskip

\noindent {\bf Proof of Claim 1.}
Let us denote
$$Ext:=\{\cap_{j=1}^k\{|y-\xi_j|>\frac{\oalpha}{k}\}\}\cap\{\cap_{l=1}^h\{|y-\eta_l|>\frac{\halpha}{h}\}\}.$$
For $\oalpha>0$ independent of $k$ we can write the first term as
\begin{equation}\label{three}
\int_{\mathbb{R}^n}E\oZ_{0}=\int_{B(\xi_1,\frac{\oalpha}{k})}E\oZ_{0}+\int_{Ext}E\oZ_{0}+\sum_{j\neq 1}\int_{B(\xi_j,\frac{\oalpha}{k})}E\oZ_{0}+\sum_{l= 1}^h\int_{B(\eta_l,\frac{\halpha}{h})}E\oZ_{0}.
\end{equation}
Considering $\overline{E}_1(y)=\mu^{\frac{n+2}{2}}E(\xi_1+\mu y)$ and using \eqref{ErrIntXi} we obtain
\begin{equation*}\begin{split}
\int_{B(\xi_1,\frac{\oalpha}{k})}E\oZ_{0}=&\int_{B(0,\frac{\oalpha}{\mu k})}\overline{E}_1(y)Z_{0}(y)\,dy\\
=&\textcolor{black}{-\gamma p \sum_{j\neq 1}\int_{B(0,\frac{\oalpha}{\mu k})}U^{p-1}U(y-\mu^{-1}(\xi_j-\xi_1))Z_{0}\,dy}\\
&\textcolor{black}{+\gamma p \mu^{\frac{n-2}{2}}\int_{B(0,\frac{\oalpha}{\mu k})}U^{p-1}U(\xi_1+\mu y)Z_{0}\,dy}\\
&\textcolor{black}{-\gamma p \sum_{l=1}^h\mu^{\frac{n-2}{2}}\lambda^{-\frac{n-2}{2}}\int_{B(0,\frac{\oalpha}{\mu k})}U^{p-1}U(\lambda^{-1}(\xi_1+\mu y-\eta_l))Z_{0}\,dy}\\
&\textcolor{black}{+\gamma p\int_{B(0,\frac{\oalpha}{\mu k})}[(U(y)+sV)^{p-1}-U^{p-1}]V(y)Z_{0}\,dy}\\
&-\mu^{\frac{n+2}{2}}\int_{B(0,\frac{\oalpha}{\mu k})}U^p(\xi_1+\mu y)Z_{0}\,dy\\
&-\sum_{j\neq 1} \int_{B(0,\frac{\oalpha}{\mu k})}U^p(y-\mu^{-1}(\xi_j-\xi_1))Z_{0}\,dy\\
&\textcolor{black}{-\sum_{l=1}^h\mu^{\frac{n+2}{2}}\lambda^{-\frac{n+2}{2}}\int_{B(0,\frac{\oalpha}{\mu k})}U^p(\lambda^{-1}(\xi_1+\mu y-\eta_l))Z_{0}\,dy},
\end{split}\end{equation*}
where
\begin{equation*}\begin{split}
V(y):=&-\sum_{j\neq 1}U(y-\mu^{-1}(\xi_j-\xi_1))+\mu^{\frac{n-2}{2}}U(\xi_1+\mu y)\\
&-\sum_{l=1}^h\mu^{\frac{n-2}{2}}\lambda^{-\frac{n-2}{2}}U(\lambda^{-1}(\xi_1+\mu y-\eta_l)).
\end{split}\end{equation*}

\noindent Doing a Taylor expansion, for $j\neq 1$ there holds
\begin{equation*}
\int_{B(0,\frac{\oalpha}{\mu k})}U^{p-1}U(y-\mu^{-1}(\xi_j-\xi_1))Z_{0}\,dy=\frac{c_1 \mu^{n-2}}{|{\xi}_j-{\xi}_1|^{n-2}}(1+{\mu^2
\over |\xi_j - \xi_1|^2}  \Theta_{k,h}(\delta,\varepsilon) ),
\end{equation*}
and
\begin{equation*}
\int_{B(0,\frac{\oalpha}{\mu k})}U^{p-1}U(\lambda^{-1}(\xi_1+\mu y-\eta_l))Z_{0}\,dy=\frac{c_1 \lambda^{n-2}}{|\xi_1-\eta_l|^{n-2}}(1+{\lambda^2
\over |\xi_j - \xi_1|^2}  \Theta_{k,h}(\delta,\varepsilon) ),
\end{equation*}
where $c_1$ is some positive constant, and, as before, $\Theta_{k,h}(\delta,\varepsilon)$ denotes a generic function, which is smooth in its variables,  and it is  uniformly bounded, together with its first derivatives, in $\delta$ and $\varepsilon$ satisfying the bounds \eqref{deltaeps}, when $k\rightarrow \infty$ and $h\rightarrow\infty$.
Proceeding in a similar way,
\begin{equation*}
\mu^{\frac{n-2}{2}}\int_{B(0,\frac{\oalpha}{\mu k})}U^{p-1}U(\xi_1+\mu y)Z_{0}\,dy=c_2\mu^{\frac{n-2}{2}}(1+(\mu k)^2\Theta_{k,h}(\delta,\varepsilon) ),
\end{equation*}
for some positive constant $c_2$.
On the other hand,
\begin{equation}\label{est1}
\bigg|\mu^{\frac{n+2}{2}}\int_{B(0,\frac{\oalpha}{\mu k})}U^p(\xi_1+\mu y)Z_{0}\,dy\bigg|\leq C\mu^{\frac{n+2}{2}}\int_{B(0,\frac{\oalpha}{\mu k})}\frac{1}{(1+|y|)^{n-2}}\leq C\mu^{\frac{n-2}{2}}k^{-2},
\end{equation}
\begin{equation}\begin{split}\label{est2}
\bigg|\sum_{j\neq 1} \int_{B(0,\frac{\oalpha}{\mu k})}U^p(y-&\mu^{-1}(\xi_j-\xi_1))Z_{0}\,dy\bigg|\\&\leq \sum_{j\neq 1}\frac{\mu^{n+2}}{|{\xi}_j-{\xi}_1|^{n+2}}\int_{B(0,\frac{\oalpha}{\mu k})}\frac{1}{(1+|y|)^{n-2}}\\
&\leq C(\mu k)^{-2}\sum_{j\neq 1}\frac{\mu^{n+2}}{|{\xi}_j-{\xi}_1|^{n+2}},
\end{split}\end{equation}
and
\begin{equation}\begin{split}\label{est3}
\bigg|\sum_{l=1}^h\mu^{\frac{n-2}{2}}\lambda^{-\frac{n+2}{2}}&\int_{B(0,\frac{\oalpha}{\mu k})}U^p(\lambda^{-1}(\xi_1+\mu y-\eta_l))Z_{0}\bigg|\\
&\leq C\sum_{l=1}^h\mu^{\frac{n-2}{2}}\lambda^{\frac{n+2}{2}}\int_{B(0,\frac{\oalpha}{\mu k})}\frac{1}{(1+|y|)^{n-2}}\\
&\leq C\mu^{\frac{n+2}{2}}\lambda^{\frac{n+2}{2}}hk^2.
\end{split}\end{equation}
Finally, putting together \eqref{est1}, \eqref{est2} and \eqref{est3},
\begin{equation}\begin{split}
\bigg|\int_{B(0,\frac{\oalpha}{\mu k})}&[(U(y)+sV)^{p-1}-U^{p-1}]V(y)Z_{0}\,dy\bigg|\\
&\leq C\left(\mu^{\frac{n-2}{2}}k^{-2}+(\mu k)^{-2}\sum_{j\neq 1}\frac{\mu^{n+2}}{|{\xi}_j-{\xi}_1|^{n+2}}+\mu^{\frac{n+2}{2}}\lambda^{\frac{n+2}{2}}hk^2\right).
\end{split}\end{equation}
To estimate the second term in \eqref{three} we apply H\"older inequality to get
\begin{equation}\begin{split}\label{Zfuera}
\bigg|\int_{Ext}E&\oZ_{0}\bigg|\leq C\|(1+|y|)^{n+2-\frac{2n}{q}}E\|_{L^q(Ext)}\|(1+|y|)^{-n-2+\frac{2n}{q}}\oZ_{0}\|_{L^{\frac{q}{q-1}}(Ext)}.
\end{split}\end{equation}
Proceeding as in \cite{dPMPP} we have
$$\|(1+|y|)^{-n-2+\frac{2n}{q}}\oZ_{0}\|_{L^{\frac{q}{q-1}}(Ext)}\leq C\mu^{\frac{n-2}{2}}k^{n-2}k^{\frac{n}{q}-n},$$
and using the estimates obtained in Section \ref{error} we see that
\begin{equation*}\begin{split}
\|(1+|y|)^{n+2-\frac{2n}{q}}&E\|_{L^q(Ext)}
\leq  \,C(\mu^{\frac{n-2}{2}}k^{n-2}k^{1-\frac{n}{q}}+\lambda^{\frac{n-2}{2}}h^{n-2}h^{1-\frac{n}{q}}),
\end{split}\end{equation*}
and thus, substituting in \eqref{Zfuera},
\begin{equation}\label{errExt1}
\bigg|\int_{Ext}E\oZ_{0}\bigg|\leq C\left(\frac{\mu^{n-2}k^{2(n-2)}}{k^{n-1}}+\frac{\mu^{\frac{n-2}{2}}\lambda^{\frac{n-2}{2}}k^{n-2}h^{n-2}}{k^{n-\frac{n}{q}}h^{\frac{n}{q}-1}}\right).
\end{equation}
Arguing as in \cite{dPMPP} the third term in \eqref{three} can be estimated as
\begin{equation}\label{errExt2}
\bigg|\sum_{j\neq 1}\int_{B(\xi_j,\frac{\oalpha}{k})}E\oZ_{0}\bigg| \leq \frac{\mu^{\frac{n-2}{2}}}{(\mu k)^{n-4}}\left[\mu^{n-2}\sum_{j\neq 1}\frac{1}{|\xi_j-\xi_1|^{n-2}}\right].
\end{equation}
Likewise,
\begin{equation*}\begin{split}
\bigg|\int_{B(\eta_l,\frac{\halpha}{h})}&E\oZ_0\bigg|=\bigg|\lambda^{\frac{n-2}{2}}\int_{B(0,\frac{\halpha}{\lambda h})}\hat{E}_l(y)\oZ_0(\lambda y+\eta_l)\bigg|\\
&\leq\lambda^{\frac{n-2}{2}}\|(1+|y|)^{n+2-\frac{2n}{q}}\hat{E}_l\|_{L^q(|y|<\frac{\halpha}{\lambda h})}\\
&\;\;\cdot\|(1+|y|)^{-n-2+\frac{2n}{q}}\mu^{-\frac{n-2}{2}}Z_0(\mu^{-1}(\lambda y+\eta_l-\xi_1)\|_{L^{\frac{q}{q-1}}(|y|<\frac{\halpha}{\lambda h})}.
\end{split}\end{equation*}
Noticing that
\begin{equation*}\begin{split}
\|(1+|y|)^{-n-2+\frac{2n}{q}}&\mu^{-\frac{n-2}{2}}Z_0(\mu^{-1}(\lambda y+\eta_l-\xi_1))\|_{L^{\frac{q}{q-1}}(\R^n)}\\
&\leq C\mu^{\frac{n-2}{2}}\left(\int_1^\frac{\halpha}{\lambda h}\frac{t^{n-1}\,dt}{t^{(n+2-\frac{2n}{q})\frac{q}{q-1}}}\right)^{\frac{q-1}{q}}\leq C\mu^\frac{n-2}{2}(\lambda h)^{2-\frac{n}{q}},
\end{split}\end{equation*}
by \eqref{errInt2} we conclude that
\begin{equation}\label{errExt3}
\bigg|\sum_{l=1}^h\int_{B(\eta_l,\frac{\halpha}{h})}E\oZ_0\bigg|\leq C \lambda^{\frac{n-2}{2}}h^{-\frac{n}{q}}\mu^{\frac{n-2}{2}}(\lambda h)^{2-\frac{n}{q}}h
\end{equation}
Claim 1 follows from these estimates applying the fact that $h=O(k)$.
\medskip

\noindent {\bf Proof of Claim 2.}
Let us estimate the second term of \eqref{threeTerms}. Notice first that
$$\bigg|\int_{\mathbb{R}^n}(\ozeta_1-1)E\oZ_{0}\bigg|\leq C\bigg|\int_{\{|y-\xi_1|>\frac{\oalpha}{k}\}}E\oZ_{0}\bigg|,$$
and we separate this integral as
$$\int_{\{|y-\xi_1|>\frac{\oalpha}{k}\}}E\oZ_{0}=\int_{Ext}E\oZ_{0}+\sum_{j=2}^k\int_{\{|y-\xi_j|\leq\frac{\oalpha}{k}\}}E\oZ_{0}+\sum_{l=1}^h\int_{\{|y-\eta_l|\leq\frac{\halpha}{h}\}}E\oZ_{0}.$$
Thus, Claim 2 follows from \eqref{errExt1}, \eqref{errExt2} and \eqref{errExt3}.
\medskip

\noindent {\bf Proof of Claim 3.}
Notice that
$$\int_{\mathbb{R}^n}\overline{\mathcal{N}}(\oophi_1,\hhphi_1)\oZ_{0}=\mu^{\frac{n+2}{2}}\int_{\mathbb{R}^n}\overline{\mathcal{N}}(\oophi_1,\hhphi_1)(\xi_1+\mu y)Z_{0}(y),$$
and thus, from estimates $\eqref{tilf1}-\eqref{tilf4}$ and the fact $h=O(k)$ we conclude
$$\int_{\mathbb{R}^n}\overline{\mathcal{N}}(\oophi_1,\hhphi_1)\oZ_{0}\leq C k^{3-n-\frac{n}{q}}\int_{\mathbb{R}^n}U^{p-1}|Z_{0}|,$$
and the claim follows.

\part{Nondegeneracy}\label{nondegeneracyProof}
As stated before, the goal of this part is to prove the nondegeneracy (see Definition \ref{defNondeg}) of the solution $u$ provided by Theorem \ref{mainThm} (we drop the dependence on $k$ and $h$ by simplicity).  Recalling the functions $z_\alpha$ defined in \eqref{z0}-\eqref{z5} we can formulate the result in Theorem \ref{teounico} as follows.

\begin{theorem}\label{nondeg}
There exists a sequence of solutions $u$ to Problem \eqref{prob}
among the ones constructed in Theorem \ref{mainThm} for which
 all bounded solutions to the equation
\begin{equation}\label{ll11}-\Delta\varphi-\gamma p |u|^{p-2}u \varphi=0
\end{equation}
are linear combination of the functions $z_\alpha$ for $\alpha=0,\ldots,N_0-1.$ (Recall that $N_0 := 5 (n-1)$.)
\end{theorem}

\noindent  For later simplification, we introduce the following functions
\begin{equation}\begin{split}\label{newzalpha}
{\bf z}_\beta := z_\beta, \quad &{\mbox {if}} \quad \beta \not= n+3, n+4, n+5, n+6, \\
{\bf z}_{n+2+\alpha} := { z_\alpha -z_{n+2+\alpha}  \over 2} , \quad &{\mbox {if}} \quad \alpha = 1, 2, 3, 4.
\end{split}\end{equation}
Since ${\bf z}_\beta$ are linear combinations of the original functions $z_\beta$, the statement of Theorem \ref{nondeg} is equivalent to
say that there exists a sequence of solutions among the ones constructed in Theorem \ref{mainThm} for which all bounded solutions to \eqref{ll11}
are linear combinations of ${\bf z}_\beta$, for $\beta = 0, \ldots , N_0-1$.

\medskip
Thus, let  $\varphi$ be a bounded solution of \eqref{ll11}, namely  $L(\varphi)=0$, with $L$ defined in \eqref{lin}.
We decompose $\varphi$ as
\begin{equation}\label{defphitilde}
\varphi(y)=\sum_{\beta=0}^{N_0-1}a_\beta {\bf z}_\beta(y)+\tilde{\varphi}(y),
\end{equation}
where  $a_\beta$ are chosen so that
\begin{equation}\label{ortCond}
\int_{\mathbb{R}^n}|u|^{p-1}{\bf z}_\beta\tilde{\varphi}=0
\end{equation}
holds. Notice that, since ${\bf z}_\beta\in \ker\{L\}$, one has $L(\tilde{\varphi})=0$ and thus our goal will be to prove that actually $\tilde{\varphi}\equiv 0$.

\medskip
Recall that our solution $u$ has the form
\begin{equation}\label{finalform}
u(y)=U(y)-\sum_{j=1}^k U_{\mu,\xi_j}(y)-\sum_{l=1}^h U_{\lambda,\eta_l}(y)+\phi(y),
\end{equation}
where $\phi$ is defined in \eqref{defPhi} as $
\phi=\sum_{j=1}^k\overline{\phi}_j+\sum_{l=1}^h\hat{\phi}_l+\psi$.

We introduce the following functions
\begin{equation}\begin{split}\label{Z0big}
Z_{00}(y)&:=\frac{n-2}{2}\left[ U + \psi \right] (y)+\nabla [U +\psi ] (y)\cdot y  , \\
Z_{\alpha 0} (y) &:= \frac{\partial}{\partial y_\alpha}U(y) + \frac{\partial}{\partial y_\alpha}\psi (y) \qquad \alpha=1,\ldots,n.
\end{split}\end{equation}
For $j$ fixed in $\{ 1, \ldots , k\}$, we define
\begin{equation*}\begin{split}
\oZ_{0j}(y)&:=\frac{n-2}{2}[ U_{\mu,\xi_j} + \bar \phi_j] (y)+\nabla [ U_{\mu,\xi_j}  + \bar \phi_j] (y)\cdot(y-\xi_j)\\
\oZ_{1j}(y)&:= \xi_j \cdot \nabla_y [ U_{\mu , \xi_j} + \bar \phi_j ] (y), \qquad  \oZ_{2j}(y):= \xi_j^\perp \cdot \nabla_y [U_{\mu , \xi_j} + \bar \phi_j]  (y),\\
\oZ_{\alpha j}(y)&:=\frac{\partial}{\partial y_\alpha}[ U_{\mu,\xi_j}(y) + \bar \phi_j ],\qquad \alpha=3,\ldots,n.
\end{split}\end{equation*}
For $l$ fixed in $\{ 1, \ldots , h\}$, we define
\begin{equation*}\begin{split}
\hZ_{0j}(y)&:=\frac{n-2}{2}[ U_{\lambda,\eta_l} + \hat \phi_l] (y)+\nabla [ U_{\lambda,\eta_l}  + \hat \phi_l] (y)\cdot(y-\eta_l)\\
\hZ_{3l}(y)&:= \eta_l \cdot \nabla_y [ U_{\lambda , \eta_l} + \hat \phi_l ] (y), \qquad  \hZ_{4l}(y):= \eta_l^\perp \cdot \nabla_y [U_{\lambda , \eta_l} + \hat \phi_l]  (y),\\
\hZ_{\alpha l}(y)&:=\frac{\partial}{\partial y_\alpha}[ U_{\lambda,\eta_l}(y) + \hat \phi_l ],\qquad \alpha=1, 2, 5, 6,\ldots,n.
\end{split}\end{equation*}
In Appendix \ref{appe1} we provide the expressions of the functions ${\bf z}_\beta$, $\beta = 0 , \ldots , N_0 -1$, in terms of the functions
$Z_{\alpha  0}$, $\overline Z_{\alpha j}$, $j=1, \ldots , k$, and $\hat Z_{\alpha , l} $, $l=1, \ldots , h$, for any $\alpha =0, \ldots , n$. These relations will be useful in other parts of our argument.

\medskip
\noindent
We rearrange the functions above in $(n+1)$ vector fields as
\begin{equation}\label{Pibig}
\Pi_\alpha:=\left[Z_{\alpha0},\oZ_{\alpha 1},\ldots,\oZ_{\alpha k},\hZ_{\alpha 1},\ldots,\hZ_{\alpha h}\right]^T, \quad \alpha = 0 , 1 , \ldots , n,
\end{equation}
and, for any given vector $d=\left[d_0,\od_1,\ldots,\od_k,\hd_1,\ldots,\hd_h\right]^T\in\mathbb{R}^{1+k+h}$ we use the notation
$$d\cdot\Pi_\alpha:=d_0Z_{\alpha 0}+\sum_{j=1}^k\od_j\oZ_{\alpha j}+\sum_{l=1}^h\hd_l\hZ_{\alpha l}.$$
With this in mind, we write the function $\tilde{\varphi}$ in \eqref{defphitilde} as
\begin{equation}\label{tildeVarphi}
\tilde{\varphi}(y)=\sum_{\alpha=0}^n c_\alpha\cdot \Pi_\alpha(y)+\varphi^\perp(y),
\end{equation}
where
$$c_\alpha:=\left[c_{\alpha 0},\oc_{\alpha 1},\ldots,\oc_{\alpha k},\hc_{\alpha 1},\ldots,\hc_{\alpha h}\right]^T,\qquad \alpha=0,\ldots, n,$$
are $(n+1)$ vectors in $\mathbb{R}^{k+h+1}$ chosen so that
\begin{equation*}\begin{split}
&\int_{\mathbb{R}^n}U^{p-1}Z_{\alpha 0}\varphi^\perp =0,\qquad \alpha=0,1,\ldots,n,\\
&\int_{\mathbb{R}^n}U^{p-1}_{\mu,\xi_j}\oZ_{\alpha j}\varphi^\perp =0,\qquad j=1,\ldots,k,\;\;\alpha=0,1,\ldots,n,\\
&\int_{\mathbb{R}^n}U^{p-1}_{\lambda,\eta_l}\hZ_{\alpha l}\varphi^\perp =0,\qquad l=1,\ldots,h,\;\;\alpha=0,1,\ldots,n.
\end{split}\end{equation*}
Hence, to prove that $\tilde{\varphi}\equiv 0$ we have to see that $c_\alpha=0$ for every $\alpha$ and $\varphi^\perp\equiv 0$. This will be consequence of the following three facts.
\medskip

\noindent {\bf Fact 1:} Since $L(\tilde{\varphi})=0$, one has that
\begin{equation}\label{eqPi}
\sum_{\alpha=0}^n c_\alpha\cdot L(\Pi_\alpha)=-L(\varphi^\perp),
\end{equation}
with $L$ defined in \eqref{lin}. We write $\varphi^\perp=\varphi_0^\perp+\sum_{j=1}^k\ovarphi_j^\perp+\sum_{l=1}^h\hvarphi_l^\perp$, where
\begin{equation*}\begin{split}
-L(\varphi_0^\perp)&=\sum_{\alpha=0}^n c_{\alpha 0}L(Z_{\alpha 0}),\\
-L(\ovarphi_j^\perp)&=\sum_{\alpha=0}^n \oc_{\alpha j} L(\oZ_{\alpha j}), \qquad j=1,\ldots, k,\\
-L(\hvarphi_l^\perp)&=\sum_{\alpha=0}^n \hc_{\alpha l} L(\hZ_{\alpha l}), \qquad l=1,\ldots, h.
\end{split}\end{equation*}
Furthermore, let us define
$$\oovarphi_j^\perp(y):=\mu^{\frac{n-2}{2}}\ovarphi_j^\perp(\mu y+\xi_j),\qquad \hhvarphi_l^\perp(y):=\lambda^{\frac{n-2}{2}}\hvarphi_l^\perp(\lambda y+\eta_l),$$
and
$$\|\varphi^\perp\|:=\|\varphi_0^\perp\|_*+\sum_{j=1}^k\|\oovarphi_j^\perp\|_*+\sum_{l=1}^h\|\hhvarphi_l^\perp\|_*.$$
Thus, as we will prove in Section \ref{fact2}, there exists a positive constant $C$ such that
\begin{equation}\label{f2}
\|\varphi^\perp\|\leq C k^{-2 + {2\over n}} \, \sum_{\alpha=0}^n\|c_\alpha\|.
\end{equation}

\medskip
\noindent {\bf Fact 2:} Condition \eqref{ortCond} is equivalent to
\begin{equation}\begin{split}\label{systemFact1}
\sum_{\alpha=0}^nc_\alpha\cdot\int_{\mathbb{R}^n}\Pi_\alpha |u|^{p-1}{\bf z}_\beta=&\sum_{\alpha=0}^n\left[c_{\alpha 0}\int_{\mathbb{R}^n}Z_{\alpha 0}|u|^{p-1}{\bf z}_\beta+\sum_{j=1}^k\oc_{\alpha j}\int_{\mathbb{R}^n}\oZ_{\alpha j}|u|^{p-1}{\bf z}_\beta\right.\\
&\left.+\sum_{l=1}^h\hc_{\alpha l}\int_{\mathbb{R}^n}\hZ_{\alpha l}|u|^{p-1}{\bf z}_\beta\right]\\
=&-\int_{\mathbb{R}^n}\varphi^\perp |u|^{p-1}{\bf z}_\beta,\qquad \beta=0,\ldots,N_0-1.
\end{split}\end{equation}
\medskip

\noindent Let us denote
\begin{equation}\begin{split}\label{cossin}
& {\bf{\ocos }}:=\left[
\begin{array}{c}
1\\\cos\otheta_2\\\ldots\\\cos\otheta_{k-1}
\end{array}
\right], \;
 {\bf{ \osin }}:=\left[
\begin{array}{c}
0\\\sin\otheta_2\\\ldots\\\sin\otheta_{k-1}
\end{array}
\right], \qquad  \otheta_j:=\frac{2\pi}{k}(j-1), \\
 &{\bf \hcos}:=\left[
\begin{array}{c}
1\\\cos\htheta_2\\\ldots\\\cos\htheta_{h-1}
\end{array}
\right], \;
 {\bf \hsin}:=\left[
\begin{array}{c}
0\\\sin\htheta_2\\\ldots\\\sin\htheta_{h-1}
\end{array}
\right], \qquad \htheta_l:=\frac{2\pi}{h}(l-1),
\end{split}\end{equation}
 and
\begin{equation}\label{onezero}
\oone:=\left[
\begin{array}{c}
1\\1\\\ldots\\1
\end{array}
\right], \qquad
\hone:=\left[
\begin{array}{c}
1\\1\\\ldots\\1
\end{array}
\right], \qquad
\ozero:=\left[
\begin{array}{c}
0\\0\\\ldots\\0
\end{array}
\right], \qquad
\hzero:=\left[
\begin{array}{c}
0\\0\\\ldots\\0
\end{array}
\right],
\end{equation}
where $\oone$ and $\ozero$ are $k$-dimensional vectors, and $\hone$ and $\hzero$ are vectors of dimension $h$. Likewise, define
$$
\tilde{c}_0:=\left[
\begin{array}{c}
c_{00}\\\ldots\\c_{n0}
\end{array}
\right]\in \R^{n+1}, \qquad
\overline{c}_\alpha:=\left[
\begin{array}{c}
\overline{c}_{\alpha 1}\\\ldots\\\overline{c}_{\alpha k}
\end{array}
\right]\in \R^k, \qquad
\hat{c}_\alpha:=\left[
\begin{array}{c}
\hat{c}_{\alpha 1}\\\ldots\\\hat{c}_{\alpha h}
\end{array}
\right]\in \R^h,
$$
and
$$
\overline{c}:=\left[
\begin{array}{c}
\overline{c}_0\\\overline{c}_1\\\ldots\\\overline{c}_{n+1}
\end{array}
\right]\in\R^{(n+1)k}, \qquad
\hat{c}:=\left[
\begin{array}{c}
\hat{c}_0\\\hat{c}_1\\\ldots\\\hat{c}_n
\end{array}
\right]\in \R^{(n+1)h}.
$$
Thus,

\begin{prop}\label{systC}Solving system \eqref{systemFact1} is equivalent to solve
\begin{equation}\begin{split}\label{t0}
c_0
\cdot \left[
\begin{array}{c}
1\\-\oone\\\hone
\end{array}
\right]
&
+c_1\cdot \left[
\begin{array}{c}
0\\-\oone\\\hzero
\end{array}
\right]
+c_3\cdot \left[
\begin{array}{c}
0\\\ozero\\-\hone
\end{array}
\right]
=\,t_0\\
&+R_{h,k} [ c_0 , c_1 , \ldots c_n],
\end{split}\end{equation}

\begin{equation}\begin{split}\label{t1}
c_1\cdot \left[
\begin{array}{c}
1\\-\ocos\\-\hone
\end{array}
\right]
&+c_2\cdot \left[
\begin{array}{c}
0\\\osin\\\hzero
\end{array}
\right]
=\,t_1+ R_{h,k} [ c_0 , c_1 , \ldots c_n],
\end{split}\end{equation}

\begin{equation}\begin{split}\label{t2}
c_1\cdot \left[
\begin{array}{c}
0\\-\osin\\\hzero
\end{array}
\right]
&+c_2\cdot \left[
\begin{array}{c}
1\\-\ocos\\-\hone
\end{array}
\right]
=\,t_2+R_{h,k} [ c_0 , c_1 , \ldots c_n],
\end{split}\end{equation}

\begin{equation}\begin{split}\label{t3}
c_3\cdot \left[
\begin{array}{c}
1\\-\oone\\-\hcos
\end{array}
\right]
&
+ c_4\cdot \left[
\begin{array}{c}
0\\\ozero\\\hsin
\end{array}
\right]
=\,t_3+R_{h,k} [ c_0 , c_1 , \ldots c_n],
\end{split}\end{equation}

\begin{equation}\begin{split}\label{t4}
c_3\cdot \left[
\begin{array}{c}
0\\\ozero\\-\hsin
\end{array}
\right]
&+c_4\cdot \left[
\begin{array}{c}
1\\-\oone\\-\hcos
\end{array}
\right]
=\,t_4+R_{h,k} [ c_0 , c_1 , \ldots c_n],
\end{split}\end{equation}

\begin{equation}\begin{split}\label{zalpha}
c_\alpha\cdot \left[
\begin{array}{c}
1\\-\oone\\-\hone
\end{array}
\right]
=t_\alpha+R_{h,k} [ c_0 , c_1 , \ldots c_n],
\end{split}\end{equation}
for $\alpha=5,\ldots,n$,
\begin{equation}\begin{split}\label{nmas1}
c_2\cdot \left[
\begin{array}{c}
0\\\oone\\\hzero
\end{array}
\right]
=t_{n+1}+R_{h,k} [ c_0 , c_1 , \ldots c_n],
\end{split}\end{equation}

\begin{equation}\begin{split}\label{nmas2}
c_4 \cdot \left[
\begin{array}{c}
0\\\ozero\\\hone
\end{array}
\right]
=t_{n+2}+R_{h,k} [ c_0 , c_1 , \ldots c_n],
\end{split}\end{equation}

\begin{equation}\begin{split}\label{nmas3}
c_0\cdot \left[
\begin{array}{c}
0\\\ocos\\\hat 0
\end{array}
\right] -c_1\cdot \left[
\begin{array}{c}
0\\\ocos\\\hat 0
\end{array}
\right]
=t_{n+3}+R_{h,k} [ c_0 , c_1 , \ldots c_n],
\end{split}\end{equation}

\begin{equation}\begin{split}\label{nmas4}
c_0\cdot \left[
\begin{array}{c}
0\\\osin\\ \hat 0
\end{array}
\right] - c_1\cdot \left[
\begin{array}{c}
0\\\osin\\\hat 0
\end{array}
\right]
=t_{n+4}+R_{h,k} [ c_0 , c_1 , \ldots c_n],
\end{split}\end{equation}

\begin{equation}\begin{split}\label{nmas5}
c_0\cdot \left[
\begin{array}{c}
0\\ \overline 0 \\ \hcos
\end{array}
\right] - c_3\cdot \left[
\begin{array}{c}
0\\ \overline 0 \\ \hcos
\end{array}
\right]
=t_{n+5}+R_{h,k} [ c_0 , c_1 , \ldots c_n],
\end{split}\end{equation}

\begin{equation}\begin{split}\label{nmas6}
c_0\cdot \left[
\begin{array}{c}
0\\ \overline 0 \\ \hsin
\end{array}
\right] - c_3\cdot \left[
\begin{array}{c}
0\\ \overline 0 \\ \hsin
\end{array}
\right]
=t_{n+6}+R_{h,k} [ c_0 , c_1 , \ldots c_n],
\end{split}\end{equation}

\begin{equation}\begin{split}\label{nmas7}
c_1\cdot \left[
\begin{array}{c}
0\\\ozero\\\hcos
\end{array}
\right]+
c_3\cdot \left[
\begin{array}{c}
0\\-\ocos\\\hzero
\end{array}
\right]
=t_{n+7}+R_{h,k} [ c_0 , c_1 , \ldots c_n],
\end{split}\end{equation}

\begin{equation}\begin{split}\label{nmas8}
c_1\cdot \left[
\begin{array}{c}
0\\\ozero\\\hsin
\end{array}
\right]+
c_4\cdot \left[
\begin{array}{c}
0\\-\ocos\\\hzero
\end{array}
\right]
=t_{n+8}+R_{h,k} [ c_0 , c_1 , \ldots c_n],
\end{split}\end{equation}

\begin{equation}\begin{split}\label{nmasalphamas4}
c_\alpha\cdot \left[
\begin{array}{c}
0\\-\ocos\\\hzero
\end{array}
\right]
=t_{n+\alpha+4}+R_{h,k} [ c_0 , c_1 , \ldots c_n],
\end{split}\end{equation}
for $\alpha=5,\ldots,n$,
\begin{equation}\begin{split}\label{2nmas5}
c_2\cdot \left[
\begin{array}{c}
0\\\ozero\\\hcos
\end{array}
\right]+
c_3\cdot \left[
\begin{array}{c}
0\\-\osin\\\hzero
\end{array}
\right]
=t_{2n+5} +R_{h,k} [ c_0 , c_1 , \ldots c_n] ,
\end{split}\end{equation}

\begin{equation}\begin{split}\label{2nmas6}
c_2\cdot \left[
\begin{array}{c}
0\\\ozero\\\hsin
\end{array}
\right]+c_4\cdot \left[
\begin{array}{c}
0\\-\osin\\\hzero
\end{array}
\right]
=t_{2n+6} +R_{h,k} [ c_0 , c_1 , \ldots c_n],
\end{split}\end{equation}
and, for $\alpha=5,\ldots,n$,
\begin{equation}\begin{split}\label{2nmasalphamas2}
c_\alpha\cdot \left[
\begin{array}{c}
0\\-\osin\\\hzero
\end{array}
\right]
=t_{2n+\alpha+2}+R_{h,k} [ c_0 , c_1 , \ldots c_n],
\end{split}\end{equation}
\begin{equation}\begin{split}\label{3nmasalphamenos2}
c_\alpha\cdot \left[
\begin{array}{c}
0\\\ozero\\-\hsin
\end{array}
\right]
=t_{3n+\alpha-2} +R_{h,k} [ c_0 , c_1 , \ldots c_n],
\end{split}\end{equation}
\begin{equation}\begin{split}\label{4nmasalphamenos6}
c_\alpha\cdot \left[
\begin{array}{c}
0\\\ozero\\-\hcos
\end{array}
\right]
=t_{4n+\alpha-6}+R_{h,k} [ c_0 , c_1 , \ldots c_n].
\end{split}\end{equation}
 Here $t_i$, $i=0,\ldots, 5n-6$, are fixed numbers such that
$$\|t_i\|\leq C\|\varphi^\perp\|.$$
Moreover, $R_{h,k} [ c_0 , c_1 , \ldots c_n]$ stands for a function, whose specific definition changes from line to line, which can be described as follows:
\begin{equation*}\begin{split}
R_{h,k} [ c_0 , c_1 , \ldots c_n] &= \Theta_{k,h} \mathcal{L} \left(
c_{00}, \ldots , c_{n0}
\right) + \overline \Theta_{k,h} \overline{\mathcal{L}} \left( \overline c_1 , \ldots , \overline c_k \right) \\
&+\hat  \Theta_{k,h} \hat{\mathcal{L}} \left( \hat c_1 , \ldots , \hat c_h \right)
\end{split} \end{equation*}
and $\mathcal{L}:\R^{n+1}\rightarrow {\R}$, $\overline{\mathcal{L}}: \R^{k(n+1)}\rightarrow {\R}$, $\hat{\mathcal{L}}: \R^{h(n+1)}\rightarrow {\R}$ are linear functions uniformly bounded when $k,h\rightarrow\infty$, and
$$\Theta_{k,h}= O (k^{1-\frac{n}{q}}),\qquad \overline \Theta_{k,h} =O ( k^{-\frac{n}{q}}),\qquad \hat \Theta_{k,h} =O( k^{-\frac{n}{q}} ),$$
where $O(1)$ denotes a quantity uniformly bounded when $k,h\rightarrow \infty$, and $\frac{n}{2}<q<n$ is the number fixed in \eqref{normStarStar}.
\end{prop}

\noindent We will prove this result in Section \ref{prop31}.
\medskip

\medskip

\noindent {\bf Fact 3:} Multiplying \eqref{eqPi} for every $Z_{\alpha 0}$, $\oZ_{\alpha j}$, $\hZ_{\alpha l}$, $\alpha=0,1,\ldots, n$, $j=1,\ldots, k$ and $l=1,\ldots, h$ and integrating in $\mathbb{R}^n$ we get a system of the form
\begin{equation}\label{syst}
M
\left[\begin{array}{c}
c_0\\c_1\\\vdots\\c_n
\end{array}\right]=
-\left[\begin{array}{c}
r_0\\r_1\\\vdots\\r_n
\end{array}\right]\qquad \hbox{ with }\qquad r_\alpha:=\left[\begin{array}{c}
\int_{\mathbb{R}^n}L(\varphi^\perp)Z_{\alpha 0}\\
\int_{\mathbb{R}^n}L(\varphi^\perp)\oZ_{\alpha 1}\\
\vdots\\
\int_{\mathbb{R}^n}L(\varphi^\perp)\oZ_{\alpha k}\\
\int_{\mathbb{R}^n}L(\varphi^\perp)\hZ_{\alpha 1}\\
\vdots\\
\int_{\mathbb{R}^n}L(\varphi^\perp)\hZ_{\alpha h}
\end{array}\right].\end{equation}
Due to the symmetries, the matrix $M$ has the form
$$M=
\left[\begin{array}{cc}
M_1 & 0\\
0 & M_2
\end{array}\right],$$
where $M_1$ and $M_2$ are square matrices of dimensions $(5\times(k+h+1))^2$ and $((n-4)\times(k+h+1))^2$ of the form
$$M_1=
\left[\begin{array}{ccccc}
\tilde{A} & \tilde{B} & \tilde{C} & \tilde{D} & \tilde{E}\\
\tilde{B}^T & \tilde{F} & \tilde{G} & \tilde{H} & \tilde{I}\\
\tilde{C}^T & \tilde{G}^T & \tilde{J} & \tilde{K} & \tilde{L}\\
\tilde{D}^T & \tilde{H}^T & \tilde{K}^T & \tilde{M} & \tilde{N}\\
\tilde{E}^T & \tilde{I}^T & \tilde{L}^T & \tilde{N}^T & \tilde{P}
\end{array}\right],\qquad
M_2=
\left[\begin{array}{cccc}
\tilde{H}_5 & 0 & 0 & 0 \\
0 & \tilde{H}_6 & 0 & 0\\
\ldots & \ldots & \ldots & \ldots\\
0 & 0 & 0 & \tilde{H}_n
\end{array}\right],$$
with
\begin{equation}\label{Halpha}
\tilde{H}_\alpha=
\left[\begin{array}{ccc}
\int L(Z_{\alpha 0})Z_{\alpha 0} & \left(\int L(Z_{\alpha 0})\oZ_{\alpha j}\right)_{j}  & \left(\int L(Z_{\alpha 0})\hZ_{\alpha l}\right)_{l} \\
\left(\int L(\oZ_{\alpha i})Z_{\alpha 0}\right)_{i} & \left(\int L(\oZ_{\alpha i})\oZ_{\alpha j}\right)_{i,j} & \left(\int L(\oZ_{\alpha i})\hZ_{\alpha l}\right)_{i,l}\\
\left(\int L(\hZ_{\alpha m})Z_{\alpha 0}\right)_{m} & \left(\int L(\hZ_{\alpha m})\oZ_{\alpha j}\right)_{m,j} &  \left(\int L(\hZ_{\alpha m})\hZ_{\alpha l}\right)_{m,l}
\end{array}\right],
\end{equation}
for $i,j=1,\ldots,k$ and $m,l=1,\ldots,h$. Thus, solving \eqref{syst} is equivalent to find a solution of
\begin{equation}\label{finalSyst}
M_1\left[
\begin{array}{c}
c_0\\
c_1\\
c_2\\
c_3\\
c_4
\end{array}
\right]=
\left[
\begin{array}{c}
r_0\\
r_1\\
r_2\\
r_3\\
r_4
\end{array}
\right]
,\qquad
\tilde{H}_\alpha c_\alpha=r_\alpha\hbox{ for }\alpha=5,\ldots,n,
\end{equation}
with $r_\alpha$ defined in \eqref{syst}.

\begin{prop}\label{nonso}
There exists $k_0,h_0$ such that, for all $k>k_0$, $h>h_0$, system \eqref{finalSyst} is solvable. Moreover, the solution has the form
%\begin{equation*}\begin{split}
%c_0
%=&\;
%v_0+
%t_0\left[
%\begin{array}{c}
%1\\-\oone\\-\hone
%\end{array}
%\right]+t_1\left[
%\begin{array}{c}
%0\\\ozero\\\hzero
%\end{array}
%\right]+t_2\left[
%\begin{array}{c}
%0\\\ozero\\\hzero
%\end{array}
%\right]
%+t_3\left[
%\begin{array}{c}
%0\\\ozero\\\hzero
%\end{array}
%\right]+t_4\left[
%\begin{array}{c}
%0\\\ozero\\\hzero
%\end{array}
%\right]\\
%&+
%\overline{t}_0\left[
%\begin{array}{c}
%0\\ \ozero\\ \hzero\\
%\end{array}\right]+\overline{t}_1 \left[
%\begin{array}{c}
%0\\ \ocos\\ \hzero
%\end{array}\right]
%+\overline{t}_2 \left[
%\begin{array}{c}
%0\\ \osin\\ \hzero
%\end{array}\right]+\overline{t}_3 \left[
%\begin{array}{c}
%0\\ \ozero\\ \hzero
%\end{array}\right]+\overline{t}_4 \left[
%\begin{array}{c}
%0\\ \ozero\\ \hzero
%\end{array}\right]
%\\
%&+\overline{t}_5 \left[
%\begin{array}{c}
%0\\ \ozero\\ \hzero
%\end{array}\right]
%+\overline{t}_6 \left[
%\begin{array}{c}
%0\\ \ozero\\ \hzero
%\end{array}\right]
%+
%\hat{t}_0 \left[
%\begin{array}{c}
%0\\ \ozero\\ \hzero
%\end{array}\right]+\hat{t}_1 \left[
%\begin{array}{c}
%0\\ \ozero\\ \hcos
%\end{array}\right]+\hat{t}_2 \left[
%\begin{array}{c}
%0\\ \ozero\\ \hsin
%\end{array}\right]\\
%&+\hat{t}_3 \left[
%\begin{array}{c}
%0\\ \ozero\\ \hzero
%\end{array}\right]
%+\hat{t}_4 \left[
%\begin{array}{c}
%0\\ \ozero\\ \hzero
%\end{array}\right]
%+\hat{t}_5 \left[
%\begin{array}{c}
%0\\ \ozero\\ \hzero
%\end{array}\right]
%+\hat{t}_6 \left[
%\begin{array}{c}
%0\\ \ozero\\ \hzero
%\end{array}\right],
%\end{split}\end{equation*}
\begin{equation*}\begin{split}
&c_0
=\;
v_0+
t_0\left[
\begin{array}{c}
1\\-\oone\\-\hone
\end{array}
\right]+t_1\left[
\begin{array}{c}
0\\\ozero\\\hzero
\end{array}
\right]+t_2\left[
\begin{array}{c}
0\\\ozero\\\hzero
\end{array}
\right]
+t_3\left[
\begin{array}{c}
0\\\ozero\\\hzero
\end{array}
\right]+t_4\left[
\begin{array}{c}
0\\\ozero\\\hzero
\end{array}
\right]\\
&\qquad+
\overline{t}_0\left[
\begin{array}{c}
0\\ \ozero\\ \hzero\\
\end{array}\right]+\overline{t}_1 \left[
\begin{array}{c}
0\\ \ocos\\ \hzero
\end{array}\right]
+\overline{t}_2 \left[
\begin{array}{c}
0\\ \osin\\ \hzero
\end{array}\right]
+
\hat{t}_0 \left[
\begin{array}{c}
0\\ \ozero\\ \hzero
\end{array}\right]+\hat{t}_1 \left[
\begin{array}{c}
0\\ \ozero\\ \hcos
\end{array}\right]+\hat{t}_2 \left[
\begin{array}{c}
0\\ \ozero\\ \hsin
\end{array}\right],\\
\end{split}\end{equation*}
%\begin{equation*}\begin{split}
%c_1
%=&\;
%v_1+
%t_0\left[
%\begin{array}{c}
%0\\-\oone\\\hzero
%\end{array}
%\right]+t_1\left[
%\begin{array}{c}
%1\\-\frac{1}{\sqrt{1-\mu^2}}\ocos\\-\hone
%\end{array}
%\right]+t_2\left[
%\begin{array}{c}
%0\\-\frac{1}{\sqrt{1-\mu^2}}\osin\\\hzero
%\end{array}
%\right]\\
%&
%+t_3\left[
%\begin{array}{c}
%0\\\ozero\\\hzero
%\end{array}
%\right]+t_4\left[
%\begin{array}{c}
%0\\\ozero\\\hzero
%\end{array}
%\right]
%+
%\overline{t}_0\left[
%\begin{array}{c}
%0\\ \ozero\\ \hzero\\
%\end{array}\right]+\overline{t}_1 \left[
%\begin{array}{c}
%0\\ -\ocos\\ \hzero\\
%\end{array}\right]
%+\overline{t}_2 \left[
%\begin{array}{c}
%0\\ -\osin\\ \hzero\\
%\end{array}\right]\\
%&+\overline{t}_3 \left[
%\begin{array}{c}
%0\\ \ozero\\ \hzero\\
%\end{array}\right]+\overline{t}_4 \left[
%\begin{array}{c}
%0\\ \ozero\\ \hzero\\
%\end{array}\right]
%+\overline{t}_5 \left[
%\begin{array}{c}
%0\\ \ozero\\ \hzero\\
%\end{array}\right]
%+\overline{t}_6 \left[
%\begin{array}{c}
%0\\ \ozero\\ \hzero\\
%\end{array}\right]
%+
%\hat{t}_0 \left[
%\begin{array}{c}
%0\\ \ozero\\ \hzero\\
%\end{array}\right]+\hat{t}_1 \left[
%\begin{array}{c}
%0\\ \ozero\\ \hzero\\
%\end{array}\right]\\
%&+\hat{t}_2 \left[
%\begin{array}{c}
%0\\ \ozero\\ \hzero\\
%\end{array}\right]+\hat{t}_3 \left[
%\begin{array}{c}
%0\\ \ozero\\ \hcos\\
%\end{array}\right]
%+\hat{t}_4 \left[
%\begin{array}{c}
%0\\ \ozero\\ \hsin\\
%\end{array}\right]
%+\hat{t}_5 \left[
%\begin{array}{c}
%0\\ \ozero\\ \hzero\\
%\end{array}\right]
%+\hat{t}_6 \left[
%\begin{array}{c}
%0\\ \ozero\\ \hzero\\
%\end{array}\right],
%\end{split}\end{equation*}
\begin{equation*}\begin{split}
&c_1
=\;
v_1+
t_0\left[
\begin{array}{c}
0\\-\oone\\\hzero
\end{array}
\right]+t_1\left[
\begin{array}{c}
1\\-\frac{1}{\sqrt{1-\mu^2}}\ocos\\-\hone
\end{array}
\right]+t_2\left[
\begin{array}{c}
0\\-\frac{1}{\sqrt{1-\mu^2}}\osin\\\hzero
\end{array}
\right]+t_3\left[
\begin{array}{c}
0\\\ozero\\\hzero
\end{array}
\right]+t_4\left[
\begin{array}{c}
0\\\ozero\\\hzero
\end{array}
\right]\\
&\qquad
+
\overline{t}_0\left[
\begin{array}{c}
0\\ \ozero\\ \hzero\\
\end{array}\right]+\overline{t}_1 \left[
\begin{array}{c}
0\\ -\ocos\\ \hzero\\
\end{array}\right]
+\overline{t}_2 \left[
\begin{array}{c}
0\\ -\osin\\ \hzero\\
\end{array}\right]
+
\hat{t}_0 \left[
\begin{array}{c}
0\\ \ozero\\ \hzero\\
\end{array}\right]+\hat{t}_1 \left[
\begin{array}{c}
0\\ \ozero\\ \hzero\\
\end{array}\right]+\hat{t}_2 \left[
\begin{array}{c}
0\\ \ozero\\ \hzero
\end{array}\right],
\end{split}\end{equation*}
%\begin{equation*}\begin{split}
%c_2
%=&\;
%v_2+
%t_0\left[
%\begin{array}{c}
%0\\\ozero\\\hzero
%\end{array}
%\right]+t_1\left[
%\begin{array}{c}
%0\\\frac{1}{\sqrt{1-\mu^2}}\osin\\\hzero
%\end{array}
%\right]+t_2\left[
%\begin{array}{c}
%1\\-\frac{1}{\sqrt{1-\mu^2}}\ocos\\-\hone
%\end{array}
%\right]+t_3\left[
%\begin{array}{c}
%0\\\ozero\\\hzero
%\end{array}
%\right]\\
%&+t_4\left[
%\begin{array}{c}
%0\\\ozero\\\hzero
%\end{array}
%\right]
%+
%\overline{t}_0\left[
%\begin{array}{c}
%0\\ \oone\\ \hzero\\
%\end{array}\right]+\overline{t}_1 \left[
%\begin{array}{c}
%0\\ \ozero\\ \hzero\\
%\end{array}\right]
%+\overline{t}_2 \left[
%\begin{array}{c}
%0\\ \ozero\\ \hzero\\
%\end{array}\right]+\overline{t}_3 \left[
%\begin{array}{c}
%0\\ \ozero\\ \hzero\\
%\end{array}\right]+\overline{t}_4 \left[
%\begin{array}{c}
%0\\ \ozero\\ \hzero\\
%\end{array}\right]
%\\
%&+\overline{t}_5 \left[
%\begin{array}{c}
%0\\ \ozero\\ \hzero\\
%\end{array}\right]
%+\overline{t}_6 \left[
%\begin{array}{c}
%0\\ \ozero\\ \hzero\\
%\end{array}\right]
%+
%\hat{t}_0 \left[
%\begin{array}{c}
%0\\ \ozero\\ \hzero\\
%\end{array}\right]+\hat{t}_1 \left[
%\begin{array}{c}
%0\\ \ozero\\ \hzero\\
%\end{array}\right]+\hat{t}_2 \left[
%\begin{array}{c}
%0\\ \ozero\\ \hzero\\
%\end{array}\right]+\hat{t}_3 \left[
%\begin{array}{c}
%0\\ \ozero\\ \hzero\\
%\end{array}\right]
%\\
%&+\hat{t}_4 \left[
%\begin{array}{c}
%0\\ \ozero\\ \hzero\\
%\end{array}\right]
%+\hat{t}_5 \left[
%\begin{array}{c}
%0\\ \ozero\\ \hcos\\
%\end{array}\right]
%+\hat{t}_6 \left[
%\begin{array}{c}
%0\\ \ozero\\ \hsin\\
%\end{array}\right],
%\end{split}\end{equation*}
\begin{equation*}\begin{split}
&c_2
=\;
v_2+
t_0\left[
\begin{array}{c}
0\\\ozero\\\hzero
\end{array}
\right]+t_1\left[
\begin{array}{c}
0\\\frac{1}{\sqrt{1-\mu^2}}\osin\\\hzero
\end{array}
\right]+t_2\left[
\begin{array}{c}
1\\-\frac{1}{\sqrt{1-\mu^2}}\ocos\\-\hone
\end{array}
\right]+t_3\left[
\begin{array}{c}
0\\\ozero\\\hzero
\end{array}
\right]+t_4\left[
\begin{array}{c}
0\\\ozero\\\hzero
\end{array}
\right]\\
&\qquad
+
\overline{t}_0\left[
\begin{array}{c}
0\\ \oone\\ \hzero\\
\end{array}\right]+\overline{t}_1 \left[
\begin{array}{c}
0\\ \ozero\\ \hzero\\
\end{array}\right]
+\overline{t}_2 \left[
\begin{array}{c}
0\\ \ozero\\ \hzero\\
\end{array}\right]+
\hat{t}_0 \left[
\begin{array}{c}
0\\ \ozero\\ \hzero\\
\end{array}\right]
+\hat{t}_1 \left[
\begin{array}{c}
0\\ \ozero\\ \hzero\\
\end{array}\right]+\hat{t}_2 \left[
\begin{array}{c}
0\\ \ozero\\ \hzero\\
\end{array}\right],
\end{split}\end{equation*}
%\begin{equation*}\begin{split}
%c_3
%=&\;
%v_3+
%t_0\left[
%\begin{array}{c}
%0\\\ozero\\-\hone
%\end{array}
%\right]+t_1\left[
%\begin{array}{c}
%0\\\ozero\\\hzero
%\end{array}
%\right]+t_2\left[
%\begin{array}{c}
%0\\\ozero\\\hzero\\
%\end{array}
%\right]+t_3\left[
%\begin{array}{c}
%1\\-\oone\\-\frac{1}{\sqrt{1-\lambda^2}}\hcos
%\end{array}
%\right]\\
%&+t_4\left[
%\begin{array}{c}
%0\\\ozero\\-\frac{1}{\sqrt{1-\lambda^2}}\hsin
%\end{array}
%\right]
%+
%\overline{t}_0\left[
%\begin{array}{c}
%0\\ \ozero\\ \hzero
%\end{array}\right]+\overline{t}_1 \left[
%\begin{array}{c}
%0\\ \ozero\\ \hzero
%\end{array}\right]
%+\overline{t}_2 \left[
%\begin{array}{c}
%0\\ \ozero\\ \hzero
%\end{array}\right]+\overline{t}_3 \left[
%\begin{array}{c}
%0\\ \ocos\\ \hzero
%\end{array}\right]\\
%&+\overline{t}_4 \left[
%\begin{array}{c}
%0\\ \osin\\ \hzero
%\end{array}\right]
%+\overline{t}_5 \left[
%\begin{array}{c}
%0\\ \ozero\\ \hzero
%\end{array}\right]
%+\overline{t}_6 \left[
%\begin{array}{c}
%0\\ \ozero\\ \hzero
%\end{array}\right]
%+
%\hat{t}_0 \left[
%\begin{array}{c}
%0\\ \ozero\\ \hzero
%\end{array}\right]+\hat{t}_1 \left[
%\begin{array}{c}
%0\\ \ozero\\ -\hcos
%\end{array}\right]\\
%&+\hat{t}_2 \left[
%\begin{array}{c}
%0\\ \ozero\\ -\hsin
%\end{array}\right]+\hat{t}_3 \left[
%\begin{array}{c}
%0\\ \ozero\\ \hzero
%\end{array}\right]
%+\hat{t}_4 \left[
%\begin{array}{c}
%0\\ \ozero\\ \hzero
%\end{array}\right]
%+\hat{t}_5 \left[
%\begin{array}{c}
%0\\ \ozero\\ \hzero
%\end{array}\right]
%+\hat{t}_6 \left[
%\begin{array}{c}
%0\\ \ozero\\ \hzero
%\end{array}\right],
%\end{split}\end{equation*}
\begin{equation*}\begin{split}
&c_3
=\;
v_3+
t_0\left[
\begin{array}{c}
0\\\ozero\\-\hone
\end{array}
\right]+t_1\left[
\begin{array}{c}
0\\\ozero\\\hzero
\end{array}
\right]+t_2\left[
\begin{array}{c}
0\\\ozero\\\hzero\\
\end{array}
\right]+t_3\left[
\begin{array}{c}
1\\-\oone\\-\frac{1}{\sqrt{1-\lambda^2}}\hcos
\end{array}
\right]+t_4\left[
\begin{array}{c}
0\\\ozero\\-\frac{1}{\sqrt{1-\lambda^2}}\hsin
\end{array}
\right]\\
&\qquad
+
\overline{t}_0\left[
\begin{array}{c}
0\\ \ozero\\ \hzero
\end{array}\right]+\overline{t}_1 \left[
\begin{array}{c}
0\\ \ozero\\ \hzero
\end{array}\right]
+\overline{t}_2 \left[
\begin{array}{c}
0\\ \ozero\\ \hzero
\end{array}\right]
+
\hat{t}_0 \left[
\begin{array}{c}
0\\ \ozero\\ \hzero
\end{array}\right]+\hat{t}_1 \left[
\begin{array}{c}
0\\ \ozero\\ -\hcos
\end{array}\right]+\hat{t}_2 \left[
\begin{array}{c}
0\\ \ozero\\ -\hsin
\end{array}\right],
\end{split}\end{equation*}
%\begin{equation*}\begin{split}
%c_4
%= &\;
%v_4+
%t_0\left[
%\begin{array}{c}
%0\\\ozero\\\hzero
%\end{array}
%\right]+t_1\left[
%\begin{array}{c}
%0\\\ozero\\\hzero
%\end{array}
%\right]+t_2\left[
%\begin{array}{c}
%0\\\ozero\\\hzero
%\end{array}
%\right]+t_3\left[
%\begin{array}{c}
%0\\\ozero\\\frac{1}{\sqrt{1-\lambda^2}}\hsin
%\end{array}
%\right]\\
%&+t_4\left[
%\begin{array}{c}
%1\\-\oone\\-\frac{1}{\sqrt{1-\lambda^2}}\hcos
%\end{array}
%\right]
%+
%\overline{t}_0\left[
%\begin{array}{c}
%0\\ \ozero\\ \hzero
%\end{array}\right]+\overline{t}_1 \left[
%\begin{array}{c}
%0\\ \ozero\\ \hzero
%\end{array}\right]+\overline{t}_2 \left[
%\begin{array}{c}
%0\\ \ozero\\ \hzero
%\end{array}\right]+\overline{t}_3 \left[
%\begin{array}{c}
%0\\ \ozero\\ \hzero
%\end{array}\right]\\
%&+\overline{t}_4 \left[
%\begin{array}{c}
%0\\ \ozero\\ \hzero
%\end{array}\right]
%+\overline{t}_5 \left[
%\begin{array}{c}
%0\\ \ocos\\ \hzero
%\end{array}\right]
%+\overline{t}_6 \left[
%\begin{array}{c}
%0\\ \osin\\ \hzero
%\end{array}\right]
%+
%\hat{t}_0 \left[
%\begin{array}{c}
%0\\ \ozero\\ \hone
%\end{array}\right]+\hat{t}_1 \left[
%\begin{array}{c}
%0\\ \ozero\\ \hzero
%\end{array}\right]\\
%&+\hat{t}_2 \left[
%\begin{array}{c}
%0\\ \ozero\\ \hzero
%\end{array}\right]+\hat{t}_3 \left[
%\begin{array}{c}
%0\\ \ozero\\ \hzero
%\end{array}\right]
%+\hat{t}_4 \left[
%\begin{array}{c}
%0\\ \ozero\\ \hzero
%\end{array}\right]
%+\hat{t}_5 \left[
%\begin{array}{c}
%0\\ \ozero\\ \hzero
%\end{array}\right]
%+\hat{t}_6 \left[
%\begin{array}{c}
%0\\ \ozero\\ \hzero
%\end{array}\right],
%\end{split}\end{equation*}
\begin{equation*}\begin{split}
&c_4
= \;
v_4+
t_0\left[
\begin{array}{c}
0\\\ozero\\\hzero
\end{array}
\right]+t_1\left[
\begin{array}{c}
0\\\ozero\\\hzero
\end{array}
\right]+t_2\left[
\begin{array}{c}
0\\\ozero\\\hzero
\end{array}
\right]+t_3\left[
\begin{array}{c}
0\\\ozero\\\frac{1}{\sqrt{1-\lambda^2}}\hsin
\end{array}
\right]+t_4\left[
\begin{array}{c}
1\\-\oone\\-\frac{1}{\sqrt{1-\lambda^2}}\hcos
\end{array}
\right]\\
&\qquad
+
\overline{t}_0\left[
\begin{array}{c}
0\\ \ozero\\ \hzero
\end{array}\right]+\overline{t}_1 \left[
\begin{array}{c}
0\\ \ozero\\ \hzero
\end{array}\right]+\overline{t}_2 \left[
\begin{array}{c}
0\\ \ozero\\ \hzero
\end{array}\right]
+
\hat{t}_0 \left[
\begin{array}{c}
0\\ \ozero\\ \hone
\end{array}\right]+\hat{t}_1 \left[
\begin{array}{c}
0\\ \ozero\\ \hzero
\end{array}\right]+\hat{t}_2 \left[
\begin{array}{c}
0\\ \ozero\\ \hzero
\end{array}\right],
\end{split}\end{equation*}
 and, for $\alpha=5,\ldots,n$,
\begin{equation*}\begin{split}
c_\alpha&=v_\alpha
+t_{\alpha }\left[\begin{array}{c}
1\\-\oone\\-\hone
\end{array}\right] + \overline \nu_{\alpha 1} \left[\begin{array}{c}
0\\\ocos\\ \hzero
\end{array}\right] + \overline \nu_{\alpha 2} \left[\begin{array}{c}
0\\\osin \\ \hzero
\end{array}\right] + \hat \nu_{\alpha 1} \left[\begin{array}{c}
0\\\ozero\\ \hcos
\end{array}\right] +\hat \nu_{\alpha 2} \left[\begin{array}{c}
0\\\ozero\\ \hsin
\end{array}\right],
\end{split}
\end{equation*}
for any $t_0,t_1,t_2,t_3,t_4$, $\overline{t}_0,\overline{t}_1,\overline{t}_2$,$\hat{t}_0,\hat{t}_1,\hat{t}_2$, and $t_{\alpha } , \overline \nu_{\alpha 1} , \overline \nu_{\alpha 2} , \hat \nu_{\alpha 1} , \hat \nu_{\alpha 2}$ real parameters. The vectors $v_\alpha\in \R^{k+h+1}$ are fixed and satisfy
$$\|v_\alpha\|\leq C\|\varphi^\perp\|, \;\alpha=0,1,\ldots, n.$$
\end{prop}

\begin{proof}
Proceeding as in \cite[Proposition 6.1]{MW} it can be checked that, for any $\alpha=0,\ldots,n$,
$$\|\overline{r}_\alpha\|\leq C\mu^{\frac{n-2}{2}}\|\varphi^\perp\|,\qquad \|\hat{r}_\alpha\|\leq C\lambda^{\frac{n-2}{2}}\|\varphi^\perp\|,$$
and combining this estimate with Lemma \ref{uno} and Lemma \ref{dos} we obtain the result.
\end{proof}

%\noindent Proposition \ref{systC} implies the uniqueness of the solution $[c_0\,\, c_1\,\, \ldots\,\, c_n]^T$ of \eqref{finalSyst}. Moreover,
%$$\sum_{\alpha=0}^n\|c_\alpha\|\leq C\|\varphi^\perp\|_*,$$
%for some constant $C>0$ and this estimate, together with \eqref{f2}, allows us to conclude
%$$c_\alpha =0\qquad\forall\alpha=0,\ldots,n,$$
%and thus $\varphi^\perp\equiv 0$. Replacing this in \eqref{tildeVarphi} the proof of Theorem \ref{nondeg} is complete.

\medskip
\noindent
We shall use the following notations:
for any $\alpha=0,1,\ldots,n,$
\begin{equation}\label{vectorsC}
\tilde{c}_\alpha:=\left[\begin{array}{c}
\oc_\alpha\\\hc_{\alpha}
\end{array}\right] \in \R^{k+h} ,\;\;\oc_\alpha:=\left[\begin{array}{c}\oc_{\alpha 1}\\\ldots\\\oc_{\alpha k}\end{array}\right]\in \R^k ,\;\;\hc_\alpha:=\left[\begin{array}{c}\hc_{\alpha 1}\\\ldots\\\hc_{\alpha h}\end{array}\right]\in \R^h,
\end{equation}
\begin{equation*}
\tilde{r}_\alpha:=\left[\begin{array}{c}
\ovr_\alpha\\\hr_\alpha
\end{array}\right]\in \R^{k+h},\;\;
\end{equation*}
where
\begin{equation*}
\ovr_\alpha:=\left[\begin{array}{c}
\int_{\mathbb{R}^n}L(\varphi^\perp)\oZ_{\alpha 1}\\
\vdots\\
\int_{\mathbb{R}^n}L(\varphi^\perp)\oZ_{\alpha k}\end{array}\right]\in \R^k,\;\;
\hr_\alpha:=\left[\begin{array}{c}
\int_{\mathbb{R}^n}L(\varphi^\perp)\hZ_{\alpha 1}\\
\vdots\\
\int_{\mathbb{R}^n}L(\varphi^\perp)\hZ_{\alpha h}
\end{array}\right]\in \R^h.
\end{equation*}

\medskip
\section{Solving the second system in \eqref{finalSyst}}

Let $\alpha $ be fixed in $\{ 5, \ldots , n \}$. This section is devoted to solve
\begin{equation}\label{semplice}
\tilde{H}_\alpha c_\alpha=r_\alpha,
\end{equation}
where $r_\alpha$ is the vector defined in \eqref{syst}.

Using \eqref{Halpha} and \eqref{Z0big} and the fact that $L({\bf z}_\alpha)=0$ it follows that
\begin{equation*}
\hbox{row}_1(\tilde{H}_\alpha)=\sum_{l=2}^{k+h+1}\hbox{row}_l(\tilde{H}_\alpha).
\end{equation*}
As a consequence,
$\left[\begin{array}{c}
1\\
-\oone\\
-\hone
\end{array}\right]\in\ker(\tilde{H}_\alpha),$
and hence $\tilde{H}_\alpha c_\alpha=r_\alpha$ has a solution only if
$r_\alpha\cdot\left[\begin{array}{c}
1\\
-\oone\\
-\hone
\end{array}\right]=0$.
This last orthogonality condition is indeed fulfilled since one has
\begin{equation}\label{rowalpha}
\hbox{row}_1(r_\alpha)=\sum_{j=2}^{k+1}\hbox{row}_j(r_\alpha)+\sum_{l=k+2}^{h+k+1}\hbox{row}_l(r_\alpha),
\end{equation}
again as consequence of the fact that $L({\bf z}_\alpha)=0$. Thus, the general solution to \eqref{semplice} has the form
\begin{equation*}
c_\alpha=\left[\begin{array}{c}
0\\\tilde{c}_\alpha
\end{array}\right]
+t\left[\begin{array}{c}
1\\-\oone\\-\hone
\end{array}\right],\qquad t\in\R,
\end{equation*}
where $\tilde c_\alpha$ solves
\begin{equation}\label{semplice1}
H_\alpha \tilde c_\alpha = \tilde r_\alpha,
\end{equation}
with
\begin{equation*}
H_\alpha=\left[
\begin{array}{cc}
\overline{H}_\alpha & \gamma_\alpha \mathds{1}_{k\times h}\\
\gamma_\alpha \mathds{1}_{h\times k} & \hat{H}_\alpha
\end{array}
\right],\qquad \alpha=5,\ldots,n,
\end{equation*}
being $\bar H_\alpha$ and $\hat H_\alpha$ square matrices of dimensions $k\times k$ and $h\times h$ respectively, defined by
$$
\bar H_\alpha := \left( \int L(\bar Z_{\alpha , i} ) \bar Z_{\alpha ,j} \, dy \right)_{i,j=1, \ldots k}, \quad
\hat H_\alpha := \left( \int L(\hat Z_{\alpha , l} ) \hat Z_{\alpha ,m} \, dy \right)_{l,m=1, \ldots h}
$$
and
$$\gamma_\alpha:=\int_{\R^n}L(\oZ_{\alpha 1})\hZ_{\alpha 1}.$$
By $\mathds{1}_{s\times t}$ we mean a $s\times t$-dimensional matrix whose entries are all 1.
Observe that
\begin{equation}\label{ss0}
\left| \gamma_\alpha \right| \leq C k^{4-2n},
\end{equation}
for some fixed constant $C$.
Arguing as in \cite{MW}, one can show that $\bar H_\alpha$ and $\hat H_\alpha$ are circulant matrices of dimensions $(k\times k)$ and $(h\times h)$ respectively (see \cite{KS} for properties).
Moreover, \cite[Proposition 5.1]{MW} ensures that
\begin{equation}\label{semplice3}
\overline{H}_\alpha[\oc_\alpha]=\os_\alpha,\;\; \hat{H}_\alpha[\hc_\alpha]=\hs_\alpha,
\end{equation}
has a solution if
$$\os_\alpha\cdot\ocos = \os_\alpha\cdot\osin = 0\;\hbox{ and }\;\hs_\alpha\cdot\hcos = \hs_\alpha\cdot\hsin = 0.$$
Actually, if a solution to \eqref{semplice3} exists, it has the form
\begin{equation*}\begin{split}
&\oc_\alpha=\ow_\alpha+\overline \nu_1\ocos+\overline \nu_2\osin, \quad
\hc_\alpha=\hw_\alpha+\hat \nu_1\hcos+\hat \nu_2\hsin,
\end{split}\end{equation*}
for all $\overline \nu_1,\,\overline \nu_2,\, \hat \nu_1,\, \hat \nu_2\in\R$, where $\ow_\alpha,\, \hw_\alpha$ are the unique solutions to
$$
\overline H_\alpha \ow_\alpha = \os_\alpha , \quad \ow_\alpha \cdot \ocos = \ow_\alpha \cdot \osin = 0
$$
and
$$
\hat H_\alpha \hw_\alpha = \hs_\alpha , \quad \hw_\alpha \cdot \hcos = \hw_\alpha \cdot \hsin = 0.
$$
Furthermore, in \cite[Proposition 5.1]{MW}  it is proved that there exists a constant $C$ independent of $k$ so that, for all $k $ large
\begin{equation*}
\|\ow_\alpha \|\leq C k^{n-4} \|\os_\alpha \|\;\;\hbox{and}\;\;\|\hw_\alpha\|\leq C k^{n-4} \|\hs_\alpha\|.
\end{equation*}

We start with the observation that system \eqref{semplice1} is solvable. Indeed,
since $L({\bf z}_{n+\alpha + 4} )= L ({\bf z}_{2n+\alpha+2} ) = 0$,  one has that $\overline r_\alpha \cdot \ocos =
\overline r_\alpha \cdot \osin =0$. Similarly, one gets that $\hat r_\alpha \cdot \hcos =
\hat r_\alpha \cdot \hsin =0$, as consequence of the fact that $L({\bf z}_{3n+\alpha -2} )= L ({\bf z}_{4n+\alpha-6} ) = 0$. Moreover, the vector $\mathds{1}_{k\times h} \hat c_\alpha$ is a multiple of $\oone$, and $\mathds{1}_{h\times k} \overline c_\alpha$ is a multiple of $\hone$. Thus $\mathds{1}_{k\times h} \hat c_\alpha \cdot \cos = \mathds{1}_{k\times h} \hat c_\alpha \cdot \sin =0$, and
$\mathds{1}_{h\times k} \overline c_\alpha \cdot \hcos = \mathds{1}_{h\times k} \overline c_\alpha\cdot \hsin =0$.
Now we observe that the solution of system \eqref{semplice1} has the form
\begin{equation}\begin{split} \label{ss1}
\overline c_\alpha &= \ow_\alpha + \overline \nu_1 \ocos + \overline \nu_2 \osin \\
 \hat c_\alpha &= \hw_\alpha + \hat \nu_1 \ocos + \hat \nu_2 \osin ,
\end{split} \end{equation}
for any value for $\overline \nu_1,\,\overline \nu_2,\, \hat \nu_1,\, \hat \nu_2\in\R$, where $\ow_\alpha$ and $\hw_\alpha$ are the unique solutions to
\begin{equation}\begin{split} \label{ss2}
\overline H_\alpha \ow_\alpha &= \overline r_\alpha -\gamma_\alpha \mathds{1}_{k\times h} \hat w_\alpha, \quad \ow_\alpha \cdot \ocos = \ow_\alpha \cdot \osin = 0, \\
\hat H_\alpha \hw_\alpha &= \hat r_\alpha -\gamma_\alpha \mathds{1}_{h\times k} \overline w_\alpha, \quad \hw_\alpha \cdot \hcos = \hw_\alpha \cdot \hsin = 0.
\end{split}
\end{equation}
Moreover, there exists a constant $C$ so that
\begin{equation}
\label{ss3}
\|\ow_\alpha \|\leq C k^{n-4} \|\overline r_\alpha \|\;\;\hbox{and}\;\;\|\hw_\alpha\|\leq C k^{n-4} \|\hat r_\alpha\|.
\end{equation}
Existence and uniqueness of solutions to \eqref{ss2} satisfying \eqref{ss3} follows from a contraction map argument. Indeed,
$\left[\begin{array}{c}
\ow_\alpha \\\hw_\alpha
\end{array}\right]$ is a solution if and only if it is a fixed point to $A_\alpha \left[\begin{array}{c}
\ow_\alpha \\\hw_\alpha
\end{array}\right] := T_\alpha^{-1} \left( \left[\begin{array}{c}
\overline r_\alpha -\gamma_\alpha \mathds{1}_{k\times h} \hat w_\alpha \\ \hat r_\alpha -\gamma_\alpha \mathds{1}_{h\times k} \overline w_\alpha
\end{array}\right] \right)$, where
we denote by $T_\alpha$ the linear map $T_\alpha \left( \left[\begin{array}{c}
\ow_\alpha \\\hw_\alpha
\end{array}\right] \right) = \left( \left[\begin{array}{c}
\overline H_\alpha \ow_\alpha \\ \hat H_\alpha \hw_\alpha
\end{array}\right] \right)$, which is invertible for vectors that are orthogonal to $\ocos $, $\osin$, in their first components, and to $\hcos$, $\hsin$ in their second components.
Let
$$
B_r := \{ \left[\begin{array}{c}
\ow_\alpha \\\hw_\alpha
\end{array}\right] \in K_\alpha \, : \,  \|\ow_\alpha \|\leq r k^{n-4} \|\overline r_\alpha \|\;\;\hbox{and}\;\;\|\hw_\alpha\|\leq r k^{n-4} \|\hat r_\alpha\| \},
$$
where $K_\alpha := \{ \left[\begin{array}{c}
\ow_\alpha \\\hw_\alpha
\end{array}\right]  \in \R^{k+h} \, : \, \ow_\alpha \cdot \ocos = \ow_\alpha \cdot \osin =0, \hw_\alpha \cdot \hcos = \hw_\alpha \cdot \hsin =0 \}$. Then, choosing $r$ large but fixed, and thanks to \eqref{ss0}, one has that $A_\alpha$ is a contraction in $B_r$. This gives the existence of solutions to \eqref{ss2}, satisfying \eqref{ss3}.

\medskip
\noindent
Summarizing the above arguments, we have

\begin{lemma} \label{uno}
Let $\alpha \in \{5, \ldots ,n \}$ be fixed. Then system \eqref{semplice} is solvable, and the solution has the form
\begin{equation}\begin{split} \label{solsemplice}
c_\alpha&=\left[\begin{array}{c}
0\\\ow_\alpha \\ \hw_\alpha
\end{array}\right]
+t\left[\begin{array}{c}
1\\-\oone\\-\hone
\end{array}\right] + \overline \nu_1 \left[\begin{array}{c}
0\\\ocos\\ \hzero
\end{array}\right] + \overline \nu_2 \left[\begin{array}{c}
0\\\osin \\ \hzero
\end{array}\right] + \hat \nu_1 \left[\begin{array}{c}
0\\\ozero\\ \hcos
\end{array}\right] +\hat \nu_2 \left[\begin{array}{c}
0\\\ozero\\ \hsin
\end{array}\right],
\end{split}
\end{equation}
for any values of $t , \overline \nu_1 , \overline \nu_2 , \hat \nu_1 , \hat \nu_2 \in \R$.
In the above formula $\left[\begin{array}{c}
\ow_\alpha\\ \hw_\alpha
\end{array}\right]$ is the unique solution to \eqref{ss2}, and satisfies \eqref{ss3}.
\end{lemma}

\medskip
\section{Solving the first system in \eqref{finalSyst}}

This section is devoted to solve the first system in \eqref{finalSyst}, namely
\begin{equation}\label{difficile}
M_1\left[
\begin{array}{c}
c_0\\
c_1\\
c_2\\
c_3\\
c_4
\end{array}
\right]=
\left[
\begin{array}{c}
r_0\\
r_1\\
r_2\\
r_3\\
r_4
\end{array}
\right].
\end{equation}
Using \eqref{Halpha} and \eqref{Z0big} and the fact that $L({\bf z}_\alpha)=0$ for every $\alpha=0,\ldots,4$, together with the result in Section \ref{appe1},
we observe that
\begin{equation*}\begin{split}
\hbox{row}_1(M_1)=&\sum_{i=1}^{k}\hbox{row}_{1+i}(M_1)+\hbox{row}_{k+h+2+i}(M_1)+\sum_{i=1}^{h}\hbox{row}_{k+1+i}(M_1)+\hbox{row}_{4k+3h+4+i}(M_1),
\end{split}\end{equation*}
\begin{equation*}\begin{split}
\hbox{row}_{k+h+2}(M_1)=&\frac{1}{\sqrt{1-\mu^2}}\left[\sum_{i=1}^{k}\cos\otheta_i\hbox{row}_{k+h+2+i}(M_1)-\sin\otheta_i \hbox{row}_{2k+2h+3+i}(M_1)\right]\\
&+\sum_{i=1}^{h}\hbox{row}_{2k+h+2+i}(M_1),
\end{split}\end{equation*}

\begin{equation*}\begin{split}
\hbox{row}_{2k+2h+3}(M_1)=&\frac{1}{\sqrt{1-\mu^2}}\left[\sum_{i=1}^{k}\sin\otheta_i\hbox{row}_{k+h+2+i}(M_1)+\cos\otheta_i \hbox{row}_{2k+2h+3+i}(M_1)\right]\\
&+\sum_{i=1}^{h}\hbox{row}_{3k+2h+3+i}(M_1),
\end{split}\end{equation*}

\begin{equation*}\begin{split}
&\hbox{row}_{3k+3h+4}(M_1)=\sum_{i=1}^{k}\hbox{row}_{3k+3h+4+i}(M_1)\\
&\qquad+\frac{1}{\sqrt{1-\lambda^2}}\left[\sum_{i=1}^{h}\cos\htheta_i\hbox{row}_{4k+3h+4+i}(M_1)-\sin\htheta_i \hbox{row}_{5k+4h+5+i}(M_1)\right],
\end{split}\end{equation*}

\begin{equation*}\begin{split}
&\hbox{row}_{4k+4h+5}(M_1)=\sum_{i=1}^{k}\hbox{row}_{4k+4h+5+i}(M_1)\\
&\qquad+\frac{1}{\sqrt{1-\lambda^2}}\left[\sum_{i=1}^{h}\sin\htheta_i\hbox{row}_{4k+3h+4+i}(M_1)+\cos\htheta_i \hbox{row}_{5k+4h+5+i}(M_1)\right].
\end{split}\end{equation*}
From these facts we deduce that system \eqref{difficile}
is solvable only if
\begin{equation}\label{ortw}
\left[
\begin{array}{c}
r_0\\r_1\\r_2\\r_3\\r_4
\end{array}
\right]\cdot w_j=0,\qquad j=0,1,\ldots,4,
\end{equation}
where (recall definitions \eqref{cossin} and \eqref{onezero})
\begin{equation}\label{vectors11}
w_0:=
\left[
\begin{array}{c}
1\\-\oone\\-\hone\\0\\-\oone\\\hzero\\0\\\ozero\\\hzero\\0\\\ozero\\-\hone\\0\\\ozero\\\hzero
\end{array}
\right],\;\;
w_1:=\left[
\begin{array}{c}
0\\\ozero\\\hzero\\1\\-\frac{1}{\sqrt{1-\mu^2}}\ocos\\-\hone\\0\\\frac{1}{\sqrt{1-\mu^2}}\osin\\\hzero\\0\\\ozero\\\hzero\\0\\\ozero\\\hzero
\end{array}
\right],\;\;
w_2:=\left[
\begin{array}{c}
0\\\ozero\\\hzero\\0\\-\frac{1}{\sqrt{1-\mu^2}}\osin\\\hzero\\1\\-\frac{1}{\sqrt{1-\mu^2}}\ocos\\-\hone\\0\\\ozero\\\hzero\\0\\\ozero\\\hzero
\end{array}
\right],\end{equation}
\begin{equation}\label{vectors12}
w_3:=\left[
\begin{array}{c}
0\\\ozero\\\hzero\\0\\\ozero\\\hzero\\0\\\ozero\\\hzero\\1\\-\oone\\-\frac{1}{\sqrt{1-\lambda^2}}\hcos\\0\\\ozero\\\frac{1}{\sqrt{1-\lambda^2}}\hsin
\end{array}
\right],\;\;
w_4:=\left[
\begin{array}{c}
0\\\ozero\\\hzero\\0\\\ozero\\\hzero\\0\\\ozero\\\hzero\\0\\\ozero\\-\frac{1}{\sqrt{1-\lambda^2}}\hsin\\1\\-\oone\\-\frac{1}{\sqrt{1-\lambda^2}}\hcos
\end{array}
\right],
\end{equation}
which belong all to $\ker(M_1)$. On the other hand, using again that $L({\bf z}_\alpha)=0$ for every $\alpha=0,\ldots,4$ one sees that
the vectors $r_\alpha$ satisfy the following relations
\begin{equation*}\begin{split}
\hbox{row}_1(r_0)=&\sum_{j=2}^{k+1}\left[\hbox{row}_j(r_0)+\hbox{row}_j(r_1)\right]+\sum_{l=k+2}^{h+k+1}\left[\hbox{row}_l(r_0)+\hbox{row}_l(r_3)\right]\\
\hbox{row}_1(r_1)=&\frac{1}{\sqrt{1-\mu^2}}\sum_{j=2}^{k+1}\left[\cos\otheta_{j-1}\hbox{row}_j(r_1)-\sin\otheta_{j-1}\hbox{row}_j(r_2)\right]+\sum_{l=k+2}^{h+k+1}\hbox{row}_l(r_1),\\
\hbox{row}_1(r_2)=&\frac{1}{\sqrt{1-\mu^2}}\sum_{j=2}^{k+1}\left[\sin\otheta_{j-1}\hbox{row}_j(r_1)+\cos\otheta_{j-1}\hbox{row}_j(r_2)\right]+\sum_{l=k+2}^{h+k+1}\hbox{row}_l(r_2),\\
\hbox{row}_1(r_3)=&\sum_{j=2}^{k+1}\hbox{row}_j(r_3)+\frac{1}{\sqrt{1-\lambda^2}}\sum_{l=k+2}^{h+k+1}\left[\cos\htheta_{l-1}\hbox{row}_l(r_3)-\sin\htheta_{l-1}\hbox{row}_l(r_4)\right],\\
\hbox{row}_1(r_4)=&\sum_{j=2}^{k+1}\hbox{row}_j(r_4)+\frac{1}{\sqrt{1-\lambda^2}}\sum_{l=k+2}^{h+k+1}\left[\sin\htheta_{l-1}\hbox{row}_l(r_3)+\cos\htheta_{l-1}\hbox{row}_l(r_4)\right].
\end{split}\end{equation*}
These facts imply that  the orthogonality conditions \eqref{ortw} are satisfied, and thus \eqref{difficile} is solvable. The solution to \eqref{difficile} has the form
\begin{equation*}
\left[
\begin{array}{c}
c_0\\c_1\\c_2\\c_3\\c_4
\end{array}
\right]=\left[
\begin{array}{c}
0\\\tilde{c}_0\\0\\\tilde{c}_1\\0\\\tilde{c}_2\\0\\\tilde{c}_3\\0\\\tilde{c}_4
\end{array}\right]
+tw_0+sw_1+rw_2+uw_3+vw_4,
\end{equation*}
for any values of $t,s,r,u,v\in\R,$
where $\tilde{c}_\alpha :=\left[\begin{array}{c}
\oc_\alpha\\\hc_{\alpha}
\end{array}\right]$ are  solutions of
\begin{equation}\label{systemQ}
Q\left[
\begin{array}{c}
\tilde{c}_0\\\tilde{c}_1\\\tilde{c}_2\\\tilde{c}_3\\\tilde{c}_4
\end{array}\right]=
\left[
\begin{array}{c}
\tilde{r}_0\\\tilde{r}_1\\\tilde{r}_2\\\tilde{r}_3\\\tilde{r}_4
\end{array}\right], \quad \tilde r_\alpha :=\left[\begin{array}{c}
\overline r_\alpha\\\hat r_{\alpha}
\end{array}\right]
\end{equation}
Here $Q$ is the square matrix of dimension $[5 (k+h)\times (k+h)]^2$ defined as
\begin{equation*}
Q:=\left[\begin{array}{ccccc}
A & B & C & D & E\\
B^T & F & G & H & I\\
C^T & G^T & J & K & L\\
D^T & H^T & K^T & M & N\\
E^T & I^T & L^T & N^T & P
\end{array}\right],
\end{equation*}
where every submatrix of $Q$ has dimension $(k+h)\times (k+h)$ and entries of the form
$$\int_{\R^n}L(V)W,$$
where
\begin{itemize}
\item[(i)] In $A$: $V,W\in \{(\oZ_{0j})_{j=1,\ldots,k},(\hZ_{0l})_{l=1,\ldots,h}\}$.
\item[(ii)] In $B$: $V\in \{(\oZ_{0j})_{j=1,\ldots,k},(\hZ_{0l})_{l=1,\ldots,h}\}$, $W\in \{(\oZ_{1j})_{j=1,\ldots,k},(\hZ_{1l})_{l=1,\ldots,h}\}$.
\item[(iii)] In $C$: $V\in \{(\oZ_{0j})_{j=1,\ldots,k},(\hZ_{0l})_{l=1,\ldots,h}\}$, $W\in \{(\oZ_{2j})_{j=1,\ldots,k},(\hZ_{2l})_{l=1,\ldots,h}\}$.
\item[(iv)] In $D$: $V\in \{(\oZ_{0j})_{j=1,\ldots,k},(\hZ_{0l})_{l=1,\ldots,h}\}$, $W\in \{(\oZ_{3j})_{j=1,\ldots,k},(\hZ_{3l})_{l=1,\ldots,h}\}$.
\item[(v)] In $E$: $V\in \{(\oZ_{0j})_{j=1,\ldots,k},(\hZ_{0l})_{l=1,\ldots,h}\}$, $W\in \{(\oZ_{4j})_{j=1,\ldots,k},(\hZ_{4l})_{l=1,\ldots,h}\}$.
\item[(vi)] In $F$: $V,W\in \{(\oZ_{1j})_{j=1,\ldots,k},(\hZ_{1l})_{l=1,\ldots,h}\}$.
\item[(vii)] In $G$: $V\in \{(\oZ_{1j})_{j=1,\ldots,k},(\hZ_{1l})_{l=1,\ldots,h}\}$, $W\in \{(\oZ_{2j})_{j=1,\ldots,k},(\hZ_{2l})_{l=1,\ldots,h}\}$.
\item[(viii)] In $H$: $V\in \{(\oZ_{1j})_{j=1,\ldots,k},(\hZ_{1l})_{l=1,\ldots,h}\}$, $W\in \{(\oZ_{3j})_{j=1,\ldots,k},(\hZ_{3l})_{l=1,\ldots,h}\}$.
\item[(ix)] In $I$: $V\in \{(\oZ_{1j})_{j=1,\ldots,k},(\hZ_{1l})_{l=1,\ldots,h}\}$, $W\in \{(\oZ_{4j})_{j=1,\ldots,k},(\hZ_{4l})_{l=1,\ldots,h}\}$.
\item[(x)] In $J$: $V,W\in \{(\oZ_{2j})_{j=1,\ldots,k},(\hZ_{2l})_{l=1,\ldots,h}\}$.
\item[(xi)] In $K$: $V\in \{(\oZ_{2j})_{j=1,\ldots,k},(\hZ_{2l})_{l=1,\ldots,h}\}$, $W\in \{(\oZ_{3j})_{j=1,\ldots,k},(\hZ_{3l})_{l=1,\ldots,h}\}$.
\item[(xii)] In $L$: $V\in \{(\oZ_{2j})_{j=1,\ldots,k},(\hZ_{2l})_{l=1,\ldots,h}\}$, $W\in \{(\oZ_{4j})_{j=1,\ldots,k},(\hZ_{4l})_{l=1,\ldots,h}\}$.
\item[(xiii)] In $M$: $V,W\in \{(\oZ_{3j})_{j=1,\ldots,k},(\hZ_{3l})_{l=1,\ldots,h}\}$.
\item[(xiv)] In $N$: $V\in \{(\oZ_{3j})_{j=1,\ldots,k},(\hZ_{3l})_{l=1,\ldots,h}\}$, $W\in \{(\oZ_{4j})_{j=1,\ldots,k},(\hZ_{4l})_{l=1,\ldots,h}\}$.
\item[(xv)] In $P$: $V,W\in \{(\oZ_{4j})_{j=1,\ldots,k},(\hZ_{4l})_{l=1,\ldots,h}\}$.
\end{itemize}
Let us analize the structure of every matrix, where we will make use of the notation
\begin{equation}\label{defBeta}
\beta_{\alpha_1,\alpha_2}:=\int L(\oZ_{\alpha_11})\hZ_{\alpha_2 1},
\end{equation}
$$\otheta_j:=\frac{2\pi}{k}(j-1)\qquad\mbox{ and }\qquad \htheta_l:=\frac{2\pi}{h}(l-1).$$

\medskip
\noindent {\bf Matrix A.} Due to the invariance properties one can check that for $i, j=1,\ldots,k$ and $l,m=1,\ldots, h$,
\begin{equation*}\begin{split}
\int_{\R^n}L(\oZ_{0i})\oZ_{0j}&=\int_{\R^n}L(\oZ_{01})\oZ_{0,|i-j|+1},\\
\int_{\R^n}L(\hZ_{0l})\hZ_{0m}&=\int_{\R^n}L(\hZ_{01})\hZ_{0,|l-m|+1},\\
\int_{\R^n}L(\oZ_{0j})\hZ_{0l}&=\int_{\R^n}L(\hZ_{0l})\oZ_{0j}=\int_{\R^n}L(\oZ_{01})\hZ_{01}=\beta_{00}.
\end{split}\end{equation*}
Thus we can write
$$A=\left[
\begin{array}{cc}
\overline{A} & A_1:=\beta_{00} \mathds{1}_{k\times h}\\
A_2:=\beta_{00} \mathds{1}_{h\times k} & \hat{A}
\end{array}
\right],$$
where $\overline{A}$ and $\hat{A}$ are circulant matrices of dimensions $(k\times k)$ and $(h\times h)$ (see \cite{KS}).

\medskip
\noindent {\bf Matrix B.} Applying the invariance properties like in matrix A we obtain
\begin{equation*}
B=\left[
\begin{array}{cc}
\overline{B} & B_1:=\left(\int L(\oZ_{0j})\hZ_{1l}\right)_{j=1,\ldots,k,\, l=1,\ldots,h}\\
B_2:=\left(\int L(\hZ_{0l})\oZ_{1j}\right)_{l=1,\ldots,h,\, j=1,\ldots,k} & 0
\end{array}
\right],
\end{equation*}
where $\overline{B}$ is a $(k\times k)$ circulant matrix. Rotating in the $(y_1,y_2)$ and $(y_3,y_4)$ one obtains
\begin{equation*}\begin{split}
\int L(\oZ_{0j})\hZ_{1l}&=\cos \otheta_j\beta_{01}-\sin\otheta_j\beta_{02}=\cos \otheta_j\beta_{01},\\
\int L(\hZ_{0l})\oZ_{1j}&=\cos \otheta_j\beta_{10}-\sin\otheta_j\beta_{20}=\cos \otheta_j\beta_{10},
\end{split}\end{equation*}
for $j=1,\ldots,k,\, l=1,\ldots,h$, since  $\beta_{02}=\beta_{20}=0$ due to the symmetry properties. Notice that both expressions are independent of $l$.

\medskip
\noindent {\bf Matrix C.} Likewise,
\begin{equation*}
C=\left[
\begin{array}{cc}
\overline{C} & C_1:=\left(\int L(\oZ_{0j})\hZ_{2l}\right)_{j=1,\ldots,k,\, l=1,\ldots,h}\\
C_2:=\left(\int L(\hZ_{0l})\oZ_{2j}\right)_{l=1,\ldots,h,\, j=1,\ldots,k} & 0
\end{array}
\right],
\end{equation*}
whith $\overline{C}$ being a $(k\times k)$ circulant matrix and
\begin{equation*}\begin{split}
\int L(\oZ_{0j})\hZ_{2l}&=\sin \otheta_j\beta_{01},\qquad j=1,\ldots,k,\, l=1,\ldots,h,\\
\int L(\hZ_{0l})\oZ_{2j}&=\sin \otheta_j\beta_{10},\qquad j=1,\ldots,k,\, l=1,\ldots,h.
\end{split}\end{equation*}

\medskip
\noindent {\bf Matrix D.}
\begin{equation*}
D=\left[
\begin{array}{cc}
0 & D_1:=\left(\int L(\oZ_{0j})\hZ_{3l}\right)_{j=1,\ldots,k,\, l=1,\ldots,h}\\
D_2:=\left(\int L(\hZ_{0l})\oZ_{3j}\right)_{l=1,\ldots,h,\, j=1,\ldots,k} & \hat{D}
\end{array}
\right],
\end{equation*}
whith $\hat{D}$ being a $(h\times h)$ circulant matrix and
\begin{equation*}\begin{split}
\int L(\oZ_{0j})\hZ_{3l}&=\cos \htheta_l\beta_{03},\qquad j=1,\ldots,k,\, l=1,\ldots,h,\\
\int L(\hZ_{0l})\oZ_{3j}&=\cos\htheta_l\beta_{30},\qquad j=1,\ldots,k,\, l=1,\ldots,h,
\end{split}\end{equation*}
where we have used that $\beta_{04}=\beta_{40}=0$.
\medskip

\noindent {\bf Matrix E.}
\begin{equation*}
E=\left[
\begin{array}{cc}
0 & E_1:=\left(\int L(\oZ_{0j})\hZ_{4l}\right)_{j=1,\ldots,k,\, l=1,\ldots,h}\\
E_2:=\left(\int L(\hZ_{0l})\oZ_{4j}\right)_{l=1,\ldots,h,\, j=1,\ldots,k} & \hat{E}
\end{array}
\right],
\end{equation*}
whith $\hat{E}$ being a $(h\times h)$ circulant matrix and
\begin{equation*}\begin{split}
\int L(\oZ_{0j})\hZ_{4l}&=\sin \htheta_l\beta_{03},\qquad j=1,\ldots,k,\, l=1,\ldots,h,\\
\int L(\hZ_{0l})\oZ_{4j}&=\sin \htheta_l\beta_{30},\qquad j=1,\ldots,k,\, l=1,\ldots,h.
\end{split}\end{equation*}

\medskip
\noindent {\bf Matrix F.}
\begin{equation*}
F=\left[
\begin{array}{cc}
\overline{F} & F_1:=\left(\int L(\oZ_{1j})\hZ_{1l}\right)_{j=1,\ldots,k,\, l=1,\ldots,h}\\
F_2:=\left(\int L(\hZ_{1l})\oZ_{1j}\right)_{l=1,\ldots,h,\, j=1,\ldots,k} & \hat{F}
\end{array}
\right],
\end{equation*}
where $\overline{F}$ and $\hat{F}$ are $(k\times k)$ and $(h\times h)$ circulant matrices respectively and
\begin{equation*}\begin{split}
\int L(\oZ_{1j})\hZ_{1l}&=\int L(\hZ_{1l})\oZ_{1j}= \cos^2\otheta_j\beta_{11}+\sin^2\otheta_j\beta_{22},
\end{split}\end{equation*}
for $j=1,\ldots,k,\, l=1,\ldots,h,$ using that $\beta_{12}=\beta_{21}=0$.

\medskip
\noindent {\bf Matrix G.}
\begin{equation*}
G=\left[
\begin{array}{cc}
\overline{G} & G_1:=\left(\int L(\oZ_{1j})\hZ_{2l}\right)_{j=1,\ldots,k,\, l=1,\ldots,h}\\
G_2:=\left(\int L(\hZ_{1l})\oZ_{2j}\right)_{l=1,\ldots,h,\, j=1,\ldots,k} & 0
\end{array}
\right],
\end{equation*}
where $\overline{G}$ is a $(k\times k)$ circulant matrix and
\begin{equation*}\begin{split}
\int L(\oZ_{1j})\hZ_{2l}&= \int L(\hZ_{1l})\oZ_{2j}= \cos\otheta_j\sin\otheta_j\beta_{11}-\sin\otheta_j\cos\otheta_j\beta_{22},
\end{split}\end{equation*}
for $j=1,\ldots,k,\, l=1,\ldots,h.$

\medskip
\noindent {\bf Matrix H.}
\begin{equation*}
H=\left[
\begin{array}{cc}
0 & H_1:=\left(\int L(\oZ_{1j})\hZ_{3l}\right)_{j=1,\ldots,k,\, l=1,\ldots,h}\\
H_2:=\left(\int L(\hZ_{1l})\oZ_{3j}\right)_{l=1,\ldots,h,\, j=1,\ldots,k} & 0
\end{array}
\right],
\end{equation*}
where, since $\beta_{14}=\beta_{41}=\beta_{23}=\beta_{32}=\beta_{24}=\beta_{42}=0$,
\begin{equation*}\begin{split}
\int L(\oZ_{1j})\hZ_{3l}&= \cos\otheta_j\cos\htheta_l\beta_{13},\\
\int L(\hZ_{1l})\oZ_{3j}&= \cos\otheta_j\cos\htheta_l\beta_{31},
\end{split}\end{equation*}
for $j=1,\ldots,k,\, l=1,\ldots,h.$

\medskip
\noindent {\bf Matrix I.}
\begin{equation*}
I=\left[
\begin{array}{cc}
0 & I_1:=\left(\int L(\oZ_{1j})\hZ_{4l}\right)_{j=1,\ldots,k,\, l=1,\ldots,h}\\
I_2:=\left(\int L(\hZ_{1l})\oZ_{4j}\right)_{l=1,\ldots,h,\, j=1,\ldots,k} & 0
\end{array}
\right],
\end{equation*}
where
\begin{equation*}\begin{split}
\int L(\oZ_{1j})\hZ_{4l}&= \cos\otheta_j\sin\htheta_l\beta_{13},\\
\int L(\hZ_{1l})\oZ_{4j}&= \cos\otheta_j\sin\htheta_l\beta_{31},
\end{split}\end{equation*}
for $j=1,\ldots,k,\, l=1,\ldots,h.$

\medskip
\noindent {\bf Matrix J.}
\begin{equation*}
J=\left[
\begin{array}{cc}
\overline{J} & J_1:=\left(\int L(\oZ_{2j})\hZ_{2l}\right)_{j=1,\ldots,k,\, l=1,\ldots,h}\\
J_2:=\left(\int L(\hZ_{2l})\oZ_{2j}\right)_{l=1,\ldots,h,\, j=1,\ldots,k} & \hat{J}
\end{array}
\right],
\end{equation*}
where $\overline{J}$ and $\hat{J}$ are $(k\times k)$ and $(h\times h)$ circulant matrices respectively and
\begin{equation*}\begin{split}
\int L(\oZ_{2j})\hZ_{2l}&=\int L(\hZ_{2l})\oZ_{2j}= \sin^2\otheta_j\beta_{11}+\cos^2\otheta_j\beta_{22},
\end{split}\end{equation*}
for $j=1,\ldots,k,\, l=1,\ldots,h$.

\medskip
\noindent {\bf Matrix K.}
\begin{equation*}
K=\left[
\begin{array}{cc}
0 & K_1:=\left(\int L(\oZ_{2j})\hZ_{3l}\right)_{j=1,\ldots,k,\, l=1,\ldots,h}\\
K_2:=\left(\int L(\hZ_{2l})\oZ_{3j}\right)_{l=1,\ldots,h,\, j=1,\ldots,k} & 0
\end{array}
\right],
\end{equation*}
where
\begin{equation*}\begin{split}
\int L(\oZ_{2j})\hZ_{3l}&= \sin\otheta_j\cos\htheta_l\beta_{13},\\
\int L(\hZ_{2l})\oZ_{3j}&= \sin\otheta_j\cos\htheta_l\beta_{31},
\end{split}\end{equation*}
for $j=1,\ldots,k,\, l=1,\ldots,h.$

\medskip
\noindent {\bf Matrix L.}
\begin{equation*}
L=\left[
\begin{array}{cc}
0 & L_1:=\left(\int L(\oZ_{2j})\hZ_{4l}\right)_{j=1,\ldots,k,\, l=1,\ldots,h}\\
L_2:=\left(\int L(\hZ_{2l})\oZ_{4j}\right)_{l=1,\ldots,h,\, j=1,\ldots,k} & 0
\end{array}
\right],
\end{equation*}
where
\begin{equation*}\begin{split}
\int L(\oZ_{2j})\hZ_{4l}&= \sin\otheta_j\sin\htheta_l\beta_{13},\\
\int L(\hZ_{2l})\oZ_{4j}&= \sin\otheta_j\sin\htheta_l\beta_{31},
\end{split}\end{equation*}
for $j=1,\ldots,k,\, l=1,\ldots,h.$

\medskip
\noindent {\bf Matrix M.}
\begin{equation*}
M=\left[
\begin{array}{cc}
\overline{M} & M_1:=\left(\int L(\oZ_{3j})\hZ_{3l}\right)_{j=1,\ldots,k,\, l=1,\ldots,h}\\
M_2:=\left(\int L(\hZ_{3l})\oZ_{3j}\right)_{l=1,\ldots,h,\, j=1,\ldots,k} & \hat{M}
\end{array}
\right],
\end{equation*}
where $\overline{M}$ and $\hat{M}$ are $(k\times k)$ and $(h\times h)$ circulant matrices respectively and, since $\beta_{34}=\beta_{43}=0$,
\begin{equation*}\begin{split}
\int L(\oZ_{3j})\hZ_{3l}&=\int L(\hZ_{3l})\oZ_{3j}= \cos^2\htheta_l\beta_{33}+\sin^2\htheta_l\beta_{44},
\end{split}\end{equation*}
for $j=1,\ldots,k,\, l=1,\ldots,h.$

\medskip
\noindent {\bf Matrix N.}
\begin{equation*}
N=\left[
\begin{array}{cc}
0 & N_1:=\left(\int L(\oZ_{3j})\hZ_{4l}\right)_{j=1,\ldots,k,\, l=1,\ldots,h}\\
N_2:=\left(\int L(\hZ_{3l})\oZ_{4j}\right)_{l=1,\ldots,h,\, j=1,\ldots,k} & \hat{N}
\end{array}
\right],
\end{equation*}
where $\hat{N}$ is a $(h\times h)$ circulant matrix and
\begin{equation*}\begin{split}
\int L(\oZ_{3j})\hZ_{4l}&= \int L(\hZ_{3l})\oZ_{4j}= \cos\htheta_l\sin\htheta_l\beta_{33}-\sin\htheta_l\cos\htheta_l\beta_{44},
\end{split}\end{equation*}
for $j=1,\ldots,k,\, l=1,\ldots,h.$

\medskip
\noindent {\bf Matrix P.}
\begin{equation*}
P=\left[
\begin{array}{cc}
\overline{P} & P_1:=\left(\int L(\oZ_{4j})\hZ_{4l}\right)_{j=1,\ldots,k,\, l=1,\ldots,h}\\
P_2:=\left(\int L(\hZ_{4l})\oZ_{4j}\right)_{l=1,\ldots,h,\, j=1,\ldots,k} & \hat{P}
\end{array}
\right],
\end{equation*}
where $\overline{P}$ and $\hat{P}$ are $(k\times k)$ and $(h\times h)$ circulant matrices respectively and
\begin{equation*}\begin{split}
\int L(\oZ_{4j})\hZ_{4l}&=\int L(\hZ_{4l})\oZ_{4j}= \sin^2\htheta_l\beta_{33}+\cos^2\htheta_l\beta_{44},
\end{split}\end{equation*}
for $j=1,\ldots,k,\, l=1,\ldots,h.$

\medskip
Notice that
$$\hat{F}=\hat{J}\qquad\mbox{and}\qquad\overline{M}=\overline{P}.$$

Henceforth, system \eqref{systemQ} can be decomposed in two different systems in the following way,
\begin{equation}\label{system1}
\left[
\begin{array}{ccccc}
\overline{A} &  \overline{B} &  \overline{C} & 0 & 0  \\
\overline{B}^T & \overline{F} & \overline{G} & 0 & 0\\
\overline{C}^T & \overline{G}^T & \overline{J} & 0 & 0 \\
0 & 0 & 0 & \overline{M} & 0 \\
0 & 0 & 0 & 0 & \overline{P} \\
\end{array}\right]
%\cdot
\left[
\begin{array}{c}
\overline{c}_0\\ \overline{c}_1\\ \overline{c}_2\\ \overline{c}_3\\ \overline{c}_4
\end{array}\right]
=
\left[
\begin{array}{c}
\overline{r}_0\\ \overline{r}_1\\ \overline{r}_2\\ \overline{r}_3\\ \overline{r}_4
\end{array}\right]
-
\left[
\begin{array}{ccccc}
A_1 &  B_1 &  C_1 & D_1 & E_1  \\
B_2^T & F_1 & G_1 & H_1 & I_1\\
C_2^T & G_2^T & J_1 & K_1 & L_1 \\
D_2^T & H_2^T & K_2^T & M_1 & N_1 \\
E_2^T & I_2^T & L_2^T & N_2^T & P_1 \\
\end{array}\right]
%\cdot
\left[
\begin{array}{c}
\hat{c}_0\\ \hat{c}_1\\ \hat{c}_2\\ \hat{c}_3\\ \hat{c}_4
\end{array}\right],
\end{equation}
\begin{equation}\label{system2}
\left[
\begin{array}{ccccc}
\hat{A} & 0 & 0 & \hat{D} & \hat{E}\\
0 & \hat{F} & 0 & 0 &0 \\
0 & 0 & \hat{J} & 0 & 0 \\
\hat{D}^T & 0 & 0  & \hat{M} & \hat{N}\\
\hat{E}^T & 0 & 0 & \hat{N}^T & \hat{P}
\end{array}\right]
\left[
\begin{array}{c}
\hat{c}_0\\ \hat{c}_1\\ \hat{c}_2\\ \hat{c}_3\\ \hat{c}_4
\end{array}\right]
=
\left[
\begin{array}{c}
\hat{r}_0\\ \hat{r}_1\\ \hat{r}_2\\ \hat{r}_3\\ \hat{r}_4
\end{array}\right]
-
\left[
\begin{array}{ccccc}
A_2 &  B_2 &  C_2 & D_2 & E_2  \\
B_1^T & F_2 & G_2 & H_2 & I_2\\
C_1^T & G_1^T & J_2 & K_2 & L_2 \\
D_1^T & H_1^T & K_1^T & M_2 & N_2 \\
E_1^T & I_1^T & L_1^T & N_1^T & P_2 \\
\end{array}\right]
%\cdot
\left[
\begin{array}{c}
\overline{c}_0\\ \overline{c}_1\\ \overline{c}_2\\ \overline{c}_3\\ \overline{c}_4
\end{array}\right].
\end{equation}
By \cite[Proposition 5.1]{MW} we know that the systems
\begin{equation*}
\left[
\begin{array}{ccccc}
\overline{A} &  \overline{B} &  \overline{C} & 0 & 0  \\
\overline{B}^T & \overline{F} & \overline{G} & 0 & 0\\
\overline{C}^T & \overline{G}^T & \overline{J} & 0 & 0 \\
0 & 0 & 0 & \overline{M} & 0 \\
0 & 0 & 0 & 0 & \overline{P} \\
\end{array}\right]
%\cdot
\left[
\begin{array}{c}
\overline{c}_0\\ \overline{c}_1\\ \overline{c}_2\\ \overline{c}_3\\ \overline{c}_4
\end{array}\right]
=
\left[
\begin{array}{c}
\overline{s}_0\\ \overline{s}_1\\ \overline{s}_2\\ \overline{s}_3\\ \overline{s}_4
\end{array}\right],
\end{equation*}
\begin{equation*}
\left[
\begin{array}{ccccc}
\hat{A} & 0 & 0 & \hat{D} & \hat{E}\\
0 & \hat{F} & 0 & 0 &0 \\
0 & 0 & \hat{J} & 0 & 0 \\
\hat{D}^T & 0 & 0  & \hat{M} & \hat{N}\\
\hat{E}^T & 0 & 0 &\hat{N}^T & \hat{P}
\end{array}\right]
\left[
\begin{array}{c}
\hat{c}_0\\ \hat{c}_1\\ \hat{c}_2\\ \hat{c}_3\\ \hat{c}_4
\end{array}\right]
=
\left[
\begin{array}{c}
\hat{s}_0\\ \hat{s}_1\\ \hat{s}_2\\ \hat{s}_3\\ \hat{s}_4
\end{array}\right]
\end{equation*}
are solvable if the orthogonality conditions
\begin{equation}\begin{split}\label{GeneralConditions}
&\os_2\cdot\oone=(\os_0+\os_1)\cdot\ocos=(\os_0+\os_1)\cdot\osin = 0,\\
&\os_3\cdot\ocos=\os_3\cdot\osin=0,\\
&\os_4\cdot\ocos=\os_4\cdot\osin=0,\\
&\hs_4\cdot\hone=(\hs_0+\hs_3)\cdot\hcos=(\hs_0+\hs_3)\cdot\hsin = 0,\\
&\hs_1\cdot\hcos=\hs_1\cdot\hsin=0,\\
&\hs_2\cdot\hcos=\hs_2\cdot\hsin=0,
\end{split}\end{equation}
hold. Moreover,
\begin{equation}\label{formCalpha}\begin{split}
&\left[\begin{array}{c}
\oc_0\\\oc_1\\\oc_2
\end{array}\right]=
\left[\begin{array}{c}
\ow_0\\\ow_1\\\ow_2
\end{array}\right]+\ot_1
\left[\begin{array}{c}
0\\0\\\oone
\end{array}\right]+\ot_2
\left[\begin{array}{c}
\ocos\\-\ocos\\0
\end{array}\right]+\ot_3
\left[\begin{array}{c}
\osin\\-\osin\\0
\end{array}\right],\;\;\;\forall \,\ot_1,\ot_2,\ot_3\in\R,\\
&\oc_3=\ow_3+\ot_4\ocos+\ot_5\osin\;\;\forall\,\ot_4,\ot_5\in\R,\\
&\oc_4=\ow_4+\ot_6\ocos+\ot_7\osin,\;\;\forall\,\ot_6,\ot_7\in\R,\\
&\hc_1=\hw_1+\hatt_1\hcos+\hatt_2\hsin\;\;\forall\,\hatt_1,\hatt_2\in\R,\\
&\hc_2=\hw_2+\hatt_3\hcos+\hatt_4\hsin,\;\;\forall\,\hatt_3,\hatt_4\in\R,\\
&\left[\begin{array}{c}
\hc_0\\\hc_3\\\hc_4
\end{array}\right]=
\left[\begin{array}{c}
\hw_0\\\hw_3\\\hw_4
\end{array}\right]+\hatt_5
\left[\begin{array}{c}
0\\0\\\hone
\end{array}\right]+\hatt_6
\left[\begin{array}{c}
\hcos\\-\hcos\\0
\end{array}\right]+\hatt_7
\left[\begin{array}{c}
\hsin\\-\hsin\\0
\end{array}\right],\;\;\;\forall \,\hatt_5,\hatt_6,\hatt_7\in\R,
\end{split}\end{equation}
with $\left[\begin{array}{c}\ow_0\\\ldots \\\ow_4\end{array}\right],\,\left[\begin{array}{c}\hw_0\\\ldots \\\hw_4\end{array}\right]$ fixed vectors such that
\begin{equation}\label{boundsW}
\|\left[\begin{array}{c}\ow_0\\\ldots\\ \ow_4\end{array}\right]\|\leq\frac{C}{k^n\mu^{n-2}}\|\left[\begin{array}{c}\os_0\\\ldots\\ \os_4\end{array}\right]\|,\;\;\;\;
\|\left[\begin{array}{c}\hw_0\\\ldots\\ \hw_4\end{array}\right]\|\leq\frac{C}{h^n\lambda^{n-2}}\|\left[\begin{array}{c}\hs_0\\\ldots\\ \hs_4\end{array}\right]\|.
\end{equation}
%\textcolor{red}{Thus we can assume that $\oc_0,\ldots,\oc_4,\hc_1,\ldots,\hc_4$ satisfy \eqref{GeneralConditions}}.
We will prove that \eqref{system1} and \eqref{system2} have a solution $\left[\begin{array}{c}\oc_0\\\hc_0\\\ldots\\\oc_4\\\hc_4\end{array}\right]$ in the space
$$X:=\left\{\left[\begin{array}{c}\oc_0\\\hc_0\\\ldots\\\oc_4\\\hc_4\end{array}\right]:\,
\begin{array}{cc}
\oc_3\cdot\ocos=\oc_3\cdot\osin=0, & \oc_4\cdot\ocos=\oc_4\cdot\osin=0,\\
\hc_1\cdot\hcos=\hc_1\cdot\hsin=0, & \hc_2\cdot\hcos=\hc_2\cdot\hsin=0.
\end{array}\right\}.$$
We need the following auxiliar result, whose proof follows straightforward using the same argument as in Lemma \ref{proj1} so we skip it.
\begin{lemma}\label{AuxLemma}
Let $h, g$ be functions in $\R^n$ such that $h(y)=h(e^{\frac{2\pi}{k}(j-1)}\overline{y},\hat{y},y')$ for all $j=1,\ldots, k$ and $g(y)=g(\overline{y},e^{\frac{2\pi}{h}(l-1)}\hat{y},y')$ for all $l=1,\ldots, h$. Then,
$$\int_{\R^n}\hat{Z}_{1l}(y)h(y)\,dy=\int_{\R^n}\hat{Z}_{2l}(y)h(y)\,dy=0,\; \forall \,l=1,\ldots, h,$$
$$\int_{\R^n}\overline{Z}_{3j}(y)g(y)\,dy=\int_{\R^n}\overline{Z}_{4j}(y)g(y)\,dy=0,\; \forall \,j=1,\ldots, k.$$
\end{lemma}
Let us focus on \eqref{system1}. By \eqref{eqPi},
\begin{equation*}\begin{split}
\overline{r}_2\cdot \oone&=\sum_{j=1}^k\int L(\varphi^\perp)\oZ_{2j}\\
&=\sum_{j=1}^k\left[\int L(\sum_{\alpha=0}^n\sum_{i=1}^k\oZ_{\alpha i})\oZ_{2j}+\int L(\sum_{\alpha=0}^n\sum_{l=1}^h\hZ_{\alpha l})\oZ_{2j}\right].
\end{split}\end{equation*}
Notice that the second term vanishes due to Lemma \ref{AuxLemma} since, by the symmetry properties of the functions,
$$\sum_{j=1}^k\int L(\sum_{\alpha=0}^n\sum_{l=1}^h\hZ_{\alpha l})\oZ_{2j}=\sum_{j=1}^k\int L(\sum_{l=1}^h\hZ_{2 l})\oZ_{2j}=\sum_{l=1}^h\int \hZ_{2l}L(\sum_{j=1}^k \oZ_{2j}),$$
and $L(\sum_{j=1}^k \oZ_{2j})$ is invariant under rotation of angle $\frac{2\pi}{k}(j-1)$ in the $(y_1,y_2)$ plane. Therefore,
\begin{equation}\begin{split}\label{ortOR2}
\overline{r}_2\cdot \oone&=\left(\int L(\sum_{\alpha=0}^n\sum_{i=1}^k \oZ_{\alpha i})\oZ_{21}\right) \sum_{j=1}^k \sin\otheta_j+\left(\int L(\sum_{\alpha=0}^k\sum_{i=1}^k \oZ_{\alpha i})\oZ_{22}\right) \sum_{j=1}^k \cos\otheta_j=0.
\end{split}\end{equation}
On the other hand, as a consequence of \cite[Lemma 6.1]{MW},
\begin{equation}
(\overline{r}_0+\overline{r}_1)\cdot\ocos=(\overline{r}_0+\overline{r}_1)\cdot\osin =0.
\end{equation}
Using the invariances under rotation in the planes $(y_1,y_2)$ and $(y_3,y_4)$ and Lemma \ref{AuxLemma} we get
\begin{equation*}\begin{split}
\int L(\varphi^\perp)\oZ_{3j}&=\int L(\sum_{\alpha=1}^n\sum_{i=1}^k\oZ_{\alpha i})\oZ_{3j}+\int L(\sum_{\alpha=1}^n\sum_{l=1}^h\hZ_{\alpha l})\oZ_{3j}\\
&=\int L(\sum_{\alpha=1}^n\sum_{i=1}^k\oZ_{\alpha i})\oZ_{31},
\end{split}\end{equation*}
and thus
\begin{equation}\label{ortOR3}
\overline{r}_3\cdot\ocos= \int L(\sum_{\alpha=1}^n\sum_{i=1}^k\oZ_{\alpha i})\oZ_{31}\left(\sum_{j=1}^k\cos{\theta_j}\right)=0.
\end{equation}
Analogously,
\begin{equation}\label{ortOR4}
\overline{r}_3\cdot\osin= \overline{r}_4\cdot\ocos= \overline{r}_4\cdot\osin= 0.
\end{equation}
Let us now check the last term in \eqref{system1}. We expect
\begin{eqnarray}
&&(C_2^T[\hc_0]+G_2^T[\hc_1]+J_1[\hc_2]+K_1[\hc_3]+L_1[\hc_4])\cdot\oone =0,\label{cond1}\\
&&(A_1[\hc_0]+B_1[\hc_1]+C_1[\hc_2]+D_1[\hc_3]+E_1[\hc_4]\nonumber\\
&&\;\; +B_2^T[\hc_0]+F_1[\hc_1]+G_1[\hc_2]+H_1[\hc_3]+I_1[\hc_4])\cdot \ocos =0,\label{cond2}\\
&&(A_1[\hc_0]+B_1[\hc_1]+C_1[\hc_2]+D_1[\hc_3]+E_1[\hc_4]\nonumber\\
&&\;\; +B_2^T[\hc_0]+F_1[\hc_1]+G_1[\hc_2]+H_1[\hc_3]+I_1[\hc_4])\cdot \osin =0,\label{cond3}\\
&&(D_2^T[\hc_0]+H_2^T[\hc_1]+K_2^T[\hc_2]+M_1[\hc_3]+N_1[\hc_4])\cdot\ocos =0,\label{cond4}\\
&&(D_2^T[\hc_0]+H_2^T[\hc_1]+K_2^T[\hc_2]+M_1[\hc_3]+N_1[\hc_4])\cdot\osin =0,\label{cond5}\\
&&(E_2^T[\hc_0]+I_2^T[\hc_1]+L_2^T[\hc_2]+N_2^T[\hc_3]+P_1[\hc_4])\cdot\ocos =0,\label{cond6}\\
&&(E_2^T[\hc_0]+I_2^T[\hc_1]+L_2^T[\hc_2]+N_2^T[\hc_3]+P_1[\hc_4])\cdot\osin =0,\label{cond7}
\end{eqnarray}
where  $\hc_0,\ldots,\hc_4$ satisfy
\begin{equation}\begin{split}\label{GeneralConditionsHat}
&\hc_1\cdot\hcos=\hc_1\cdot\hsin=0,\\
&\hc_2\cdot\hcos=\hc_2\cdot\hsin=0.
\end{split}\end{equation}
Notice that, since $C_2$ and $G_2$ have all their rows identical,
$$C_2^T[\hc_0]\cdot\oone = \beta_{10}\left(\sum_{l=1}^h \hc_{0l}\right)\left(\sum_{j=1}^k \sin\otheta_j\right) =0,$$
$$G_2^T[\hc_1]\cdot\oone= (\beta_{11}-\beta_{22})\left(\sum_{l=1}^h \hc_{1l}\right)\left(\sum_{j=1}^k \sin\otheta_j\cos\otheta_j\right) =0.$$
Likewise, using the definition of $J_1$ and Lemma \ref{AuxLemma},
$$J_1[\hc_2]\cdot\oone = \left(\sum_{l=1}^h \hc_{2l}\right) \int L\left(\sum_{j=1}^k \oZ_{2j}\right)\hZ_{11}=0,$$
since $L\left(\sum_{j=1}^k \oZ_{2j}\right)$ is invariant under rotation of angle $\theta_j$, and
$$K_1[\hc_3]\cdot\oone=\beta_{13}(\hc_3\cdot\hcos)\left(\sum_{j=1}^k\sin\otheta_j\right)=0,$$
$$L_1[\hc_4]\cdot\oone=\beta_{13}(\hc_4\cdot\hsin)\left(\sum_{j=1}^k\sin\otheta_j\right)=0.$$
Thus, \eqref{cond1} follows. Furthermore, again by \cite[Lemma 6.1]{MW},
$$(A_1+B_2^T)_{jl}=\int L(\hZ_{0l})(\oZ_{0j}+\oZ_{1j})=0,$$
$$(B_1+F_1)_{jl}=\int L(\hZ_{1l})(\oZ_{0j}+\oZ_{1j})=0,$$
$$(C_1+G_1)_{jl}=\int L(\hZ_{2l})(\oZ_{0j}+\oZ_{1j})=0,$$
and thus $(A_1+B_2^T)[\hc_0]=(B_1+F_1)[\hc_1]=(C_1+G_1)[\hc_2]=0$. Likewise,
$$(D_1+H_1)_{jl}=\int L(\hZ_{3l})(\oZ_{0j}+\oZ_{1j})=0,$$
$$(E_1+I_1)_{jl}=\int L(\hZ_{4l})(\oZ_{0j}+\oZ_{1j})=0,$$
and therefore $(D_1+H_1)[\hc_3]=(E_1+I_1)[\hc_4]=0$.
%On the other hand,
%$$D_1[\hc_3]\cdot\ocos=\beta_{03}(\hc_3\cdot\hcos)\left(\sum_{j=1}^k\cos\otheta_j\right)=0,$$
%$$E_1[\hc_4]\cdot\ocos=\beta_{03}(\hc_4\cdot\hsin)\left(\sum_{j=1}^k\cos\otheta_j\right)=0,$$
%\textcolor{red}{and, applying \eqref{ortHat},
%$$H_1[\hc_3]\cdot\ocos=\beta_{13}(\hc_3\cdot\hcos)\left(\sum_{j=1}^k\cos^2\otheta_j\right)=0,$$
%$$I_1[\hc_4]\cdot\ocos=\beta_{13}(\hc_3\cdot\hsin)\left(\sum_{j=1}^k\cos^2\otheta_j\right)=0,$$
%that concludes the proof of \eqref{cond2}.}
\eqref{cond3} can be analogously proved. Furthermore,
$$D_2^T[\hc_0]\cdot\ocos=\beta_{31}(\hc_0\cdot\hcos)\left(\sum_{j=1}^k\cos\otheta_j\right)=0,$$
$$M_1[\hc_3]\cdot\ocos=\left[\beta_{33}\left(\sum_{l=1}^h\hc_{3l}\cos^2\htheta_l\right)+\beta_{44}\left(\sum_{l=1}^h\hc_{3l}\sin^2\htheta_l\right)\right]\left(\sum_{j=1}^k \cos\otheta_j\right)=0,$$
$$N_1[\hc_4]\cdot\ocos=(\beta_{33}-\beta_{44})\left(\sum_{l=1}^h\hc_{4l}\sin\htheta_l\cos\htheta_l\right)\left(\sum_{j=1}^k \cos\otheta_j\right)=0,$$
and, due to \eqref{GeneralConditionsHat},
$$H_2^T[\hc_1]\cdot\ocos=\beta_{31}(\hc_1\cdot\hcos)\left(\sum_{j=1}^k\cos^2\otheta_j\right)=0,$$
$$K_2^T[\hc_2]\cdot\ocos=\beta_{31}(\hc_2\cdot\hcos)\left(\sum_{j=1}^k\sin\otheta_j\cos\otheta_j\right)=0,$$
so \eqref{cond4} holds. Identities \eqref{cond5}-\eqref{cond7} can be obtained in a similar way, and thus \eqref{system1} is solvable. An analogous reasoning proves the solvability of \eqref{system2}.
Thus, the systems \eqref{system1} and \eqref{system2} have a solution in $X$ with the form \eqref{formCalpha}, where $[\ow_0,\ldots,\ow_4,\hw_0,\ldots,\hw_4]$ satisfies \eqref{GeneralConditions}.
It can be checked that
\begin{equation}\begin{split}\label{boundsBeta}
|\beta_{\alpha \alpha}|\leq Ck^{-2n+4},\;\alpha=0,2,4,&\qquad |\beta_{\alpha \alpha}|\leq Ck^{-2n+6},\;\alpha=1,3,\\
|\beta_{\alpha_1 \alpha_2}|\leq Ck^{-2n+6},\;\alpha_1,&\;\alpha_2=0,1,3,\;\alpha_1\neq\alpha_2.
\end{split}\end{equation}
where $\beta_{\alpha_1,\alpha_2}$ was defined in \eqref{defBeta}, and henceforth, by \eqref{boundsW} and recalling that $h=O(k)$,
$$\|\ow_\alpha\|\leq Ck^{n-4}\|\overline{r}_\alpha\|,\qquad \|\hw_\alpha\|\leq Ck^{n-4}\|\hat{r}_\alpha\|,\qquad \alpha=0, 1, \ldots, 4.$$
As it was done in the case $\alpha\geq 5$, we will solve the systems by means of a fixed point argument. If we denote
$$\overline{M}_1:=\left[
\begin{array}{ccccc}
\overline{A} &  \overline{B} &  \overline{C} & 0 & 0  \\
\overline{B}^T & \overline{F} & \overline{G} & 0 & 0\\
\overline{C}^T & \overline{G}^T & \overline{J} & 0 & 0 \\
0 & 0 & 0 & \overline{M} & 0 \\
0 & 0 & 0 & 0 & \overline{P} \\
\end{array}\right],\;\overline{\overline{M}}_1:=\left[
\begin{array}{ccccc}
A_1 &  B_1 &  C_1 & D_1 & E_1  \\
B_2^T & F_1 & G_1 & H_1 & I_1\\
C_2^T & G_2^T & J_1 & K_1 & L_1 \\
D_2^T & H_2^T & K_2^T & M_1 & N_1 \\
E_2^T & I_2^T & L_2^T & N_2^T & P_1 \\
\end{array}\right],$$
$$\hat{M}_1:=\left[
\begin{array}{ccccc}
\hat{A} & 0 & 0 & \hat{D} & \hat{E}\\
0 & \hat{F} & 0 & 0 &0 \\
0 & 0 & \hat{J} & 0 & 0 \\
\hat{D}^T & 0 & 0  & \hat{M} & \hat{N}\\
\hat{E}^T & 0 & 0 & \hat{N}^T & \hat{P}
\end{array}\right],\;\hat{\hat{M}}_1:=\left[
\begin{array}{ccccc}
A_2 &  B_2 &  C_2 & D_2 & E_2  \\
B_1^T & F_2 & G_2 & H_2 & I_2\\
C_1^T & G_1^T & J_2 & K_2 & L_2 \\
D_1^T & H_1^T & K_1^T & M_2 & N_2 \\
E_1^T & I_1^T & L_1^T & N_1^T & P_2 \\
\end{array}\right],$$
then $[\ow_0,\ldots,\ow_4,\hw_0,\ldots,\hw_4]$ is a solution of \eqref{system1}-\eqref{system2} if and only if it is a fixed point of
$$F\left[
\begin{array}{c}
\ow_0\\ \ldots\\ \ow_4\\ \hw_0\\\ldots\\ \hw_4
\end{array}\right]:=S^{-1}\left(
\left[
\begin{array}{c}
\overline{r}_0\\ \overline{r}_1\\ \overline{r}_2\\ \overline{r}_3\\ \overline{r}_4
\end{array}\right]
-
\overline{\overline{M}}_1
\left[
\begin{array}{c}
\hat{w}_0\\ \hat{w}_1\\ \hat{w}_2\\ \hat{w}_3\\ \hat{w}_4
\end{array}\right],
\left[
\begin{array}{c}
\hat{r}_0\\ \hat{r}_1\\ \hat{r}_2\\ \hat{r}_3\\ \hat{r}_4
\end{array}\right]
-
\hat{\hat{M}}_1
\left[
\begin{array}{c}
\overline{w}_0\\ \overline{w}_1\\ \overline{w}_2\\ \overline{w}_3\\ \overline{w}_4
\end{array}\right]
\right),$$
where
$$S\left(\left[
\begin{array}{c}
\overline{w}_0\\ \overline{w}_1\\ \overline{w}_2\\ \overline{w}_3\\ \overline{w}_4
\end{array}\right], \left[
\begin{array}{c}
\hat{w}_0\\ \hat{w}_1\\ \hat{w}_2\\ \hat{w}_3\\ \hat{w}_4
\end{array}\right] \right):=\left(\overline{M}_1\left[
\begin{array}{c}
\overline{w}_0\\ \overline{w}_1\\ \overline{w}_2\\ \overline{w}_3\\ \overline{w}_4
\end{array}\right], \hat{M}_1\left[
\begin{array}{c}
\hat{w}_0\\ \hat{w}_1\\ \hat{w}_2\\ \hat{w}_3\\ \hat{w}_4
\end{array}\right]\right).$$
Notice that $S$ is a linear map which is invertible for vectors satisfying the orthogonality conditions \eqref{GeneralConditions}. Let
\begin{equation*}\begin{split}
B_r:=\{\left[\begin{array}{c}
\overline{w}_0\\ \ldots \\ \overline{w}_4
\end{array}\right], &
\left[
\begin{array}{c}
\hat{w}_0\\ \ldots\\ \hat{w}_4
\end{array}\right]\in K : \|\left[\begin{array}{c}
\overline{w}_0\\ \ldots \\ \overline{w}_4
\end{array}\right]\|\leq rk^{n-4}\|\left[\begin{array}{c}
\overline{r}_0\\ \ldots \\ \overline{r}_4
\end{array}\right]\|,\\
&\|\left[
\begin{array}{c}
\hat{w}_0\\ \ldots\\ \hat{w}_4
\end{array}\right]\|\leq rk^{n-4}\|\left[
\begin{array}{c}
\hat{r}_0\\ \ldots\\ \hat{r}_4\end{array}\right]\},
\end{split}\end{equation*}
for some fixed $r$ large, where
$$K:=\{\left[\begin{array}{c}
\overline{w}_0\\ \ldots \\ \overline{w}_4
\end{array}\right]\in\R^{5\times k},
\left[
\begin{array}{c}
\hat{w}_0\\ \ldots\\ \hat{w}_4
\end{array}\right]\in \R^{5\times h}\mbox{ satisfying }\eqref{GeneralConditions}.\}.$$
Thanks to the particular form of the matrices $\overline{\overline{M}}_1$, $\hat{\hat{M}}$ (all their submatrices are combinations of sinus and cosinus multiplied by a term $\beta_{\alpha_1,\alpha_2}$) and \eqref{boundsBeta} it can be checked that $F$ is a contraction mapping that sends $B_r$ into $B_r$. This finishes the proof of the existence of a solution to \eqref{system1}-\eqref{system2} satisfying \eqref{GeneralConditions}.

Define the vectors
\begin{equation}\label{vectors21}
\overline{u}_0:=
\left[
\begin{array}{c}
0\\ \ozero\\ \hzero\\
0\\ \ozero\\ \hzero\\
0\\ \oone\\ \hzero\\
0\\ \ozero\\ \hzero\\
0\\ \ozero\\ \hzero
\end{array}\right],\,
\overline{u}_1:=
\left[
\begin{array}{c}
0\\ \ocos\\ \hzero\\
0\\ -\ocos\\ \hzero\\
0\\ \ozero\\ \hzero\\
0\\ \ozero\\ \hzero\\
0\\ \ozero\\ \hzero
\end{array}\right],\,
\overline{u}_2:=
\left[
\begin{array}{c}
0\\ \osin\\ \hzero\\
0\\ -\osin\\ \hzero\\
0\\ \ozero\\ \hzero\\
0\\ \ozero\\ \hzero\\
0\\ \ozero\\ \hzero
\end{array}\right],\,
%\overline{u}_3:=
%\left[
%\begin{array}{c}
%0\\ \ozero\\ \hzero\\
%0\\ \ozero\\ \hzero\\
%0\\ \ozero\\ \hzero\\
%0\\ \ocos\\ \hzero\\
%0\\ \ozero\\ \hzero
%\end{array}\right],
\end{equation}
%\begin{equation}\label{vectors22}
%\overline{u}_4:=
%\left[
%\begin{array}{c}
%0\\ \ozero\\ \hzero\\
%0\\ \ozero\\ \hzero\\
%0\\ \ozero\\ \hzero\\
%0\\ \osin\\ \hzero\\
%0\\ \ozero\\ \hzero
%\end{array}\right],\,
%\overline{u}_5:=
%\left[
%\begin{array}{c}
%0\\ \ozero\\ \hzero\\
%0\\ \ozero\\ \hzero\\
%0\\ \ozero\\ \hzero\\
%0\\ \ozero\\ \hzero\\
%0\\ \ocos\\ \hzero
%\end{array}\right],\,
%\overline{u}_6:=
%\left[
%\begin{array}{c}
%0\\ \ozero\\ \hzero\\
%0\\ \ozero\\ \hzero\\
%0\\ \ozero\\ \hzero\\
%0\\ \ozero\\ \hzero\\
%0\\ \osin\\ \hzero
%\end{array}\right],\,
%\overline{v}_0:=
%\left[
%\begin{array}{c}
%0\\ \ozero\\ \hzero\\
%0\\ \ozero\\ \hzero\\
%0\\ \ozero\\ \hzero\\
%0\\ \ozero\\ \hzero\\
%0\\ \ozero\\ \hone
%\end{array}\right],
%\end{equation}
\begin{equation}\label{vectors23}
\hat{u}_0:=
\left[
\begin{array}{c}
0\\ \ozero\\ \hzero\\
0\\ \ozero\\ \hzero\\
0\\ \ozero\\ \hzero\\
0\\ \ozero\\ \hzero\\
0\\ \ozero\\ \hone
\end{array}\right],\,
\hat{u}_1:=
\left[
\begin{array}{c}
0\\ \ozero\\ \hcos\\
0\\ \ozero\\ \hzero\\
0\\ \ozero\\ \hzero\\
0\\ \ozero\\ -\hcos\\
0\\ \ozero\\ \hzero
\end{array}\right],\,
\hat{u}_2:=
\left[
\begin{array}{c}
0\\ \ozero\\ \hsin\\
0\\ \ozero\\ \hzero\\
0\\ \ozero\\ \hzero\\
0\\ \ozero\\ -\hsin\\
0\\ \ozero\\ \hzero
\end{array}\right].\,
%\overline{v}_3:=
%\left[
%\begin{array}{c}
%0\\ \ozero\\ \hzero\\
%0\\ \ozero\\ \hcos\\
%0\\ \ozero\\ \hzero\\
%0\\ \ozero\\ \hzero\\
%0\\ \ozero\\ \hzero
%\end{array}\right],\,
%\overline{v}_4:=
%\left[
%\begin{array}{c}
%0\\ \ozero\\ \hzero\\
%0\\ \ozero\\ \hsin\\
%0\\ \ozero\\ \hzero\\
%0\\ \ozero\\ \hzero\\
%0\\ \ozero\\ \hzero
%\end{array}\right].
\end{equation}
%\begin{equation}\label{vectors24}
%\overline{v}_5:=
%\left[
%\begin{array}{c}
%0\\ \ozero\\ \hzero\\
%0\\ \ozero\\ \hzero\\
%0\\ \ozero\\ \hcos\\
%0\\ \ozero\\ \hzero\\
%0\\ \ozero\\ \hzero
%\end{array}\right],\,
%\overline{v}_6:=
%\left[
%\begin{array}{c}
%0\\ \ozero\\ \hzero\\
%0\\ \ozero\\ \hzero\\
%0\\ \ozero\\ \hsin\\
%0\\ \ozero\\ \hzero\\
%0\\ \ozero\\ \hzero
%\end{array}\right],
%\end{equation}
We can summarize this section in the following lemma.
\begin{lemma}\label{dos}
System \eqref{difficile} is solvable, and the solution has the form
\begin{equation*}\begin{split}
\left[
\begin{array}{c}
\tilde{c}_0\\\ldots\\\tilde{c}_4
\end{array}
\right]=&
\,w+t_0w_0+t_1w_1+t_2w_2+t_3w_3+t_4w_4\\
&+
\overline{t}_0 \overline{u}_0+\overline{t}_1 \overline{u}_1+\overline{t}_2 \overline{u}_2+
\hat{t}_0 \hat{u}_0+\hat{t}_1 \hat{u}_1+\hat{t}_2 \hat{u}_2,
\end{split}\end{equation*}
where
$$w:=\left[
\begin{array}{ccccccccccccccc}
0 &\ow_0 &\hw_0 & 0 & \ow_1&\hw_1& 0 & \ow_2&\hw_2 & 0 & \ow_3&\hw_3 & 0 & \ow_4&\hw_4
\end{array}
\right]^T$$
($\ow_0,\ldots, \ow_4,\hw_0,\ldots,\hw_4$ being the unique solution to the system \eqref{system1}-\eqref{system2} satisfying \eqref{GeneralConditions}and \eqref{boundsW}), $t_0,\ldots,t_4$, $\overline{t}_0,\overline{t}_1,\overline{t}_2$, $\hat{t}_0,\hat{t}_1,\hat{t}_2$ $\in \R$ and $\tilde{c}_0,\ldots,\tilde{c}_4$, $w_0,\ldots,w_4$, $\overline{u}_0,\overline{u}_1,\overline{u}_2$, $\hat{u}_0,\hat{u}_1,\hat{u}_2$ are defined in \eqref{vectorsC}, \eqref{vectors11}, \eqref{vectors12}, \eqref{vectors21} and \eqref{vectors23}.
\end{lemma}

%\section{Proof of Proposition \ref{nonso}}\label{nnn}

\section{Proof of Proposition \ref{systC}}\label{prop31}

\medskip
With this in mind, we observe that
 Proposition \ref{systC} is a consequence of the following estimates: if $k, \, h \to \infty$
\begin{equation}\begin{split}\label{int1}
\int |u|^{p-1}Z_{\alpha 0} Z_{\beta 0} & = \int U^{p-1}Z_{00}^2 +O(\mu^{\frac{n-2}{2}}+\lambda^{\frac{n-2}{2}})\qquad \mbox{if } \alpha=\beta=0,\\
 & = \int U^{p-1}Z_{10}^2 +O(\mu^{\frac{n-2}{2}}+\lambda^{\frac{n-2}{2}})\qquad \mbox{if } \alpha=\beta\neq 0,\\
& = O(\mu^{\frac{n-2}{2}}+\lambda^{\frac{n-2}{2}})\qquad\mbox{otherwise},
\end{split}\end{equation}

\begin{equation}\begin{split}\label{int2}
\int |u|^{p-1}\oZ_{\alpha i} \oZ_{\beta j} & = \int U^{p-1}Z_{00}^2 +O(\mu^{\frac{n-2}{2}})\qquad \mbox{if } \alpha=\beta=0,\, i=j,\\
& = \int U^{p-1}Z_{10}^2 +O(\mu^{\frac{n-2}{2}})\qquad \mbox{if } \alpha=\beta\neq 0,\, i=j\\
& = O(\mu^{\frac{n-2}{2}}) \qquad\mbox{otherwise},
\end{split}\end{equation}

\begin{equation}\begin{split}\label{int3}
\int |u|^{p-1}\hZ_{\alpha l} \hZ_{\beta m} & = \int U^{p-1}Z_{00}^2 +O(\lambda^{\frac{n-2}{2}})\qquad \mbox{if } \alpha=\beta=0,\,l=m,\\
& = \int U^{p-1}Z_{10}^2 +O(\lambda^{\frac{n-2}{2}})\qquad \mbox{if } \alpha=\beta\neq 0,\,l=m,\\
& = O(\lambda^{\frac{n-2}{2}}) \qquad\mbox{otherwise},
\end{split}\end{equation}

\begin{equation}\begin{split}\label{int4}
\int |u|^{p-1}\oZ_{\alpha i} Z_{\beta 0} & = O(\mu^{\frac{n-2}{2}}), \quad \int |u|^{p-1}\hZ_{\alpha l} Z_{\beta 0}  = O(\lambda^{\frac{n-2}{2}}),
\end{split}\end{equation}
\begin{equation}\begin{split}\label{int6}
\int |u|^{p-1}\oZ_{\alpha i} \hZ_{\beta l} & = O(\mu^{\frac{n-2}{2}}\lambda^{\frac{n-2}{2}}).
\end{split}\end{equation}
In the formulas above,
$i,j=1,\ldots,k,$ $l,m=1,\ldots, h$, $\alpha,\beta =0,\ldots, n$.

The proof of \eqref{int2} and \eqref{int3} follows like (8.3) in \cite{MW}, and \eqref{int1}, \eqref{int4}  are obtained analogously. Let us prove \eqref{int6}.

The key point here is to notice that if $y\in B(\xi_i,\frac{\oalpha}{k})$ then $|y-\eta_l|\geq C$, with $C$ independent of $i$, $l$ and $k$, and $|y-\xi_i|\geq C$ whenever $y\in B(\eta_l,\frac{\halpha}{h})$, where $C$ is independent of $i$, $l$ and $h$. Consider the case $\alpha=\beta=0$. We split the integral into four parts.
\begin{equation*}\begin{split}
\int |u|^{p-1}\oZ_{0 i} \hZ_{0 l} =& \int_{B(\xi_i,\frac{\oalpha}{k})} |u|^{p-1}\oZ_{0 i} \hZ_{0 j}+\int_{B(\eta_l,\frac{\halpha}{h})} |u|^{p-1}\oZ_{0 i} \hZ_{0 j}\\
& +\int_{\R^n\setminus B(0,2)} |u|^{p-1}\oZ_{0 i} \hZ_{0 j}\\
&+\int_{B(0,2)\setminus (B(\xi_i,\frac{\oalpha}{k})\cup B(\eta_l,\frac{\halpha}{h}))} |u|^{p-1}\oZ_{0 i} \hZ_{0 j}\\
=:&\, i_1+i_2+i_3+i_4,
\end{split}\end{equation*}
Firstly, using the definition of $\oZ_{0i}$ and $\hZ_{0l}$,
\begin{equation*}\begin{split}
i_1\leq & \,C\int_{B(\xi_i,\frac{\oalpha}{k})} |u|^{p-1}\oZ_{0 i}\frac{\lambda^{\frac{n-2}{2}}}{|y-\eta_l|^{n-2}}\leq C\mu^{\frac{n-2}{2}}\lambda^{\frac{n-2}{2}}\int_{B(0,\frac{\oalpha}{\mu k})}U^{p-1}Z_{00}\\
\leq&\, C\mu^{\frac{n-2}{2}}\lambda^{\frac{n-2}{2}},
\end{split}\end{equation*}
where the second inequality follows by the change of variable $x=\xi_i+\mu y$. $i_2$ follows in the same way only by translating to $x=\eta_l+\lambda y$. On the other hand, we have that
$$i_3\leq C\mu^{\frac{n-2}{2}}\lambda^{\frac{n-2}{2}}\int_{\R^n\setminus B(0,2)}\frac{1}{|y|^4|y|^{n-2}|y|^{n-2}}\,dy\leq C\mu^{\frac{n-2}{2}}\lambda^{\frac{n-2}{2}}.$$
To estimate $i_4$ we take into account that, since $\xi_i$ and $\eta_l$ are separated, $|y-\xi_i|^{-(n-2)}$ and $|y-\eta_l|^{-(n-2)}$ cannot be singular at the same time, so the behavior of the integral comes determined by the singularity of only one of them. That is,
\begin{equation*}\begin{split}
i_4\leq &\, C\mu^{\frac{n-2}{2}}\lambda^{\frac{n-2}{2}}\int_{B(0,2)\setminus (B(\xi_i,\frac{\oalpha}{k})\cup B(\eta_l,\frac{\halpha}{h}))} |u|^{p-1}\frac{1}{|y-\xi_i|^{n-2}}\frac{1}{|y-\eta_l|^{n-2}}\,dy\\
\leq &\,C\mu^{\frac{n-2}{2}}\lambda^{\frac{n-2}{2}}\int_{B(0,2)\setminus (B(\xi_i,\frac{\oalpha}{k})\cup B(\eta_l,\frac{\halpha}{h}))}\left(\frac{1}{|y-\xi_i|^{n-2}}+\frac{1}{|y-\eta_l|^{n-2}}\right)\,dy\\
\leq & \, C \mu^{\frac{n-2}{2}}\lambda^{\frac{n-2}{2}}.
\end{split}\end{equation*}
The case $\alpha,\beta\neq 0$ follows analogously just by noticing that
$$\oZ_{\alpha i}\sim \frac{\mu^{\frac{n-2}{2}}}{|y-\xi_i|^{n-1}},\qquad \hZ_{\beta l}\sim \frac{\lambda^{\frac{n-2}{2}}}{|y-\eta_l|^{n-1}},$$
and hence \eqref{int6} is proved.

We now need the following result.

\begin{prop}\label{behaviorPi}
The functions $\pi_\alpha$ can be decomposed as
$$\pi_\alpha(y)=\sum_{j=1}^k\opi_{\alpha j}(y)+\sum_{l=1}^h\hpi_{\alpha l}+\tilde{\pi}_\alpha(y),$$
where
$$\opi_{\alpha j}(y)=\opi_{\alpha 1}(e^{-\frac{2\pi(j-1)}{k}}\overline{y},\hat{y},y'),\qquad \hpi_{\alpha l}(y)=\hpi_{\alpha 1}(\overline{y}, e^{-\frac{2\pi(l-1)}{h}}\hat{y},y').$$
Furthermore, there exists a positive constant $C$ such that
$$\|\tilde{\pi}_\alpha\|_*\leq C(k^{1-\frac{n}{q}}+h^{1-\frac{n}{q}}),\qquad \alpha=0,1,\ldots,n,$$
$$\|\overline{\opi}_{\alpha 1}\|_*\leq C k^{-\frac{n}{q}}, \;\alpha=0,\ldots,n,\;\;\mbox{ where }\;\;\overline{\opi}_{\alpha 1}:=\mu^{\frac{n-2}{2}}\opi_{\alpha 1}(\xi_1+\mu y),$$
$$\|\hat{\hpi}_{\alpha 1}\|_*\leq C h^{-\frac{n}{q}}, \;\alpha=0,\ldots,n,\;\;\mbox{ where }\;\;\hat{\hpi}_{\alpha 1}:=\lambda^{\frac{n-2}{2}}\hpi_{\alpha 1}(\eta_1+\lambda y).$$
\end{prop}
\noindent We omit the proof of this result.

\medskip
Thanks to Proposition \ref{behaviorPi}, we get
\begin{equation}\label{intpi1}
\bigg|\int |u|^{p-1}Z_{\alpha 0}\pi_\beta\bigg|\leq C\|\tilde{\pi}_\beta\|_*,
\end{equation}
\begin{equation}\label{intpi2}
\bigg|\int |u|^{p-1}\oZ_{\alpha i}\pi_\beta\bigg|\leq C\|\opi_{\beta 1}\|_*, \qquad \bigg|\int |u|^{p-1}\hZ_{\alpha l}\pi_\beta\bigg|\leq C\|\hpi_{\beta 1}\|_*.
\end{equation}
Notice next (see (8.5) in \cite{MW} for a proof) that
$$\int U^{p-1}Z_{00}^2=\int U^{p-1}Z_{10}^2>0,$$
and thus we can define
$$t_\beta:=-\frac{1}{\int U^{p-1}Z_{00}^2}\int \varphi^\perp |u|^{p-1}{\bf z}_\beta,$$
which satisfies $|t_\beta|\leq C\|\varphi^\perp\|_*$, with $C$ independent of $k$ and $h$.

Consider \eqref{systemFact1} for $\beta=0$. Thus, by using the definition of $z_0$, \eqref{int1}-\eqref{int6}, \eqref{intpi1}, \eqref{intpi2} and Proposition \ref{behaviorPi} we get
\begin{equation*}\begin{split}
\sum_{\alpha=0}^n&\left[c_{\alpha 0}\int_{\mathbb{R}^n}Z_{\alpha 0}|u|^{p-1}z_0+\sum_{j=1}^k\oc_{\alpha j}\int_{\mathbb{R}^n}\oZ_{\alpha j}|u|^{p-1}z_0+\sum_{l=1}^h\hc_{\alpha l}\int_{\mathbb{R}^n}\hZ_{\alpha l}|u|^{p-1}z_0\right]\\
&=\,c_{00}\int U^{p-1}Z_{00}^2-\sum_{j=1}^k\oc_{0j}\int U^{p-1}Z_{00}^2-\sum_{l=1}^h\hc_{0l}\int U^{p-1}Z_{00}^2\\
&\;\;-\sum_{j=1}^k\oc_{1j}\int U^{p-1}Z_{00}^2-\sum_{l=1}^h\hc_{3l}\int U^{p-1}Z_{00}^2+O(k^{1-\frac{n}{q}}+h^{1-\frac{n}{q}})\mathcal{L}\left[\begin{array}{c} c_{00}\\\ldots\\c_{n0}\end{array}\right]\\
&\;\;+O(k^{-\frac{n}{q}})\overline{\mathcal{L}}\left[\begin{array}{c} \oc_0\\\ldots\\\oc_{n}\end{array}\right]+O(h^{-\frac{n}{q}})\hat{\mathcal{L}}\left[\begin{array}{c} \hc_0\\\ldots\\\hc_{n}\end{array}\right],
\end{split}\end{equation*}
where $\mathcal{L}$, $\overline{\mathcal{L}}$ and $\hat{\mathcal{L}}$ are linear functions with coefficients uniformly bounded in $k$ and $h$. Identity \eqref{t0} follows straightforward from here. \eqref{t1}-\eqref{4nmasalphamenos6} are obtained in the same way.

\section{Proof of \eqref{f2}}\label{fact2}

We proceed as in \cite[Section 9]{MW}. Indeed, we decompose $\varphi_0^\perp$ as
$$\varphi_0^\perp=\sum_{\alpha=0}^n c_{\alpha 0} \varphi_{\alpha 0}^\perp,\qquad\mbox{with}\qquad L(\varphi_{\alpha 0}^\perp)=-L(Z_{\alpha 0}),$$
which is equivalent to
\begin{equation}\label{VarphiZero}
\Delta(\varphi_{\alpha 0}^\perp)+p\gamma U^{p-1}(\varphi_{\alpha 0}^\perp)+a_0(y)\varphi_{\alpha 0}^\perp=-L(Z_{\alpha 0}),
\end{equation}
where $a_0(y):=p\gamma(|u|^{p-1}-U^{p-1})$. Adapting the arguments of \cite{MW} it can be seen that $a\in L^{\frac{n}{2}}(\R^n)$,
\begin{equation}\label{kz00}
|y|^{-n-2}L(Z_{00})(|y|^{-2}y)=-L(Z_{00})(y),
\end{equation}
\begin{equation}\label{kzAlpha0}
 |y|^{-n-2}L(Z_{\alpha 0})(|y|^{-2}y)=L(Z_{\alpha 0})(y),\;\alpha=1,\ldots,n,
 \end{equation}
and
\begin{equation}\label{acero}\|L(Z_{\alpha 0})\|_{L^{\frac{2n}{n+2}}(\R^n)}\leq C(\mu^{\frac{n-1}{n}}+\lambda^{\frac{n-1}{n}}).
\end{equation}
We will solve \eqref{VarphiZero} as a fixed point problem. Let us consider the problem
$$L_0(\varphi)=h-a_0(y)\phi,$$
where $L_0(\varphi):= \Delta \varphi+p\gamma U^{p-1}\varphi$ and $h\in L^{\frac{2n}{n+2}}(\R^n)$ satisfies
$$h(y)=|y|^{-n-2}h(|y|^{-2}y).$$
Let $T$ be the operator that associates to every $\phi$ the solution $\varphi$ to this problem, that is,
$$\varphi=T(h-a_0(y)\phi).$$
Naming $A(\phi):=T(h-a_0(y)\phi)$ we are going to see that this operator is a contraction and that maps the ball
$$B:=\{\phi\in\mathcal{D}^{1,2}(\R^n): \|\phi\|_*\leq C(\mu^{\frac{n-1}{n}}+\lambda^{\frac{n-1}{n}}), \phi(y)=|y|^{-n+2}\phi(|y|^{-2}y)\},$$
into herself. Indeed, assume $\phi\in B$. Thus,
$$a_0(y)\phi(y)=|y|^{-n-2}a_0(|y|^{-2}y)\phi(|y|^{-2}y),$$
and, by \cite[Proposition 9.1]{MW}, we know that
$$\|\varphi\|_*\leq C\|h-a_0(y)\phi\|_{L^{\frac{2n}{n+2}}(\R^n)}\leq C\left(\|h\|_{L^{\frac{2n}{n+2}}(\R^n)}+\|a_0(y)\phi\|_{L^{\frac{2n}{n+2}}(\R^n)}\right),$$
and $\varphi(y)=|y|^{2-n}\varphi(|y|^{-2}y)$.

We study the last term in two different regions. First, in
$$\R^n\setminus (\{\cup_{j=1}^k B(\xi_j,\frac{\oalpha}{k})\}\cup \{\cup_{l=1}^h B(\eta_l, \frac{\halpha}{h})\})$$
we can estimate $a_0$ as
$$|a_0(y)|\leq C U^{p-2}\left[\sum_{j=1}^k\frac{\mu^{\frac{n-2}{2}}}{|y-\xi_j|^{n-2}}+\sum_{l=1}^h\frac{\lambda^{\frac{n-2}{2}}}{|y-\eta_l|^{n-2}}\right],$$
and consequently,
\begin{equation}\begin{split}\label{a0first}
\int_{\R^n\setminus \left(\{\cup_{j=1}^k B(\xi_j,\frac{\oalpha}{k})\}\cup \{\cup_{l=1}^h B(\eta_l, \frac{\halpha}{h})\}\right)}|a_0(y)|^{\frac{2n}{n+2}}\,dy\leq C\left(k^{-(n-1)}+h^{-(n-1)}\right).
\end{split}\end{equation}
Consider now $j\in\{1,\ldots,k\}$ and the ball $B(\xi_j,\frac{\oalpha}{k})$. Here
$$|a_0(y)|\leq C|U_{\mu,\xi_j}(y)|^{p-1},$$
and thus
\begin{equation}\begin{split}\label{a0sec}
\sum_{j=1}^k\int_{B(\xi_j,\frac{\oalpha}{k})} |a_0(y)|^{\frac{2n}{n+2}}\,dy&\leq C\sum_{j=1}^k\int_{B(\xi_j,\frac{\oalpha}{k})}\left[\frac{\mu^{-\frac{n-2}{2}}}{1+|y-\xi_j|^{n-2}}\right]^{(p-1)\frac{2n}{n+2}}\\
&\leq C k \mu^{-2\frac{2n}{n+2}}\mu^n\int_{B(0,\frac{1}{\mu k})}\left[\frac{1}{1+|y|^{n-2}}\right]^{(p-1)\frac{2n}{n+2}}\,dy\\
&\leq C k^{-(n-1)}.
\end{split}\end{equation}
Likewise,
\begin{equation}\label{a0third}
\sum_{l=1}^k\int_{B(\eta_l,\frac{\halpha}{h})} |a_0(y)|^{\frac{2n}{n+2}}\,dy\leq Ch^{-(n-1)}.
\end{equation}
Putting \eqref{a0first}, \eqref{a0sec} and \eqref{a0third} together we conclude that if $\phi\in B$, then
$$\|a_0(y)\phi\|_{L^{\frac{2n}{n+2}}(\R^n)}\leq \|\phi\|_{L^\infty(\R^n)}\|a_0(y)\|_{L^{\frac{2n}{n+2}}(\R^n)}\leq C(\mu^{\frac{n-1}{n}}+\lambda^{\frac{n-1}{n}}).$$
Furthermore
\begin{equation}\begin{split}
\|A(\phi_1)-A(\phi_2)\|_*&\leq C \|a_0(y)(\phi_1-\phi_2)\|_{L^{\frac{2n}{n+2}}(\R^n)}\\
&\leq C \|a_0(y)\|_{L^{\frac{2n}{n+2}}(\R^n)}\|\phi_1-\phi_2\|_*\\
&=o(1)\|\phi_1-\phi_2\|_*,
\end{split}\end{equation}
where $o(1)$ denotes a quantity which goes to zero when $k$, $h$ tend to infinity.
Thus, $A$ defines a contraction mapping whenever
$$\|h\|_{L^{\frac{2n}{n+2}}}\leq C(\mu^{\frac{n-1}{n}}+\lambda^{\frac{n-1}{n}}).$$
Hence, considering $h=L(Z_{\alpha 0})$, by \eqref{acero} we conclude the existence of a solution to \eqref{VarphiZero} satisfying
$$\|\varphi_{\alpha 0}^\perp\|_*\leq  C(\mu^{\frac{n-1}{n}}+\lambda^{\frac{n-1}{n}}).$$

\noindent Consider now $j\in\{1,\ldots,k\}$, $l\in\{1,\ldots, h\}$, and let us write,
$$\ovarphi_j^\perp=\sum_{\alpha=0}^n c_{\alpha j} \ovarphi_{\alpha j}^\perp,\qquad\mbox{with}\qquad L(\ovarphi_{\alpha j}^\perp)=-L(\oZ_{\alpha j}),$$
$$\hvarphi_l^\perp=\sum_{\alpha=0}^n c_{\alpha l} \hvarphi_{\alpha l}^\perp,\qquad\mbox{with}\qquad L(\hvarphi_{\alpha l}^\perp)=-L(\hZ_{\alpha l}).$$
Performing the change of variables
$$\overline{\ovarphi}_j^\perp(y):=\mu^{\frac{n-2}{2}}\ovarphi_{\alpha j}^\perp(\mu y +\xi_j),\qquad \hat{\hvarphi}_l^\perp(y):=\lambda^{\frac{n-2}{2}}\hvarphi_{\alpha l}^\perp(\lambda y +\eta_l),$$
the previous equations turn into
$$\Delta(\overline{\ovarphi}_j^\perp)+p\gamma U^{p-1}(\overline{\ovarphi}_j^\perp)+p\gamma \overline{a}_j(y)\overline{\ovarphi}_j^\perp=\oh_j(y),$$
$$\Delta(\hat{\hvarphi}_l^\perp)+p\gamma U^{p-1}(\hat{\hvarphi}_l^\perp)+p\gamma \hat{a}_l(y)\hat{\hvarphi}_l^\perp=\hh_l(y),$$
where
$$\overline{a}_j(y):=p\gamma[(\mu^{-\frac{n-2}{2}}|u|(\mu y +\xi_j))^{p-1}-U^{p-1}],\;\; \oh_j(y):=-\mu^{\frac{n+2}{2}}L(\oZ_{\alpha j})(\mu y+\xi_j),$$
$$\hat{a}_l(y):=p\gamma[(\lambda^{-\frac{n-2}{2}}|u|(\lambda y +\eta_l))^{p-1}-U^{p-1}],\;\; \hh_l(y):=-\lambda^{\frac{n+2}{2}}L(\hZ_{\alpha l})(\lambda y+\eta_l).$$
Performing an analogous fixed point argument we conclude \eqref{f2}.

\medskip

\section{Final argument}\label{final}

Let $\left[\begin{array}{c} c_0\\ c_1\\\ldots\\c_n\end{array}\right]$ be the solution to \eqref{finalSyst} provided by Proposition \ref{nonso}, and let $t_0,t_1,t_2,t_3,t_4$, $\overline{t}_0,\overline{t}_1,\overline{t}_2$, $\hat{t}_0,\hat{t}_1,\hat{t}_2,$ and $t_\alpha,\overline{\nu}_{\alpha 1},\overline{\nu}_{\alpha 2},\hat{\nu}_{\alpha 1}, \hat{\nu}_{\alpha 2}$, $\alpha=5,\ldots,n$, be the associated parameters. Thus, it follows straightforward the existence of a unique vector of parameters
$$(t_0^*,\ldots,t_4^*,\overline{t}_0^*,\overline{t}_1^*,\overline{t}_2^*,\hat{t}_0^*,\hat{t}_1^*,\hat{t}_2^*,t_5^*,\overline{\nu}_{5 1}^*,\overline{\nu}_{5 2}^*,\hat{\nu}_{5 1}^*, \hat{\nu}_{5 2}^*,\ldots,t_n^*,\overline{\nu}_{n 1}^*,\overline{\nu}_{n 2}^*,\hat{\nu}_{n 1}^*, \hat{\nu}_{n 2}^*)$$
such that $\left[\begin{array}{c} c_0\\ c_1\\\ldots\\c_n\end{array}\right]$ solves the system in Proposition \ref{systC} and, equivalently, \eqref{systemFact1}. Moreover,
\begin{equation*}\begin{split}
\|(t_0^*,t_1^*,t_2^*,t_3^*,t_4^*,\overline{t}_0^*,\overline{t}_1^*,\overline{t}_2^*,\hat{t}_0^*,\hat{t}_1^*,\hat{t}_2^*,&t_5^*,\overline{\nu}_{5 1}^*,\overline{\nu}_{5 2}^*,\hat{\nu}_{5 1}^*, \hat{\nu}_{5 2}^*,\ldots,t_n^*,\overline{\nu}_{n 1}^*,\overline{\nu}_{n 2}^*,\hat{\nu}_{n 1}^*, \hat{\nu}_{n 2}^*)\|\leq C\|\varphi^\perp\|,
\end{split}\end{equation*}
and therefore
$$\|\left[\begin{array}{c} c_0\\ c_1\\\ldots\\c_n\end{array}\right]\|\leq C\|\varphi^\perp\|.$$
This estimate, together with \eqref{f2}, allows us to conclude
$$c_\alpha =0\qquad\forall\,\alpha=0,\ldots,n,$$
and thus $\varphi^\perp\equiv 0$. Replacing this in \eqref{tildeVarphi} the proof of Theorem \ref{nondeg} is complete.

\medskip

\section{Appendix}\label{appe1}

According to their definitions, see \eqref{z0}--\eqref{z5} and \eqref{newzalpha}, it is convenient to rewrite the functions $z_\alpha$ as
\begin{equation*}\begin{split}\label{z0small}
z_0(y)=&Z_{00}(y)-\sum_{j=1}^k\left[\oZ_{0j}(y)+\oZ_{1j}(y)\right]-\sum_{l=1}^h\left[\hZ_{0l}(y)+\hZ_{3l}(y)\right],
\end{split}\end{equation*}
\begin{equation*}\begin{split}
z_1(y)=&Z_{10}(y)-\sum_{j=1}^k\frac{\cos{\otheta_j}\oZ_{1j}(y)-\sin{\otheta_j}\oZ_{2j}(y)}{\sqrt{1-\mu^2}}-\sum_{l=1}^h\hZ_{1l}(y),\\
z_2(y)=&Z_{20}(y)-\sum_{j=1}^k\frac{\sin{\otheta_j}\oZ_{1j}(y)+\cos{\otheta_j}\oZ_{2j}(y)}{\sqrt{1-\mu^2}}-\sum_{l=1}^h\hZ_{2l}(y),
\\
z_3(y)=&Z_{30}(y)-\sum_{j=1}^k\oZ_{3j}(y)-\sum_{l=1}^h\frac{\cos{\htheta_l}\hZ_{3l}(y)-\sin{\htheta_l}\hZ_{4l}(y)}{\sqrt{1-\lambda^2}},\\
z_4(y)=&Z_{40}(y)-\sum_{j=1}^k\oZ_{4j}(y)-\sum_{l=1}^h\frac{\sin{\htheta_l}\hZ_{3l}(y)+\cos{\htheta_l}\hZ_{4l}(y)}{\sqrt{1-\lambda^2}},\\
z_\alpha(y)=&Z_{\alpha 0}(y)-\sum_{j=1}^k\oZ_{\alpha j}(y)-\sum_{l=1}^h\hZ_{\alpha l}(y),\quad \alpha=5,\ldots,n \\
z_{n+1}(y)&=-\sum_{j=1}^k\oZ_{2j}(y),\quad
z_{n+2}(y)=-\sum_{l=1}^h\hZ_{4l}(y),\\
\end{split}\end{equation*}
\begin{equation*}\begin{split}
z_{n+7}(y)&=-\sqrt{1-\mu^2}\sum_{j=1}^k\cos{\otheta_j}\oZ_{3j}(y)+\sqrt{1-\lambda^2}\sum_{l=1}^h\cos{\htheta_l}\hZ_{1l}(y),\\
z_{n+8}(y)&=-\sqrt{1-\mu^2}\sum_{j=1}^k\cos{\otheta_j}\oZ_{4j}(y)+\sqrt{1-\lambda^2}\sum_{l=1}^h\sin{\htheta_l}\hZ_{1l}(y),\\
z_{n+\alpha+4}(y)&=-\sqrt{1-\mu^2}\sum_{j=1}^k\cos{\otheta_j}\oZ_{\alpha j}(y), \quad \alpha=5,\ldots,n,\\
z_{2n+5}(y)&=-\sqrt{1-\mu^2}\sum_{j=1}^k\sin{\otheta_j}\oZ_{3j}+\sqrt{1-\lambda^2}\sum_{l=1}^h\cos{\htheta_l}\hZ_{2l}\\
z_{2n+6}(y)&=-\sqrt{1-\mu^2}\sum_{j=1}^k\sin{\otheta_j}\oZ_{4j}+\sqrt{1-\lambda^2}\sum_{l=1}^h\sin{\htheta_l}\hZ_{2l}\\
\end{split}\end{equation*}
and, for $\alpha=5,\ldots,n$,
\begin{equation*}\begin{split}
z_{2n+\alpha+2}(y)=-\sqrt{1-\mu^2}&\sum_{j=1}^k\sin{\otheta_j}\oZ_{\alpha j},\quad
z_{3n+\alpha-2}(y)=-\sqrt{1-\lambda^2}\sum_{l=1}^h\sin{\htheta_l}\hZ_{3l},\\
z_{4n+\alpha-6}(y)&=-\sqrt{1-\lambda^2}\sum_{l=1}^h\cos{\htheta_l}\hZ_{4l}.
\end{split}\end{equation*}
The proof of the above identities follows from straightforward computations, and the symmetry properties for
$U (y) + \psi (y)$, for $ U_{\mu , \xi_j} (y) + \overline \phi_j (y) $ and $U_{\lambda , \eta_l} (y) + \hat \phi_l (y)$ respectively.

A less straightforward computation gives that,  $\alpha = n+3, n+4, n+5, n+6$, we have
\begin{equation}\begin{split} \label{soloquesto}
{\bf z}_{n+3}(y)&=z_1 - 2  \sqrt{1-\mu^2}\sum_{j=1}^k \cos\otheta_j\left[   \oZ_{0j}(y)  + \oZ_{1j} \right] ,\\
{\bf z}_{n+4}(y)&=z_2 - 2  \sqrt{1-\mu^2}\sum_{j=1}^k \sin\otheta_j \left[   \oZ_{0j}(y)  + \oZ_{1j} \right], \\
{\bf z}_{n+5}(y)&=z_1 - 2  \sqrt{1-\lambda^2} \sum_{l=1}^h   \cos\htheta_l  \left[  \hZ_{0l}(y)  + \hZ_{3l} \right], \\
{\bf z}_{n+6}(y)&=z_1 - 2   \sqrt{1-\lambda^2} \sum_{j=1}^h  \sin\htheta_l\left[  \hZ_{0l}(y)  +  \hZ_{4l} \right] .\\
\end{split}\end{equation}
We shall prove the validity of the first identity in \eqref{soloquesto}.
The proofs of the validity of the the other expressions in \eqref{soloquesto} are similar. We write
$$
{\bf z}_{n+3} = z_1 + T (u), \qquad T (u) := (|y|^2 -1 ) {\partial u \over \partial y_1} - 2y_1 ({n-2 \over 2} u (y) + \nabla u (y) \cdot y ).
$$
Thus \eqref{soloquesto} follows from \eqref{newzalpha} and from
\begin{equation}
\label{newsoloquesto}
T(u) = -2\sum_{j=1}^k \xi_{j1} \left[ \oZ_{0j} +\nabla (U_{\mu , \xi_j} + \overline \phi_j  ) (y) \cdot \xi_j \right]
.
\end{equation}
From the explicit expression of $u$ in \eqref{finalform}, we get
$$
T(u) = T(U +\psi ) - \sum_{j=1}^k T (U_{\mu , \xi_j} + \overline \phi_j  ) - \sum_{l=1}^h T(U_{\lambda , \eta_l} +\hat \phi_l ) .
$$
We shall first show that $ T(U +\psi ) (y) \equiv 0$, and $T(U_{\lambda , \eta_l} +\hat \phi_l )  (y) \equiv 0$ for any $l=1, \ldots , h$.

\medskip
Observe that, if $v$ is any smooth function and if we define $h(z) := {\partial  \over \partial z_1} \left( |z|^{2-n} v({z \over |z|^2} ) \right)$, then we have
\begin{equation*}\begin{split}
h(z) &= -{2 z_1 \over |z|^n} \left[ {n-2 \over 2} v   \left( {z \over |z|^2} \right) + \nabla v \left( {z \over |z|^2} \right) \cdot \left( {z \over |z|^2} \right) \right] \\
&+ {1\over |z|^n} {\partial v \over \partial z_1} \left( {z \over |z|^2} \right),
\end{split} \end{equation*}
and $g(y):= {1\over |y|^{n-2} } h ({y \over |y|^2} )$ takes the form
$$
g(y) = -2 y_1 \left[ {n-2 \over 2} v(y) + \nabla v (y) \cdot y \right] + |y|^2 {\partial v \over \partial y_1} (y).
$$
With this is mind, one gets that if $v$ is Kelvin invariant
$
v(y ) = |y|^{n-2} v \left( {y \over |y|^2} \right), $ then
\begin{equation}\label{gino}{1\over |y|^{n-2}} {\partial v \over \partial y_1} \left( {y\over |y|^2} \right) =
-2 y_1 \left[ {n-2 \over 2} v(y) + \nabla v (y) \cdot y \right] + |y|^2 {\partial v \over \partial y_1} (y).
\end{equation}
On the other hand, if $v$ is Kelvin invariant (with respect to the origin) and even in $y_1$, then also
the function ${\partial v \over \partial y_1}$ is Kelvin invariant, that is $
{\partial v \over \partial y_1}(y ) = |y|^{n-2} {\partial v \over \partial y_1} \left( {y \over |y|^2} \right). $
By \eqref{gino}, we get that any function $v$ which is invariant under Kelvin transform (with respect to the origin) and even in the $y_1$ direction,
one that $T(v) (y) \equiv 0$. Since the functions $ (U +\psi ) (y) $, and $(U_{\lambda , \eta_l} +\hat \phi_l )  (y)$ for any $l=1, \ldots , h$
are invariant under Kelvin transform and even in $y_1$, we get the proof of our claim.

\medskip
Let us fix $j \in \{1 , \ldots , k\}$. We write, for $v(y) = ( U_{\mu , \xi_j} + \overline \phi_j)  (y)$,
\begin{equation*}\begin{split}
T (U_{\mu , \xi_j} + \overline \phi_j  ) (y) &=\underbrace{(|y|^2 -1) {\partial v \over \partial y_1 } (y) - 2 (y-\xi_j )_1 [ {n-2 \over 2} v (y) + \nabla v (y) \cdot y] }_{=:T_j (v)} \\
&-2 (\xi_j)_1\left[ \oZ_{0j} (y) + \nabla v(y) \cdot \xi_j \right] .
\end{split}\end{equation*}
We claim that $T_j (U_{\mu , \xi_j} + \overline \phi_j  ) (y) \equiv 0$. To prove this fact, we recall that
$$
v(y):= (U_{\mu , \xi_j} + \overline \phi_j  ) (y) = \mu^{-{n-2 \over 2} } (U + \overline{\ophi}_1 ) \left({y-\xi_j \over \mu} \right),
$$
see Section \ref{sec3}. Also, $\mu$ and $|\xi|$ are related so that  $U_{\mu , \xi_j} + \overline \phi_j $ is invariant under Kelvin transform. Thus, from \eqref{gino}, we get
$$
T_j (v) (y) = {1\over |y|^{n-2}} {\partial v \over \partial y_1} \left({y \over |y|^2} \right)- {\partial v \over \partial y_1} (y) + 2 (\xi_j)_1 \left[ {n-2 \over 2} v(y) + \nabla v (y) \cdot y\right].
$$
We note that, in this case, $U_{\mu , \xi_j} + \overline \phi_j $ is not even in the $y_1$ variable, so that one gets
\begin{equation*}\begin{split}
{1\over |y|^{n-2}} {\partial v \over \partial y_1} \left({y \over |y|^2} \right)&= {\partial v \over \partial y_1} (y) - (\xi_j)_1 \left[ (n-2) v(y) +2 \nabla v (y) \cdot y\right].
\end{split}
\end{equation*}
This concludes the proof of \eqref{newsoloquesto}.

\end{document}